\documentclass[11pt]{amsart}

\usepackage{fullpage} 
\usepackage[foot]{amsaddr} 
\usepackage{graphicx} 
\usepackage{amsmath, amssymb, amsfonts, amsthm, mathtools} 
\usepackage{enumerate} 

\usepackage{color} 
\usepackage{url} 
\usepackage[plainpages = false, colorlinks, citecolor = blue, filecolor = black, linkcolor = red, urlcolor = cyan]{hyperref} 
\usepackage{appendix}
\usepackage{multicol}
\usepackage{comment}

\theoremstyle{definition} 
\newtheorem{definition}{Definition}[section] 
\theoremstyle{plain} 
\newtheorem{theorem}[definition]{Theorem}
\newtheorem{proposition}[definition]{Proposition}

\theoremstyle{remark}

\allowdisplaybreaks 

\usepackage[backend = bibtex, giveninits = true, style = alphabetic, natbib = true, url = false, sorting = nyt, doi = true, abbreviate = true, backref = false]{biblatex}
\renewbibmacro{in:}{\ifentrytype{article}{}{\printtext{\bibstring{in}\intitlepunct}}}

\newcommand{\R}{\mathbb{R}} 
 
\newcommand{\Z}{\mathbb{Z}} 
\newcommand{\N}{\mathbb{N}}

\newcommand{\vareps}{\varepsilon} 
\newcommand{\del}{\partial} 
\newcommand{\n}{\mathbf{n}} 
\newcommand{\tn}{\textnormal} 
\renewcommand{\vec}[1]{\hat{\boldsymbol{#1}}} 

\begin{document} 
\title{Steklov Eigenvalues of Nearly Circular Area-Normalized Domains} 
\author{Lucas Alland}
\address{Department of Mathematics \& Statistics, Swarthmore College} 
\email{lalland1@swarthmore.edu} 

\author{Robert Viator} 
\address{Department of Mathematics, Denison University, Granville, OH} 
\email{viatorr@denison.edu} 


\date{\today} 

\begin{abstract} 
We consider Steklov eigenvalues of nearly circular domains in $\R^{2}$ of fixed unitary area. In \cite{viator2018}, the authors treated such domains as perturbations of the disk, and they computed the first-order term of the asymptotic expansions of the Steklov eigenvalues for reflection-symmetric perturbations; here, we expand these first-order results beyond reflection-symmetry.  We also recover the second-order asymptotic expansions, which enable us to prove that no Steklov eigenvalue beyond the first positive one is locally shape-optimized by the disk.
\end{abstract} 

\maketitle


\section{Introduction} 
Let $\Omega\subset\R^{2}$ be a bounded domain. The Steklov eigenvalue problem for $(\lambda, u)$ on $\Omega$ is given by 
\begin{subequations} 
\begin{alignat}{2} 
\label{eq:Steklov1} \Delta u & = 0 && \ \ \tn{ in } \Omega, \\ 
\label{eq:Steklov2} \del_\n u & = \lambda u && \ \ \tn{ on } \del\Omega, 
\end{alignat} 
\end{subequations} 
where $\Delta$ is the Laplacian acting on $H^1(\Omega)$, $\del_\n u = \nabla u\cdot \n$ is the unit outward normal derivative on the boundary $\del\Omega$, and $\lambda$ is the spectral parameter. It is well-known that the Steklov spectrum is discrete as long as the trace operator $T\colon H^1(\Omega)\to L^2(\del\Omega)$ is compact \cite{girouard2017}. Moreover, the eigenvalues are real and we enumerate them, counting multiplicity, in increasing order 
\begin{equation*} 
0 < \lambda_0(\Omega) < \lambda_1(\Omega)\le \lambda_2(\Omega)\le \dots \nearrow\infty. 
\end{equation*} 
The Steklov eigenvalue problem has received considerable attention in the literature; see the survey papers \cite{girouard2017, colbois2023} and references therein. 

The Steklov eigenvalue problem was first introduced by Vladimir Steklov in \cite{stekloff1902} to describe the stationary heat distribution in a body $\Omega$ whose heat flux through the boundary is proportional to the temperature. For planar domains, the Steklov eigenvalues are the squares of the natural frequencies of a vibrating free membrane with all its mass concentrated along the boundary \cite[p. 95]{Bandle}. Steklov eigenvalues also have applications in optimal material design for both electromagnetism and torsional rigidity \cite{lipton1998a, lipton1998b}. Recently, Cakoni et al. \cite{cakoni2016} used Steklov eigenvalues in nondestructive testing, where they established a crucial relationship between small changes in the (possibly complex valued) refractive index of a scattering object and the corresponding change in the eigenvalue of a modified Steklov problem. For this problem, numerical results in \cite{cakoni2016} revealed that a localised defect of the refractive index in a disc perturbs only a small number of modified Steklov eigenvalues. 

The problem of Steklov shape optimization has seen many new developments in the past decade.  In \cite{bogosel2017}, the authors discovered that the global optimizing domain for the $k$th Steklov eigenvalue (among shapes of fixed area) is a union of at most $k$ disjoint Jordan domains.  The authors of \cite{kaoosting2017} explore the shape-optimization of the Steklov problem among simply connected domains in $\mathbb{R}^2$ numerically, developing a conjecture that symmetric ``star-shaped" domains are the optimizers among this class of shapes in the plane.  In \cite{fraser2019}, the authors explore the fixed surface volume Steklov optimization problem for contractible domains in $\mathbb{R}^n$ for $n\geq 3$, proving that Weinstock's inequality \cite{weinstock1954} for disks in $\mathbb{R}^2$ does \emph{not} hold in dimension $n \geq 3$.  In \cite{fraser2016}, the same authors find a remarkable connection between optimal metrics for the first non-zero Steklov eigenvalue of surfaces with boundary and free-boundary minimal surfaces: this relationship is explored in great numerical detail in \cite{oudetosting2021} and \cite{oudetosting2023}.  In \cite{viator2018}, \cite{viator2022}, and \cite{tanviator2024}, the authors further explore the area- and volume-constrained problems for domains in $\mathbb{R}^d$ for $d\geq 3$ using perturbative methods, showing for all $d\geq 3$, the ball is never a local maximizer for an infinite list of Steklov eigenvalues; this is done by expanding the Steklov eigenvalues in a power series expansion in a shape-deformation parameter $\varepsilon >0$ and calculating the order-$\varepsilon$ coefficients for the power series expansions. 

\subsection{Main results} 
We now further explore the method outlined in \cite{viator2018}, \cite{viator2022}, and \cite{tanviator2024} in $\mathbb{R}^2$; in particular, we seek out the order-$\varepsilon^2$ terms in the power series expansions of the Steklov eigenvalues of nearly-circular domains.  The primary deviations of this paper from \cite{viator2018} are twofold.  First, we consider nearly-circular domains of general shape, ignoring the constraint of reflection-symmetry imposed in \cite{viator2018}.  Second, we will consider only domains $\Omega_\varepsilon$ of fixed area by normalizing the domains to have unit area \emph{before} we begin the asymptotic analysis.

In particular, we consider domains of the form
\begin{align}
    \label{eq:Domainintro}
    \Omega_\vareps = \{(r,\theta) \mid 0 \leq \theta <2\pi, 0 \leq r \leq \frac{1+\vareps\rho(\theta)}{\sqrt{v(\vareps)}}\}\text{ , } 
\end{align}
where $v(\varepsilon)$ is an area-normalization factor, and $\rho \in C^1(S^1)$.  We expand the Steklov eigenvalues of the domain $\Omega_\vareps$ as a power series in $\vareps$, and we carefully extract the order-$\vareps$ and order-$\vareps^2$ terms.  The order-$\vareps^2$ terms give insight into the case where the order-$\vareps$ terms are zero, revealing the following theorem about all non-zero Steklov eigenvalues after the first:

\begin{theorem} \label{thm:noballs} 
Let $\sigma_k(\Omega)$ be the $k$th non-zero Steklov eigenvalue of the area-normalized domain $\Omega$.  If $k\geq 2$, then $\sigma_k$ is never \emph{locally} maximal when $\Omega$ is a disk. 
\end{theorem} 

The proof of Theorem \ref{thm:noballs} relies on explicit calculation of the second-order terms in the perturbation expansion of the Steklov eigenvalues in terms of the perturbation parameter $\vareps$. 
\subsection{Motivating Numerics}
In order to see Theorem \ref{thm:noballs} in action, we perform a numerical experiment with Steklov eigenvalues of area-normalized nearly-circular domains.  In Figures \ref{num:stekpert2} and \ref{num:stekpert8}, we plot the Steklov eigenvalues for domains $\Omega_\varepsilon$ of the form \eqref{eq:Domainintro} as a function of $\varepsilon$.  The unperturbed eigenvalues are the non-negative integers scaled by $\sqrt{\pi}$, each with multiplicity 2.  We see that by perturbing the $(n+\lceil \frac{n}{2}\rceil)$th Fourier mode of $\rho$, we can make both branches of the $n$th eigenvalue increasing functions of $\varepsilon$. For $n=2$ we perturb the 3rd Fourier mode and for for $n=8$ we perturb the 12th Fourier mode. We can see in figures \ref{fig:k=3} and \ref{fig:k=12} that both branches of the 2nd and 8th eigenvalues increase in $\vareps$ away from the disk.
\begin{figure}
\label{num:stekpert2}
    \centering
    \includegraphics[width=0.5\linewidth]{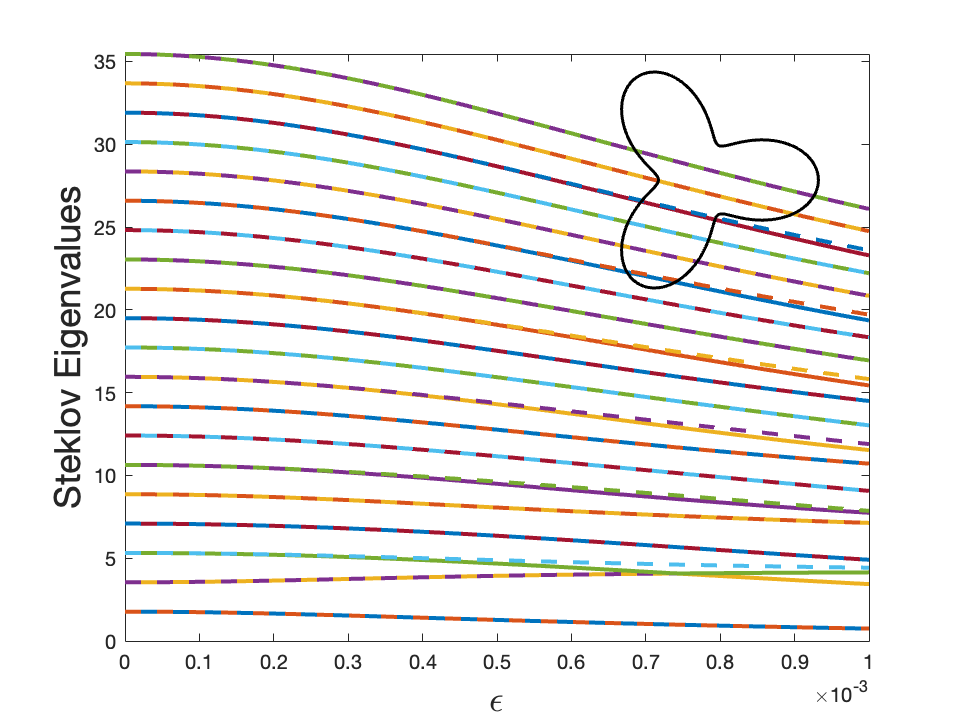}
    \caption{Eigenvalues on $\Omega_\vareps$ for $\rho(\theta) = 500\cos(3\theta)$}
    \label{fig:k=3}
\end{figure}
\begin{figure}
\label{num:stekpert8}
    \centering
    \includegraphics[width=0.5\linewidth]{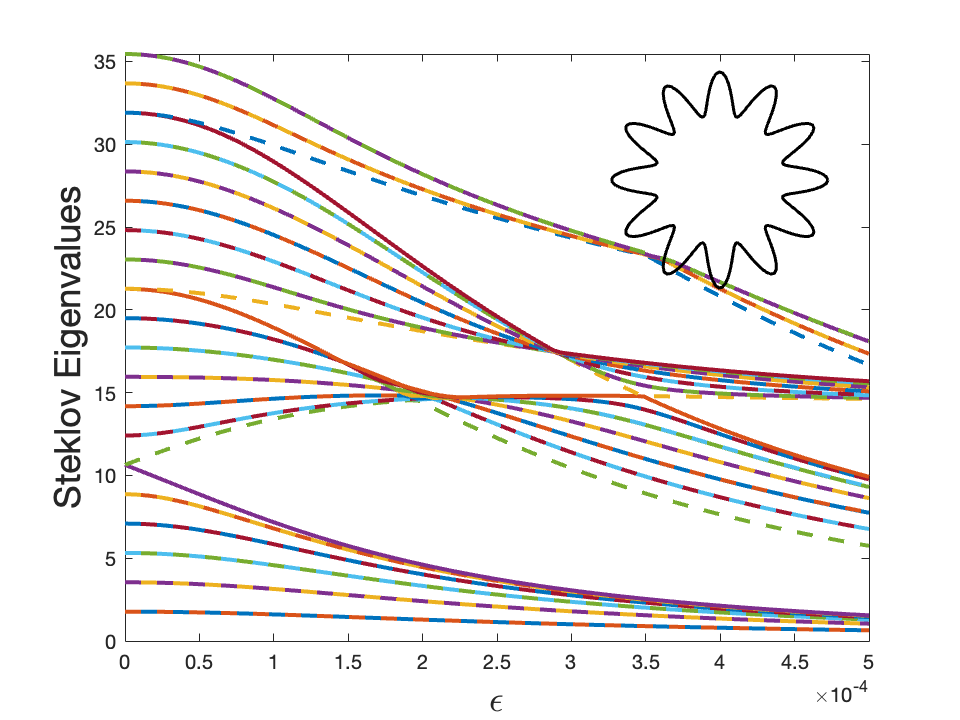}
    \caption{Eigenvalues on $\Omega_\vareps$ for $\rho(\theta) = 500\cos(12\theta)$}
    \label{fig:k=12}
\end{figure}

\subsection{Outline} 
This paper is organized as follows. We begin by reviewing polar coordinates  and compute the first-order and second-order asymptotic expansions for geometrical quantities related to $\Omega_\vareps$ in Section \ref{sec:prelim}, including the area-normalization constraint.  We also, in the same section, briefly discuss the analytic dependence of the Steklov eigenvalues on the parameter $\varepsilon$ for suitable functions $\rho$. In Section \ref{solnasymptotics}, we use the perturbative ansatz presented in Section \ref{sec:prelim} to derive the full asymptotic expansion for the Steklov eigenvalue problem on $\Omega_\varepsilon$.  In Section \ref{sec:asymptotic-first}, we derive the first-order asymptotic expansion for the Steklov eigenvalues of $\Omega_\vareps$. 
In Section \ref{sec:asymptotic-sec}, we obtain the second-order asymptotic expansion for Steklov eigenvalues of $\Omega_\vareps$, which we leverage to prove Theorem \ref{thm:noballs}. 

\section{Preliminaries} \label{sec:prelim} 
In this section, we first review polar coordinates, including the polar basis vectors $\hat{r}$ and $\hat{\theta}$. We then area-normalize the nearly-circular domain $\omega_\vareps$ to obtain a nearly circular domain $\Omega_\vareps$ of unit area, and we calculate the first- and second-order asymptotic expansions for the unit outward normal vector $\vec \nu_\vareps$ to $\partial \Omega_\vareps$ and powers of $R_\vareps$, the radius of $\partial \Omega_\vareps$.

\subsection{Polar coordinates in $\mathbb{R}^2$} \label{sec:Coordinates} 
We will work in the following moving basis for $\mathbb{R}^2$:
\begin{align}
\label{eq:movingBasis}
    \hat r = \begin{bmatrix}
        \cos \theta\\
        \sin \theta
    \end{bmatrix}\ \ \hat \theta = \begin{bmatrix}
        -\sin \theta\\
        \cos \theta
    \end{bmatrix}\text{ . } 
\end{align}
We record useful properties of this basis:
\begin{enumerate}
    \item $\frac{d}{d\theta} \hat r = \hat \theta$.
    \item $||\hat r|| = ||\hat \theta|| = 1$ and $\hat r \cdot \hat \theta = 0$.
    \item $(a\hat r + b\hat \theta) \cdot (c\hat r + d\hat \theta) = ac + bd$.
    \item For a differentiable $f: \R^2 \rightarrow \R$:
    \begin{align}
        \label{eq:polarGradient}
        \nabla f = \frac{\partial f}{\partial r}\hat r + \frac{1}{r}\frac{\partial f}{\partial \theta}\hat \theta\text{ . } 
    \end{align}
\end{enumerate}

We also recall the Steklov eigenvalue problem on the unit disk $$=D = \{ (r, \theta) : 0 \leq r \leq 1, 0 \leq \theta < 2\pi \}.$$  The problem \eqref{eq:Steklov1}, \eqref{eq:Steklov2} becomes
\begin{align}
    \frac{1}{r} \partial_r \left ( r \partial_r u \right ) + \frac{1}{r^2}\partial^2_{\theta} u & = 0  \text{ in } D^{\circ}\\
    \partial_r u |_{r=1} & = \lambda u    
\end{align}
The eigenvalues for this problem are the non-negative integers.  The eigenvalue $\lambda = 0$ is multiplicity one, with eigenspace given by constant functions defined on $D$.  For $\lambda = n >0$, the eigenvalue has multiplicity two: the eigenspace $E_n$ is given by $E_n = \text{span} \{r^n \operatorname{cos}(n\theta), r^n\operatorname{sin}(n\theta) \}$.

In the case of dilated disk $D_a= a D$ with $a>0$, the eigenspaces remain the same, but the corresponding eigenvalues $\tilde{\lambda}_n$ for the eigenspace $E_n$ satisfy the homothety property:  $\tilde{\lambda}_n = \frac{1}{a} \lambda_n = \frac{1}{a} n$ .

\subsection{Area-Normalized Domains}
We define the \textit{general nearly circular domain} $\omega_\vareps \subset \R^2$ in polar coordinated by
\begin{align}
    \label{eq:genDomain}
    \omega_\vareps = \{(r,\theta) \mid 0 \leq \theta < 2\pi, 0 \leq r \leq 1 +\vareps\rho(\theta)\}
\end{align}
where $\rho \in C^2(S^1)$ where $S^1$ denotes the unit circle in $\R^2$. $\rho$ may be expanded in the Fourier basis by
\begin{align}
    \label{eq:rhoFourier}
    \rho(\theta) = \sum_{j=0}^\infty a_j\sin(j\theta) + b_j\cos(j\theta)
\end{align}
We recall Cauchy's product for infinite sums,
\begin{align}
    \label{cauchyproduct}
    (\sum_{j=0}^\infty a_j)(\sum_{j=0}^\infty b_j) = \sum_{j=0}^\infty\sum_{i=0}^ja_ib_{j-i}\text{ , } 
\end{align}
to calculate a Fourier basis expansion for $\rho^2$ in terms of the expansion of $\rho$:
\begin{align}
    \label{eq:rhoSquaredFourier}
    \rho^2(\theta) =& \frac{1}{2}\sum_{j=0}^\infty\sum_{i=0}^j\Big[-a_ia_{j-i}(\cos(j\theta)-\cos((j-2i)\theta)) + a_ib_{j-i}(\sin(j\theta)-\sin((j-2i)\theta)) \\
    &+ b_ia_{j-i}(\sin(j\theta)+\sin((j-2i)\theta))+b_ib_{j-i}(\cos(j\theta)+\cos((j-2i)\theta))\Big]\nonumber\text{ . } 
\end{align}
Now we define the function $v(\vareps) = |\omega_\vareps|$ and explicitly calculate it:
\begin{align}
    v(\varepsilon) =& \int_0^{2\pi}\int_0^{1+\varepsilon\rho(\theta)} r dr d\theta\\
    =& \frac{1}{2}\int_0^{2\pi} 1 + 2\varepsilon\rho(\theta) +(\varepsilon\rho(\theta))^2 d\theta\nonumber\\
    =&\frac{1}{2}\Big[2\pi +4\pi\varepsilon b_0 + \frac{1}{2}\varepsilon^2\sum_{j=0}^\infty\sum_{i=0}^j\int_0^{2\pi}[-a_ia_{j-i}(\cos(j\theta)-\cos((j-2i)\theta))\nonumber\\
    &+ a_ib_{j-i}(\sin(j\theta)-\sin((j-2i)\theta)) \nonumber+ b_ia_{j-i}(\sin(j\theta)+\sin((j-2i)\theta))\nonumber\\
    &+b_ib_{j-i}(\cos(j\theta)+\cos((j-2i)\theta))]d\theta\Big] \nonumber\\
    =& \pi +2\pi\varepsilon b_0+\frac{\pi}{2}\varepsilon^2\sum_{j=0}^\infty\sum_{i=0}^j\Big[\delta_{j,0}(b_ib_{j-i}-a_ia_{j-i}) + \delta_{j-2i, 0}(a_ia_{j-i} + b_ib_{j-i})\Big]\nonumber\\
    =&\pi +2\pi\varepsilon b_0+\frac{\pi}{2}\varepsilon^2\Big[b_0^2-a_0^2 +\sum_{\text{even } j}^\infty a_{\frac{j}{2}}^2+b_{\frac{j}{2}}^2\Big] \nonumber\\
    =&\pi +2\pi\varepsilon b_0+\frac{\pi}{2}\varepsilon^2\Big[b_0^2 - a_0^2 +\sum_{j=0}^\infty a_{j}^2+b_{j}^2\Big]\nonumber\\
    \label{eq:vfunc}
    =&\pi +2\pi\varepsilon b_0+\frac{\pi}{2}\varepsilon^2\Big[2b_0^2 +\sum_{j=1}^\infty a_{j}^2+b_{j}^2\Big]\text{ . } 
\end{align}
So we may now define the \textit{area-normalized nearly circular domain}, $\Omega_\vareps$ by
\begin{align}
    \label{eq:Domain}
    \Omega_\vareps = \{(r,\theta) \mid 0 \leq \theta <2\pi, 0 \leq r \leq \frac{1+\vareps\rho(\theta)}{\sqrt{v(\vareps)}}\}\text{ , } 
\end{align}
and we see that it is indeed area-normalized:
\begin{align}
    |\Omega_\varepsilon| &= \int_0^{2\pi}\int_0^{\frac{1+\varepsilon\rho(\theta)}{\sqrt{v(\varepsilon)}}} rdrd\theta=\int_0^{2\pi}\frac{1}{2}(\frac{1+\varepsilon\rho(\theta)}{\sqrt{v(\varepsilon)}})^2 d\theta\\
    &= \frac{1}{v(\varepsilon)}\int_0^{2\pi}\frac{1}{2}(1+\varepsilon\rho(\theta))^2 d\theta=\frac{v(\varepsilon)}{v(\varepsilon)}= 1\text{ . }\nonumber 
\end{align}

Following the approach of \cite{viator2018}, \cite{viator2022}, we make the following perturbation ansatz for the solution $u_\varepsilon = u_\varepsilon(r,\theta )$ of \eqref{eq:Steklov1}, \eqref{eq:Steklov2}:

\begin{align}
\label{eq:ansatz1}
    u_\varepsilon= \sum_{k=0}^\infty r^k(A_k(\varepsilon)\sin(k\theta)+B_k(\varepsilon)\cos(k\theta)),
\end{align}
where the coefficients $A_k(\varepsilon), B_k(\varepsilon)$ are expanded in a power series in $\varepsilon$:
\begin{align}
\label{eq:ansatz2}
    A_k(\varepsilon) &= \delta_{k,n}\gamma_n + \mu_k\varepsilon + \eta_k\varepsilon^2 + O(\varepsilon^3)\\
\label{eq:ansatz3}
    B_k(\varepsilon) &= \delta_{k,n}\alpha_n + \beta_k\varepsilon + \sigma_k\varepsilon^2+ O(\varepsilon^3) \text{.}
\end{align} Here $n$ enumerates the eigenpairs, not counting multiplicity. For simplicity, we will always assume $n > 0$ since the $0$th eigenpair has trivial eigenvalue, always identically 0.

\subsection{Asymptotic expansions for geometric quantities} 
We calculate a first- and second-order asympotic expansion for the outward unit normal $\vec \nu_{\vareps}$ to $\partial \Omega_\vareps$.

We see we may parameterize $\partial \Omega_\varepsilon$ in $\theta$ by $\frac{1}{\sqrt{v(\varepsilon)}}(1+\varepsilon\rho(\theta))\hat{r}$. Thus we see the tangent vector to the boundary is 
\begin{align}
    \Vec{\tau}_\vareps = \frac{d}{d\theta} \frac{1}{\sqrt{v(\varepsilon)}}(1+\varepsilon\rho(\theta))\hat{r}= \frac{\varepsilon\rho^\prime(\theta)}{\sqrt{v(\varepsilon)}}\hat{r} + \frac{1+\varepsilon\rho(\theta)}{\sqrt{v(\varepsilon)}}\hat{\theta}\text{ . } 
\end{align}
Clearly $\tilde{\nu}_\vareps$ is normal to $\partial\Omega_\varepsilon$ when $\tilde{\nu}_\vareps\cdot\vec{\tau}_\vareps = 0$ so we find the normal vector 
\begin{align}
    \tilde{\nu}_\vareps = (\varepsilon\rho(\theta) +1)\hat{r} -\varepsilon\rho^\prime(\theta)\hat{\theta}\text{ . } 
\end{align}
Furthermore, we see that $\tilde{\nu}_0 = \hat{r}$ is the outward unit normal to $\Omega_0$. So we just need $\vec{\nu}_\vareps = \frac{\tilde{\nu}_\vareps}{||\tilde\nu_\vareps||}$. We find that the outward-facing unit normal is given by
\begin{align}
    \vec{\nu}_\vareps=\frac{(\varepsilon\rho(\theta)+1)\hat{r} -\varepsilon\rho^{\prime}(\theta)\hat{\theta}}{\sqrt{\varepsilon^{2}\rho^{2}(\theta)+2\varepsilon\rho(\theta)+1+\varepsilon^{2}{\rho^{\prime}}^{2}(\theta)}}\text{ . } 
\end{align}
By Taylor's Theorem we can find first- and second-order approximations of $\vec{{\nu}}_\vareps$ by evaluating $\frac{d}{d\varepsilon}\vec{{\nu}}_\varepsilon \mid_{\varepsilon = 0}$ and$\frac{d^2}{d\varepsilon^2}\vec{{\nu}}_\varepsilon \mid_{\varepsilon = 0}$  

By the product rule, $\frac{d\vec{\nu}_\varepsilon}{d\varepsilon}=(\frac{d}{d\varepsilon}\frac{1}{||\tilde{\nu}_\varepsilon||})\tilde{\nu}_\varepsilon + \frac{1}{||\tilde{\nu}_\varepsilon||}(\frac{d}{d\varepsilon}\tilde{\nu}_\varepsilon)$. We will calculate these derivatives separately:
\begin{align}
    \frac{d}{d\varepsilon}\frac{1}{||\tilde{\nu}_\varepsilon||} &= \frac{d}{d \varepsilon} (||\tilde\nu_\varepsilon||^2)^{-\frac{1}{2}}\\
    &=\frac{-1}{(2||\tilde\nu_\varepsilon||^2)^{\frac{3}{2}}}\frac{d}{d \varepsilon} (||\tilde\nu_\varepsilon||^2)\nonumber\\
    &=\frac{-1}{||\tilde\nu_\varepsilon||^3}(\frac{1}{2}\frac{d}{d \varepsilon} (||\tilde\nu_\varepsilon||^2))\nonumber\\
    &=\frac{-1}{||\tilde\nu_\varepsilon||^3}(\frac{1}{2}\frac{d}{d\varepsilon}(\varepsilon^{2}\rho^{2}(\theta)+2\varepsilon\rho(\theta)+1+\varepsilon^{2}{\rho^{\prime}}^{2}(\theta)))\nonumber\\
    &=\frac{-1}{||\tilde\nu_\varepsilon||^{3}}(\rho(\theta)+\rho^{2}(\theta)\varepsilon+\rho^{\prime^2}(\theta)\varepsilon)\\
\frac{d}{d\varepsilon}\tilde\nu_\varepsilon&=\frac{d}{d\varepsilon}(\varepsilon\rho(\theta)+1)\hat{r} -\varepsilon\rho^{\prime}(\theta)\hat{\theta}=\rho(\theta)\hat{r}-\rho^{\prime}(\theta)\hat{\theta}\text{ . } 
\end{align}
Applying the product rule now yields
\begin{align}
    \frac{d\vec{\nu}_\vareps}{d\varepsilon}=& \frac{\rho(\theta)\hat{r}-\rho^{\prime}(\theta)\hat{\theta}}{||\tilde\nu_\varepsilon||}-\frac{(\rho(\theta)+\rho^{2}(\theta)\varepsilon+\rho^{\prime^2}(\theta)\varepsilon)((\varepsilon\rho(\theta)+1)\hat{r} -\varepsilon\rho^{\prime}(\theta)\hat{\theta})}{||\tilde\nu_\varepsilon||^{3}}\\
    \label{eq:normalfirstderivative}
    =&\frac{1}{||\tilde\nu_\varepsilon||^3}[(||\tilde\nu_\varepsilon||^2\rho(\theta)-\rho(\theta)-\varepsilon\rho^{\prime^{2}}(\theta)-2\varepsilon\rho^{2}(\theta)-\varepsilon^{2}\rho^{3}(\theta)-\varepsilon^{2}\rho(\theta)\rho^{\prime^{2}}(\theta))\hat{r}\\
    &-(||\tilde\nu_\varepsilon||^2\rho^\prime(\theta)-\varepsilon\rho(\theta)\rho^{\prime}(\theta)-\varepsilon^{2}\rho^{2}(\theta)\rho^{\prime}(\theta)-\varepsilon^{2}\rho^{\prime^{3}}(\theta))\hat{\theta}\nonumber\text{ . } 
\end{align}
Using $\frac{d\vec{\nu}_\vareps}{d\varepsilon}$, we calculate $\frac{d^2\vec{\nu}_\vareps}{d\varepsilon^2}$. Let 
\begin{align}
    \vec\nu_{\vareps, \hat{r}}&=||\tilde\nu_\varepsilon||^2\rho(\theta)-\rho(\theta)-\varepsilon\rho^{\prime^{2}}(\theta)-2\varepsilon\rho^{2}(\theta)-\varepsilon^{2}\rho^{3}(\theta)-\varepsilon^{2}\rho(\theta)\rho^{\prime^{2}}(\theta)\\
    \vec\nu_{\vareps, \hat{\theta}}&=||\tilde\nu_\varepsilon||^2\rho^\prime(\theta)-\varepsilon\rho(\theta)\rho^{\prime}(\theta)-\varepsilon^{2}\rho^{2}(\theta)\rho^{\prime}(\theta)-\varepsilon^{2}\rho^{\prime^{3}}(\theta)\text{ . } 
\end{align}
Then,
\begin{align}
    \frac{d^2\vec{\nu}_\vareps}{d\varepsilon^2} &= \frac{d}{d\varepsilon}[\frac{1}{||\tilde\nu_\varepsilon|^{3}}(\vec\nu_{\vareps, \hat{r}}\hat{r} - \vec\nu_{\vareps, \hat{\theta}}\hat{\theta})]\\
    &= (\frac{d}{d\varepsilon}\frac{1}{||\tilde\nu_\varepsilon||^3})(\vec\nu_{\vareps, \hat{r}}\hat{r}-\vec\nu_{\vareps, \hat{\theta}}\hat{\theta})+\frac{1}{||\tilde\nu_\varepsilon||^3}(\frac{d}{d\varepsilon}[\vec\nu_{\vareps, \hat{r}}\hat{r}-\vec\nu_{\vareps, \hat{\theta}}\hat{\theta}])\text{ . } 
\end{align}

We calculate these derivatives separately:
\begin{align}
    \frac{d}{d\varepsilon}\frac{1}{||\tilde\nu_\varepsilon||^3} =& \frac{d}{d\varepsilon} (\frac{1}{||\tilde\nu_\varepsilon||})^3\\
    =&3(\frac{1}{||\tilde\nu_\varepsilon||})^2\frac{d}{d\varepsilon}(\frac{1}{||\tilde\nu_\varepsilon||})\nonumber\\
    =&\frac{3}{||\tilde\nu_\varepsilon||^2}\frac{-1}{||\tilde\nu_\varepsilon||^3}(\rho(\theta)+\rho^{2}(\theta)\varepsilon+\rho^{\prime^2}(\theta)\varepsilon)\nonumber\\
    =&\frac{-3}{||\tilde\nu_\varepsilon||^5}(\rho(\theta)+\rho^{2}(\theta)\varepsilon+\rho^{\prime^2}(\theta)\varepsilon)\\
    \frac{d}{d\varepsilon}[\tilde\nu_{\vareps, \hat{r}}\hat{r}-\tilde\nu_{\vareps, \hat{\theta}}\hat{\theta}] =&\Big((\frac{d}{d\varepsilon}||\tilde\nu_\varepsilon||^2)\rho(\theta)- {\rho^\prime}^2(\theta)-2\rho^2(\theta)-2\varepsilon\rho^3(\theta)-2\varepsilon\rho(\theta){\rho^\prime}^2(\theta)\Big)\hat{r}\\
    &- \Big((\frac{d}{d\varepsilon}||\tilde\nu_\varepsilon||^2)\rho^\prime(\theta) -\rho(\theta)\rho^\prime(\theta)-2\varepsilon\rho^2(\theta)\rho^\prime(\theta)-2\varepsilon{\rho^\prime}^3(\theta)\Big)\hat{\theta}\nonumber\\
    =&\Big((2\rho(\theta)+ 2\varepsilon\rho^2(\theta) + 2\varepsilon{\rho^\prime}^2(\theta)
    )\rho(\theta) - {\rho^\prime}^2(\theta)-2\rho^2(\theta)-2\varepsilon\rho^3(\theta)\\
    &-2\varepsilon\rho(\theta){\rho^\prime}^2(\theta)\Big)\hat{r}- \Big((2\rho(\theta) + 2\varepsilon\rho^2(\theta) + 2\varepsilon{\rho^\prime}^2(\theta)
    )\rho^\prime(\theta)-\rho(\theta)\rho^\prime(\theta)\nonumber\\&-2\varepsilon\rho^2(\theta)\rho^\prime(\theta)-2\varepsilon{\rho^\prime}^3(\theta)\Big)\hat{\theta}\nonumber\text{ . } 
\end{align}
Substituting into the product rule gives
\begin{align}
    \label{eq:normalsecondderivative}
    \frac{d^2\vec{\nu}_\vareps}{d\varepsilon^2} =& \frac{-3}{||\tilde\nu_\varepsilon||^5}(\rho(\theta)+\rho^{2}(\theta)\varepsilon+\rho^{\prime^2}(\theta)\varepsilon)(\tilde\nu_{\vareps, \hat{r}}\hat{r}-\tilde\nu_{\vareps, \hat{\theta}}\hat{\theta})+\frac{1}{||\tilde\nu_\varepsilon||^3}(((2\rho(\theta)+ 2\varepsilon\rho^2(\theta)\\
    &+ 2\varepsilon{\rho^\prime}^2(\theta)
    )\rho(\theta) - {\rho^\prime}^2(\theta)-2\rho^2(\theta)-2\varepsilon\rho^3(\theta)-2\varepsilon\rho(\theta){\rho^\prime}^2(\theta))\hat{r}- ((2\rho(\theta) + 2\varepsilon\rho^2(\theta)\nonumber\\
    &+ 2\varepsilon{\rho^\prime}^2(\theta)
    )\rho^\prime(\theta)-\rho(\theta)\rho^\prime(\theta)-2\varepsilon\rho^2(\theta)\rho^\prime(\theta)-2\varepsilon{\rho^\prime}^3(\theta))\hat{\theta})\nonumber\text{ . } 
\end{align}
Now, we evaluate (\ref{eq:normalfirstderivative}) and (\ref{eq:normalsecondderivative}) at $\vareps = 0$:
\begin{align}
    \frac{d\vec{\nu}_\vareps}{d\varepsilon}(0) &=-\rho^\prime(\theta)\hat{\theta}\\
    \frac{d^2\vec{\nu}_\vareps}{d\varepsilon^2}(0) &= -{\rho^\prime}^2(\theta)\hat{r} + 2\rho(\theta)\rho^\prime (\theta)\hat{\theta}
\end{align}
Thus, the outward-facing unit normal has the expansion
\begin{align}
    \label{eq:normalexpansion}
    \vec{\nu}_\vareps= \hat{r} - \rho^\prime(\theta)\hat{\theta}\varepsilon + \Big(-{\rho^\prime}^2(\theta)\hat{r} + 2\rho(\theta)\rho^\prime(\theta)\hat{\theta}\Big)\frac{\varepsilon^2}{2}+O(\varepsilon^3)\text{ . } 
\end{align}

We wish to find first- and second- order expansion of $R_\vareps^k$, where $R_\vareps$ is the radius of $\partial \Omega_\vareps$ and  $k \in \N$ is arbitrary. From (\ref{eq:Domain}) we have that
\begin{align}
    R_\vareps = \frac{1+\vareps\rho(\theta)}{\sqrt{v(\vareps)}}\text{ . } 
\end{align}
So,
\begin{align}
    \label{eq:unexpandedrkboundry}
    R_\vareps^k &= (1+\vareps\rho(\theta))^k(v(\vareps))^{-\frac{k}{2}}\text{ . } 
\end{align}
We will find our expansion for $R^k_\vareps$ by expanding each of these factors. By the binomial theorem we have
\begin{align}
    \label{eq:rkexpansionbinomial}
    (1+\varepsilon\rho(\theta))^k = 1+k\rho(\theta)\varepsilon +\frac{1}{2}k(k-1)\rho^2(\theta)\varepsilon^2 + O(\varepsilon^3)\text{ , } 
\end{align}
and by Taylor's theorem we have
\begin{align}
    (v(\vareps))^{-\frac{k}{2}}= v^{-\frac{k}{2}}(0) + \frac{d}{d\varepsilon}[v^{-\frac{k}{2}}(\varepsilon)](0)\varepsilon + \frac{1}{2}\frac{d^2}{d\varepsilon^2}[v^{-\frac{k}{2}}(\varepsilon)](0)\varepsilon^2 + O(\varepsilon^3)
\end{align} We thus compute:
\begin{align}
    v^{-\frac{k}{2}}(0) =& \frac{1}{\sqrt{\pi^k}}\\
    \frac{d}{d\varepsilon}[v^{-\frac{k}{2}}(\varepsilon)](0) =& \Big(-\frac{k}{2}(v(\varepsilon))^{-\frac{k+2}{2}}\frac{dv}{d\varepsilon}\Big)(0)\\
    =&\Big(-\frac{k}{2}v(\varepsilon)^{-\frac{k+2}{2}}(2\pi b_0 + \pi\varepsilon(2b_0^2+\sum_{j=1}^\infty a_j^2 + b_j^2))\Big)(0)\nonumber\\
    =&-\frac{k}{\sqrt{\pi^k}}b_0\\
    \frac{d^2}{d\varepsilon^2}[v^{-\frac{k}{2}}(\varepsilon)](0)=&\frac{d}{d\varepsilon}[-\frac{k}{2}v(\varepsilon)^{-\frac{k+2}{2}}(2\pi b_0 + \pi\varepsilon(2b_0^2+\sum_{j=1}^\infty a_j^2 + b_j^2)](0)\\
    =&(\frac{k(k+2)}{4}(v(\varepsilon))^{-\frac{k+4}{2}}(2\pi b_0 + \pi\varepsilon(2b_0^2+\sum_{j=1}^\infty a_j^2 + b_j^2)^2\nonumber\\
    &-\frac{k}{2}\pi(v(\varepsilon))^{-\frac{k+2}{2}}(2b_0^2+\sum_{j=1}^\infty a_j^2 + b_j^2))(0)\nonumber\\
    =&k\frac{(k+1)b_0^2-\frac{1}{2}\sum_{j=1}^\infty a_j^2 + b_j^2}{\sqrt{\pi^k}}\text{ . } 
\end{align}
So we have
\begin{align}
    \label{eq:rkexpansiontaylors}
    (v(\vareps))^{-\frac{k}{2}} &= \frac{1}{\sqrt{\pi^k}} -\frac{k}{\sqrt{\pi^k}}b_0\varepsilon + k\frac{(k+1)b_0^2-\frac{1}{2}\sum_{j=1}^\infty a_j^2 + b_j^2}{2\sqrt{\pi^k}}\varepsilon^2 + O(\varepsilon^3)\text{ . } 
\end{align}
We substitute (\ref{eq:rkexpansionbinomial}) and (\ref{eq:rkexpansiontaylors}) into (\ref{eq:unexpandedrkboundry}):
\begin{align}
    R_\vareps^k =& \Big(\frac{1}{\sqrt{\pi^k}} -\frac{k}{\sqrt{\pi^k}}b_0\varepsilon + k\frac{(k+1)b_0^2-\frac{1}{2}\sum_{j=1}^\infty a_j^2 + b_j^2}{2\sqrt{\pi^k}}\varepsilon^2\Big)\\
    &\cdot\Big(1+k\rho(\theta)\varepsilon +\frac{1}{2}k(k-1)\varepsilon^2\rho^2(\theta)\Big) + O(\varepsilon^3)\nonumber\\
    \label{rkbdry}
    =&\frac{1}{\sqrt{\pi^k}} + \frac{k}{\sqrt{\pi^k}}(\rho(\theta)-b_0)\varepsilon + \frac{k}{2\sqrt{\pi^k}}\Big((k-1)\rho^2(\theta)+(k+1)b_0^2\\
    &-\frac{1}{2}\sum_{j=1}^\infty [a_j^2+b_j^2]-2kb_0\rho(\theta)\Big)\varepsilon^2+O(\varepsilon^3)\nonumber\text{ . } 
\end{align}

\subsection{Analytic Dependence on $\varepsilon >0$}
  In \cite{viator2020}, the authors show that the perturbation of Steklov eigenvalues associated with domains of the form \eqref{eq:genDomain} is analytic in the parameter $\varepsilon$, or more specifically, that the Dirichlet-to-Neumann operator (DNO) is analytic in this parameter.  This was accomplished by introducing new coordinates
\begin{align}
\label{analyticcoords}
(r', \theta')= ((1+\varepsilon \rho (\theta))^{-1}r, \theta) \text{ , }
\end{align}
which transforms $\omega_\varepsilon$ to the disk and the Laplacian into a more sophisticated $\varepsilon$-dependent operator.  We then show that the harmonic extension of functions defined on $\partial \omega_\varepsilon$, which are also transformed under the above change of coordinates, depend analytically on $\varepsilon$, and then leverage the analyticity of harmonic extensions and formula for the normal derivative in the new coordinates to show that the DNO is analytic as well.

This process can be repeated for the perturbation problem for area-normalized domains \eqref{eq:Domain} presented in this paper.  Here, the corresponding change of coordinates would be
\begin{align}
\label{analyticcoordsnormal}
(r', \theta') = \Big( (\frac{\sqrt{v(\varepsilon)}}{1+\varepsilon \rho(\theta)} r, \theta \Big)\text{ . } 
\end{align}
Proceeding as in \cite{viator2020} section 2.1, we see that the term $\sqrt{v(\varepsilon)}$ comes out as a non-zero factor in the tranformed Laplacian, which can be divided from both sides of the homogeneous equation which the harmonic extension must satisfy.  The analyticity of the harmonic extensions then follows in the same manner as \cite{viator2020}.  

To prove the analyticity of the DNO, the authors of \cite{viator2020} again work in the tranformed coordinates \eqref{analyticcoords}, obtaining the following representation of the DNO $G_{\rho, \varepsilon}$:
\begin{align*}
G_{\rho,\vareps} \varphi & = M_{\rho}(\vareps)\left [ \left ( 1 + \frac{\vareps^2(\rho')^2}{(1+\vareps \rho)^2} \right )\partial_r u_\vareps - \frac{\vareps \rho'}{1+\vareps\rho}\partial_\theta u_\vareps \right ] 
\end{align*}
where $M_\rho(\varepsilon)$ is an analytic real-valued function of $\varepsilon$, and $u_\vareps$ is the harmonic extension of $\varphi$ into the perturbed (non-area-normalized) domain.  Applying the same method to the domain \eqref{eq:Domain} and the cooresponding change of coordinates \eqref{analyticcoordsnormal}, it is easy to see that 
\begin{align*}
G_{\rho,\vareps} \varphi & = M_{\rho}(\vareps)\sqrt{v(\vareps)}\left [ \left ( 1 + \frac{\vareps^2(\rho')^2}{(1+\vareps \rho)^2} \right )\partial_r u_\vareps - \frac{\vareps \rho'}{1+\vareps\rho}\partial_\theta u_\vareps \right ] \\
& = \tilde{M}_\rho(\varepsilon)\left [ \left ( 1 + \frac{\vareps^2(\rho')^2}{(1+\vareps \rho)^2} \right )\partial_r u_\vareps - \frac{\vareps \rho'}{1+\vareps\rho}\partial_\theta u_\vareps \right ] 
\end{align*}
where $\tilde{M}_{\rho}(\vareps) = \sqrt{v(\varepsilon)}M_{\rho}(\vareps)$.  Since $\sqrt{v(\varepsilon)}$ is an analytic function of $\varepsilon$ in a neighborhood of $\varepsilon=0$, the analyticity of the Dirichlet-to-Neumann operator associated with domains \eqref{eq:Domain} in the parameter $\varepsilon$, and hence the analyticity of the Steklov eigenvalues of $\Omega_\varepsilon$, follows.


\section{An Asymptotic Expansion for Steklov Eigenvalues of Nearly Circular Domains} 
\label{solnasymptotics}
In this section, we derive the asymptotic expansion for the Steklov eigenvalues $\lambda(\vareps)\coloneqq \lambda^\vareps$ on a volume-normalized nearly-circular domain $\Omega_\vareps$ of the form \eqref{eq:Domain}. Recall that the unperturbed eigenvalues of $\Omega_0 = \frac{1}{\sqrt{\pi}}D$ are $\ell \sqrt{\pi}$ for $\ell\in\N$, with corresponding eigenfunctions $r^\ell \operatorname{cos}(\ell \theta)$ and $r^\ell \operatorname{sin}(\ell \theta)$.

\subsection{Left Hand Side ($\partial_{\vec{\nu}}\ u$)}

Recall, $\partial_{\vec{\nu}}\ u = \nabla u \cdot \vec{\nu}_\vareps$.

We first calculate $\nabla u$, substituting our expansion \eqref{eq:ansatz1} into $\eqref{eq:polarGradient}$:
\begin{align}
    \nabla u =& \frac{\partial}{\partial r}[\sum_{k=0}^\infty r^k(A_k(\varepsilon)\sin(k\theta)+B_k(\varepsilon)\cos(k\theta))]\hat{r} \nonumber \\
    &+\frac{1}{r}\frac{\partial}{\partial \theta}[\sum_{k=0}^\infty r^k(A_k(\varepsilon)\sin(k\theta)+B_k(\varepsilon)\cos(k\theta))]\hat{\theta} \nonumber \\
    =&\sum_{k=0}^\infty kr^{k-1}[(A_k(\varepsilon)\sin(k\theta)+B_k(\varepsilon)\cos(k\theta))\hat{r} + (A_k(\varepsilon)\cos(k\theta)-B_k(\varepsilon)\sin(k\theta))\hat{\theta}]\text{ . } 
\end{align}
Recalling (\ref{eq:normalexpansion}) we calculate $\partial_{\vec{\nu}}\ u$:
\begin{align}
    \partial_{\vec{\nu}}\ u=& \sum_{k=0}^\infty kr^{k-1}[(A_k(\varepsilon)\sin(k\theta)+B_k(\varepsilon)\cos(k\theta))\hat{r} + (A_k(\varepsilon)\cos(k\theta)-B_k(\varepsilon)\sin(k\theta))\hat{\theta}]\\
    &\cdot (\hat{r} - \rho^\prime(\theta)\hat{\theta}\varepsilon + (-(\rho^\prime(\theta))^2\hat{r} + 2\rho(\theta)\rho^\prime(\theta) \hat{\theta})\frac{\varepsilon^2}{2}+O(\varepsilon^3))\nonumber\\
    =&\sum_{k=0}^\infty kr^{k-1}[(A_k(\varepsilon)\sin(k\theta)+B_k(\varepsilon)\cos(k\theta))\hat{r} + (A_k(\varepsilon)\cos(k\theta)-B_k(\varepsilon)\sin(k\theta))\hat{\theta}]\nonumber\\
    &\cdot ((1-(\rho^\prime(\theta))^2\frac{\varepsilon^2}{2})\hat{r}+ (\rho(\theta)\rho^\prime(\theta)\varepsilon^2- \rho^\prime(\theta)\varepsilon)\hat{\theta}+O(\varepsilon^3))\nonumber\\
    =& \sum_{k=0}^\infty kr^{k-1}[(A_k(\varepsilon)\sin(k\theta)+B_k(\varepsilon)\cos(k\theta))(1-(\rho^\prime(\theta))^2\frac{\varepsilon^2}{2})\\
    &+ (A_k(\varepsilon)\cos(k\theta)-B_k(\varepsilon)\sin(k\theta))(\rho(\theta)\rho^\prime(\theta)\varepsilon^2- \rho^\prime(\theta)\varepsilon)] + O(\varepsilon^3)\nonumber\text{ . } 
\end{align}
Recalling \eqref{eq:ansatz2} and \eqref{eq:ansatz3}, we group by powers of epsilon:
\begin{align}
     \partial_{\vec{\nu}}\ u=&\sum_{k=0}^\infty kr^{k-1}[((\delta_{k,n}\gamma_n + \mu_k\varepsilon + \eta_k\varepsilon^2)\sin(k\theta)+(\delta_{k,n}\alpha_n + \beta_k\varepsilon + \sigma_k\varepsilon^2)\cos(k\theta))\\
     &\cdot(1-(\rho^\prime(\theta))^2\frac{\varepsilon^2}{2})+ ((\delta_{k,n}\gamma_n + \mu_k\varepsilon + \eta_k\varepsilon^2)\cos(k\theta)-(\delta_{k,n}\alpha_n + \beta_k\varepsilon + \sigma_k\varepsilon^2)\sin(k\theta)) \nonumber \\
     &\cdot(\rho(\theta)\rho^\prime(\theta)\varepsilon^2- \rho^\prime(\theta)\varepsilon)] + O(\varepsilon^3) \nonumber \\
     \label{dnuprerk}
     =&\sum_{k=0}^\infty kr^{k-1}[(\delta_{k,n}\alpha_n\cos(k\theta)+\delta_{k,n}\gamma_n\sin(k\theta))  \\ &+ \Big (\rho^\prime(\theta)\delta_{k,n}\alpha_n\sin(k\theta)
     - \rho^\prime(\theta)\delta_{k,n}\gamma_n\cos(k\theta) + \beta_k\cos(k\theta) + \mu_k\sin(k\theta) \Big )\varepsilon \nonumber \\ & + \Big (-\frac{1}{2}(\rho^\prime(\theta))^2\delta_{k,n}\alpha_n\cos(k\theta)
     -\frac{1}{2}(\rho^\prime(\theta))^2\delta_{k,n}\gamma_n\sin(k\theta) + \rho^\prime(\theta)\beta_k\sin(k\theta) \nonumber \\ & - \rho^\prime(\theta)\mu_k\cos(k\theta) -\rho(\theta)\rho^\prime(\theta)\delta_{k,n}\alpha_n\sin(k\theta)+\rho(\theta)\rho^\prime(\theta)\delta_{k,n}\gamma_n\cos(k\theta) \nonumber \\ &+ \sigma_k\cos(k\theta) + \eta_k\sin(k\theta) \Big )\varepsilon^2] \nonumber \\
     &+ O(\varepsilon^3) \nonumber\text{ . } 
\end{align}
Since the normal derivative is taken on $\partial \Omega_\vareps$, we substitute \eqref{rkbdry} into \eqref{dnuprerk} to see:
\begin{align}
     \partial_{\vec{\nu}}\ u= &\sum_{k=0}^\infty \frac{k}{\sqrt{\pi^{k-1}}}(1 + (k-1)(\rho(\theta)-b_0)\varepsilon + \frac{k-1}{2}((k-2)\rho^2(\theta)+kb_0^2\\
     &-\frac{1}{2}\sum_{j=1}^\infty[a_j^2+b_j^2]-2(k-1)b_0\rho(\theta))\varepsilon^2)[(\delta_{k,n}\alpha_n\cos(k\theta)+\delta_{k,n}\gamma_n\sin(k\theta)) \nonumber \\
     &+ (\rho^\prime(\theta)\delta_{k,n}\alpha_n\sin(k\theta) - \rho^\prime(\theta)\delta_{k,n}\gamma_n\cos(k\theta) + \beta_k\cos(k\theta) + \mu_k\sin(k\theta))\varepsilon \nonumber  \\
     &+ (-\frac{1}{2}(\rho^\prime(\theta))^2\delta_{k,n}\alpha_n\cos(k\theta) -\frac{1}{2}(\rho^\prime(\theta))^2\delta_{k,n}\gamma_n\sin(k\theta) + \rho^\prime(\theta)\beta_k\sin(k\theta) \nonumber \\
     &-\rho^\prime(\theta)\mu_k\cos(k\theta)-\rho(\theta)\rho^\prime(\theta)\delta_{k,n}\alpha_n\sin(k\theta)+\rho(\theta)\rho^\prime(\theta)\delta_{k,n}\gamma_n\cos(k\theta) \nonumber \\
     &+ \sigma_k\cos(k\theta) + \eta_k\sin(k\theta))\varepsilon^2]+ O(\varepsilon^3) \nonumber \\
     \label{dnuexpanded}
     =&\sum_{k=0}^\infty \frac{k}{\sqrt{\pi^{k-1}}}[(\delta_{k,n}\alpha_n\cos(k\theta) + \delta_{k,n}\gamma_n\sin(k\theta)) \\ 
     &+ \Big ( (k-1)(\rho(\theta)-b_0)(\delta_{k,n}\alpha_n\cos(k\theta) \delta_{k,n}\gamma_n\sin(k\theta)) + \beta_k\cos(k\theta) + \mu_k\sin(k\theta) \nonumber \\ &+\rho^\prime(\theta)\delta_{k,n}\alpha_n\sin(k\theta) -\rho^\prime(\theta)\delta_{k,n}\gamma_n\cos(k\theta) \Big )\varepsilon \nonumber \\ 
     & + \Big ( \frac{k-1}{2}((k-2)\rho^2(\theta) + kb_0^2-\frac{1}{2}\sum_{j=1}^\infty[a_j^2 + b_j^2]  \nonumber \\
 & -2(k-1)b_0\rho(\theta))(\delta_{k,n}\alpha_n\cos(k\theta) + \delta_{k,n}\gamma_n\sin(k\theta))  -\frac{1}{2}(\rho^\prime(\theta))^2\delta_{k,n}\alpha_n\cos(k\theta) \nonumber \\ &-\frac{1}{2}(\rho^\prime(\theta))^2\delta_{k,n}\gamma_n\sin(k\theta) + \rho^\prime(\theta)\beta_k\sin(k\theta)-\rho^\prime(\theta)\mu_k\cos(k\theta)  \nonumber \\
 & -\rho(\theta)\rho^\prime(\theta)\delta_{k,n}\alpha_n\sin(k\theta) 
     +\rho(\theta)\rho^\prime(\theta)\delta_{k,n}\gamma_n\cos(k\theta)+ \sigma_k\cos(k\theta) + \eta_k\sin(k\theta) \nonumber \\ 
     & + (k-1)(\rho(\theta)-b_0)(\beta_k\cos(k\theta) +\mu_k\sin(k\theta) + \rho^\prime(\theta)\delta_{k,n}\alpha_n\sin(k\theta) \nonumber \\ & - \rho^\prime(\theta)\delta_{k,n}\gamma_n\cos(k\theta)) \Big )\varepsilon^2] \nonumber \\ & + O(\varepsilon^3) \nonumber \text{ . } 
\end{align}
\subsection{Right Hand Side ($\lambda u$)}
Recall that $\lambda_n$, the $n$th eigenvalue (not counting multiplicity), is analytic in the parameter $\varepsilon$. Thus we form the power series expansion
\begin{align}
    \lambda_n(\varepsilon) = \lambda_n^{(0)} + \lambda_n^{(1)}\varepsilon + \lambda_n^{(2)}\varepsilon^2 + O(\varepsilon^3)
\end{align}
with which, after using expansions \eqref{eq:ansatz1}, \eqref{eq:ansatz2}, \eqref{eq:ansatz3} and grouping by powers of $\varepsilon$, we find
\begin{align}
    \lambda u =& (\lambda_n^{(0)} + \lambda_n^{(1)}\varepsilon + \lambda_n^{(2)}\varepsilon^2)\sum_{k=0}^\infty r^k[(\delta_{k,n}\gamma_n + \mu_k\varepsilon + \eta_k\varepsilon^2)\sin(k\theta)+(\delta_{k,n}\alpha_n + \beta_k\varepsilon\\
    &+ \sigma_k\varepsilon^2)\cos(k\theta)] + O(\varepsilon^3) \nonumber \\
    \label{prelambdau}
    =&\sum_{k=0}^\infty r^k[\lambda_n^{(0)}(\delta_{k,n}\alpha_n\cos(k\theta)+\delta_{k,n}\gamma_n\sin(k\theta)) + (\lambda_n^{(0)}(\beta_k\cos(k\theta)+\mu_k\sin(k\theta)) \\
    &+ \lambda_n^{(1)}(\delta_{k,n}\alpha_n\cos(k\theta)+\delta_{k,n}\gamma_n\sin(k\theta)))\varepsilon + (\lambda_n^{(0)}(\sigma_k\cos(k\theta) + \eta_k\sin(k\theta)) \nonumber\\
    &+ \lambda_n^{(1)}(\beta_k\cos(k\theta)+\mu_k\sin(k\theta)) + \lambda_n^{(2)}(\delta_{k,n}\alpha_n\cos(k\theta)+\delta_{k,n}\gamma_n\sin(k\theta)))\varepsilon^2]+ O(\varepsilon^3) \nonumber\text{ . } 
\end{align} 
Since the eigenvalue condition is on $\partial \Omega_\vareps$, we substitute \eqref{rkbdry} into \eqref{prelambdau} to obtain
\begin{align}
    \lambda u =& \sum_{k=0}^\infty \frac{1}{\sqrt{\pi^k}}(1 + k(\rho(\theta) - b_0)\varepsilon + \frac{k}{2}((k-1)\rho^2(\theta)+(k+1)b_0^2-\frac{1}{2}\sum_{j=1}^\infty[a_j^2+b_j^2]\\
    &-2kb_0\rho(\theta))\varepsilon^2)[\lambda_n^{(0)}(\delta_{k,n}\alpha_n\cos(k\theta)+\delta_{k,n}\gamma_n\sin(k\theta)) + \Big ( \lambda_n^{(0)}(\beta_k\cos(k\theta) \nonumber \\
    &+\mu_k\sin(k\theta))+ \lambda_n^{(1)}(\delta_{k,n}\alpha_n\cos(k\theta)+\delta_{k,n}\gamma_n\sin(k\theta))\Big )\varepsilon + \Big ( \lambda_n^{(0)}(\sigma_k\cos(k\theta) + \eta_k\sin(k\theta)) \nonumber  \\
    &+ \lambda_n^{(1)}(\beta_k\cos(k\theta)+\mu_k\sin(k\theta)) + \lambda_n^{(2)}(\delta_{k,n}\alpha_n\cos(k\theta)+\delta_{k,n}\gamma_n\sin(k\theta))\Big )\varepsilon^2]+ O(\varepsilon^3) \nonumber \\
    \label{lambdaupresub}
    =&\sum_{k=0}^\infty \frac{1}{\sqrt{\pi^k}}[\lambda_n^{(0)}(\delta_{k,n}\alpha_n\cos(k\theta)+\delta_{k,n}\gamma_n\sin(k\theta)) \\  & + \Big (\lambda_n^{(0)}(\beta_k\cos(k\theta) + \mu_k\sin(k\theta))
    + \lambda_n^{(1)}(\delta_{k,n}\alpha_n\cos(k\theta)+\delta_{k,n}\gamma_n\sin(k\theta))\nonumber \\ & + k\lambda_n^{(0)}(\rho(\theta)-b_0)(\delta_{k,n}\alpha_n\cos(k\theta) +\delta_{k,n}\gamma_n\sin(k\theta)) \Big )\varepsilon \nonumber \\ 
    & + \Big ( \lambda_n^{(0)}(\sigma_k\cos(k\theta)+\eta_k\sin(k\theta)) + \lambda_n^{(1)}(\beta_k\cos(k\theta)+\mu_k\sin(k\theta)) \nonumber \\
    &+ \lambda_n^{(2)}(\delta_{k,n}\alpha_n\cos(k\theta)+\delta_{k,n}\gamma_n\sin(k\theta)) + \lambda_n^{(0)}\frac{k}{2}((k-1)\rho^2(\theta)+(k+1)b_0^2 \nonumber \\
    &-\frac{1}{2}\sum_{j=1}^\infty[a_j^2+b_j^2]-2kb_0\rho(\theta))(\delta_{k,n}\alpha_n\cos(k\theta)+\delta_{k,n}\gamma_n\sin(k\theta))\nonumber \\ 
    &+k(\rho(\theta)-b_0) \cdot(\lambda_n^{(0)}(\beta_k\cos(k\theta)\mu_k\sin(k\theta))+ \lambda_n^{(1)}(\delta_{k,n}\alpha_n\cos(k\theta)+\delta_{k,n}\gamma_n\sin(k\theta))) \Big )\varepsilon^2] \nonumber \\ &+ O(\varepsilon^3) \nonumber \text{ . } 
\end{align}
\subsection{Leading order terms} Equating the leading order terms in \eqref{dnuexpanded} and \eqref{lambdaupresub} gives
\begin{align}
    \sum_{k=0}^\infty \frac{k}{\sqrt{\pi^{k-1}}}(\delta_{k,n}\alpha_n\cos(n\theta) + \delta_{k_n}\gamma_n\sin(n\theta)) = \sum_{k=0}^\infty\frac{\lambda_n^{(0)}}{\sqrt{\pi^{k}}}(\delta_{k,n}\alpha_n\cos(n\theta) + \delta_{k_n}\gamma_n\sin(n\theta))\text{ . }
\end{align}
Then, collapsing the sum using $\delta_{k,n} = \begin{cases}
    0 & k\neq n\\
    1 & k=n
\end{cases}$ obtains:
\begin{align}
    n\sqrt{\pi}(\alpha_n\cos(n\theta) + \gamma_n\sin(n\theta)) = \lambda_n^{(0)}(\alpha_n\cos(n\theta) + \gamma_n\sin(n\theta))\text{ . }
\end{align}
which recovers $\lambda_n^{(0)} = n\sqrt{\pi}$, the solution of the Steklov problem on $\Omega_0=\frac{1}{\sqrt{\pi}}D$ from Section \ref{sec:prelim}. We make this substitution into \eqref{lambdaupresub} to obtain: 
\begin{align}
\label{lambdauexpanded}
    \lambda u = 
    &\sum_{k=0}^\infty \frac{1}{\sqrt{\pi^k}}[n\sqrt{\pi}(\delta_{k,n}\alpha_n\cos(k\theta)+\delta_{k,n}\gamma_n\sin(k\theta)) \\
    &+ \Big (n\sqrt{\pi}(\beta_k\cos(k\theta) + \mu_k\sin(k\theta))
    + \lambda_n^{(1)}(\delta_{k,n}\alpha_n\cos(k\theta)+\delta_{k,n}\gamma_n\sin(k\theta)) \nonumber \\  
    & + kn\sqrt{\pi}(\rho(\theta)-b_0)(\delta_{k,n}\alpha_n\cos(k\theta )
    +\delta_{k,n}\gamma_n\sin(k\theta)) \Big )\varepsilon  \nonumber \\
    &+ \Big (n\sqrt{\pi}(\sigma_k\cos(k\theta)+\eta_k\sin(k\theta)) + \lambda_n^{(1)}(\beta_k\cos(k\theta)+\mu_k\sin(k\theta)) \nonumber \\
    &+ \lambda_n^{(2)}(\delta_{k,n}\alpha_n\cos(k\theta)+\delta_{k,n}\gamma_n\sin(k\theta)) + n\sqrt{\pi}\frac{k}{2}((k-1)\rho^2(\theta)+(k+1)b_0^2 \nonumber \\
    &-\frac{1}{2}\sum_{j=1}^\infty[a_j^2+b_j^2]-2kb_0\rho(\theta))(\delta_{k,n}\alpha_n\cos(k\theta)+\delta_{k,n}\gamma_n\sin(k\theta)) \nonumber \\
    & +k(\rho(\theta)-b_0)\cdot(n\sqrt{\pi}(\beta_k\cos(k\theta)+\mu_k\sin(k\theta))+ \\ &\lambda_n^{(1)}(\delta_{k,n}\alpha_n\cos(k\theta)+\delta_{k,n}\gamma_n\sin(k\theta))) \Big )\varepsilon^2] \nonumber \\ &+ O(\varepsilon^3) \nonumber\text{ . } 
\end{align}

\section{First Order}
\label{sec:asymptotic-first} 
\subsection{Forming $M_n^{(1)}$}
\label{formingmn1}
We equate the order-$\varepsilon$ terms of \eqref{dnuexpanded} and \eqref{lambdauexpanded}:
\begin{align}
    &\sum_{k=0}^\infty \frac{k}{\sqrt{\pi^{k-1}}}((k-1)(\rho(\theta)-b_0)(\delta_{k,n}\alpha_n\cos(k\theta) + \delta_{k,n}\gamma_n\sin(k\theta)) + \beta_k\cos(k\theta)\\
    &+ \mu_k\sin(k\theta) +\rho^\prime(\theta)\delta_{k,n}\alpha_n\sin(k\theta) -\rho^\prime(\theta)\delta_{k,n}\gamma_n\cos(k\theta)) \nonumber\\
    =&\sum_{k=0}^\infty \frac{1}{\sqrt{\pi^k}}(n\sqrt{\pi}(\beta_k\cos(k\theta) + \mu_k\sin(k\theta)) + \lambda_n^{(1)}(\delta_{k,n}\alpha_n\cos(k\theta)+\delta_{k,n}\gamma_n\sin(k\theta))\\
    &+ kn\sqrt{\pi}(\rho(\theta)-b_0)(\delta_{k,n}\alpha_n\cos(k\theta)+\delta_{k,n}\gamma_n\sin(k\theta)))\nonumber
\end{align}
which after rearranging terms yields the equation
\begin{align}
\label{firstordersummation}
    &\sum_{k=0}^\infty \frac{1}{\sqrt{\pi^k}}\lambda_n^{(1)}(\delta_{k,n}\alpha_n\cos(k\theta)+\delta_{k,n}\gamma_n\sin(k\theta)) \\
    & =\sum_{k=0}^\infty[ \frac{k}{\sqrt{\pi^{k-1}}}((k-1-n)(\rho(\theta)-b_0)(\delta_{k,n}\alpha_n\cos(k\theta) + \delta_{k,n}\gamma_n\sin(k\theta)) \nonumber \\ 
    &+ (1 - \frac{n}{k}) (\beta_k\cos(k\theta) + \mu_k\sin(k\theta)) 
    +\rho^\prime(\theta)\delta_{k,n}(\alpha_n\sin(k\theta) -\gamma_n\cos(k\theta))] \nonumber\text{ . } 
\end{align}
We recall that $\delta_{k,n} = \begin{cases}
    0 & k\neq n\\
    1 & k=n
\end{cases}$ to pull terms out of sums where possible:
\begin{align}
\label{orderepscollapsed}
    \frac{\lambda_n^{(1)}}{\sqrt{\pi^n}}(\alpha_n\cos(n\theta)+\gamma_n\sin(n\theta)) =& \frac{1}{\sqrt{\pi^{n-1}}}[-n\rho(\theta)(\alpha_n\cos(n\theta) + \gamma_n\sin(n\theta))\\
    &+n\rho^\prime(\theta)(\alpha_n\sin(n\theta)- \gamma_n\cos(n\theta))+nb_0(\alpha_n\cos(n\theta) \nonumber \\
    &+ \gamma_n\sin(n\theta))]+ \sum_{k=0}^\infty \frac{1}{\sqrt{\pi^{k-1}}}[k(\beta_k\cos(k\theta) + \mu_k\sin(k\theta)) \nonumber \\
    &-n(\beta_k\cos(k\theta) + \mu_k\sin(k\theta))] \nonumber\text{ . } 
\end{align}
From here we may generate two equations using ``Fourier's Trick": we recall that, from the orthogonality of sines and cosines, we have, $\forall i \in \N\  \forall j \in \N\setminus\{0\}$
\begin{align}
    \label{sinecosine}
    \int_0^{2\pi}\sin(i\theta)\cos(j\theta)d\theta &= 0\\
    \label{sinesinecosinecosine}
    \int_0^{2\pi}\sin(i\theta)\sin(j\theta) d\theta = \int_0^{2\pi}\cos(i\theta)\cos(j\theta) d\theta &= \begin{cases}
    0 & i \neq j\\
    \pi & i = j
\end{cases}\text{ . } 
\end{align}
Thus, we first multiply \eqref{orderepscollapsed} by $\cos(n\theta)$ and integrate both sides from $0$ to $2\pi$ to obtain
\begin{align}
    \frac{1}{\sqrt{\pi^{n-2}}}\lambda_n^{(1)}\alpha_n =& \frac{1}{\sqrt{\pi^{n-1}}}\Big[\int_0^{2\pi} n\rho^\prime(\theta)\sin(n\theta)\cos(n\theta) -n\rho(\theta)\cos(n\theta)\cos(n\theta) d\theta + n\pi b_0\Big]\alpha_n\\
    &+ \frac{1}{\sqrt{\pi^{n-1}}}\Big[\int_0^{2\pi} -n\rho^\prime(\theta)\cos(n\theta)\cos(n\theta) - n\rho(\theta)\sin(n\theta)\cos(n\theta)d\theta\Big]\gamma_n \nonumber \\
    &+\frac{1}{\sqrt{\pi^{n-1}}}[n\pi\beta_n - n\pi\beta_n] \text{ . } \nonumber
\end{align}
That is,
\begin{align}
    \label{lambda1cosintegrals}
    \lambda_n^{(1)}\alpha_n =& \Big (\frac{n}{\sqrt{\pi}}\int_0^{2\pi} \rho^\prime(\theta)\sin(n\theta)\cos(n\theta) -\rho(\theta)\cos^2(n\theta) d\theta + n\sqrt{\pi} b_0 \Big )\alpha_n \\
    &+ \Big ( \frac{n}{\sqrt{\pi}}\int_0^{2\pi} -\rho^\prime(\theta)\cos^2(n\theta) - \rho(\theta)\sin(n\theta)\cos(n\theta)d\theta \Big )\gamma_n \nonumber\text{ . } 
\end{align}
Similarly, multiplying \eqref{orderepscollapsed} by $\sin(n\theta)$ and integrating reveals
\begin{align}
    \label{lambdasinintegrals}
    \lambda_n^{(1)}\gamma_n =& \Big(\frac{n}{\sqrt{\pi}}\int_0^{2\pi}\rho^\prime(\theta)\sin^2(n\theta) -\rho(\theta)\cos(n\theta)\sin(n\theta) d\theta \Big )\alpha_n\\
    &+ \Big (\frac{n}{\sqrt{\pi}}\int_0^{2\pi} -\rho^\prime(\theta)\cos(n\theta)\sin(n\theta)-\rho(\theta)\sin^2(n\theta)d\theta + n\sqrt{\pi} b_0 \Big )\gamma_n \nonumber\text{ . } 
\end{align}
So we see that
\begin{align}
    \lambda_n^{(1)}\begin{bmatrix}
        \alpha_n\\
        \gamma_n
    \end{bmatrix} = M_n^{(1)} \begin{bmatrix}
        \alpha_n\\
        \gamma_n
    \end{bmatrix}\text{ . } 
\end{align} i.e. $\lambda_n^{(1)}$ must be an eigenvalue of the matrix $M_n^{(1)} = \begin{bmatrix}
    m_{n,\,(1,1)}^{(1)} & m_{n,\,(1,2)}^{(1)}\\
    m_{n,\,(2,1)}^{(1)} & m_{n,\,(2,2)}^{(1)}
\end{bmatrix}$ given by
\begin{align}
    m_{n,\,(1,1)}^{(1)} &= \frac{n}{\sqrt{\pi}}\int_0^{2\pi} \rho^\prime(\theta)\sin(n\theta)\cos(n\theta) -\rho(\theta)\cos^2(n\theta) d\theta + n\sqrt{\pi} b_0\\
    m_{n,\,(1,2)}^{(1)} &= \frac{n}{\sqrt{\pi}}\int_0^{2\pi} -\rho^\prime(\theta)\cos^2(n\theta) - \rho(\theta)\sin(n\theta)\cos(n\theta)d\theta\\
    m_{n,\,(2,1)}^{(1)} &= \frac{n}{\sqrt{\pi}}\int_0^{2\pi}\rho^\prime(\theta)\sin^2(n\theta) -\rho(\theta)\cos(n\theta)\sin(n\theta) d\theta\\
    m_{n,\,(2,2)}^{(1)} &= \frac{n}{\sqrt{\pi}}\int_0^{2\pi} -\rho^\prime(\theta)\cos(n\theta)\sin(n\theta)-\rho(\theta)\sin^2(n\theta)d\theta + n\sqrt{\pi} b_0\text{ . } 
\end{align}
\subsection{Computing $\lambda_n^{(1)}$}
\label{computinglambdan1}
To compute $\lambda_n^{(1)}$ we must evaluate the integrals in the components of $M_n^{(1)}$. To do so, we recall our Fourier expansion of $\rho$ from (\ref{eq:rhoFourier}) and thus also calculate one for $\rho^\prime$:
\begin{align}
    \rho(\theta) &= \sum_{j=0}^\infty a_j\sin(j\theta) + b_j\cos(j\theta)\\
    \label{eq:rhoPrimeFourier}
    \rho^\prime(\theta) &= \sum_{j=0}^\infty ja_j\cos(j\theta)-jb_j\sin(j\theta)\text{ . } 
\end{align}

Finally, we again apply from the orthogonality of sine and cosine as in \eqref{sinecosine} and \eqref{sinesinecosinecosine}.
So now we may compute the following integrals:
    
\begin{align}
        \int_0^{2\pi}(\rho^\prime(\theta)\sin(n\theta)\cos(n\theta)&-\rho(\theta)\cos^2(n\theta))d\theta  \\
        =&\frac{1}{2}\int_0^{2\pi}\Big ( \sin(2n\theta)(\sum_{j=0}^\infty ja_j\cos(j\theta) - jb_j\sin(j\theta)) \nonumber \\
        &-( (1+\cos(2n\theta)) (\sum_{j=0}^\infty a_j\sin(j\theta) + b_j\cos(j\theta))\Big )d\theta \nonumber\\
        =&-\pi n b_{2n} - \frac{1}{2}\pi b_{2n} - \pi b_0
\end{align}

    \begin{align}
        \int_0^{2\pi}(\rho(\theta)\sin(n\theta)\cos(n\theta)&+\rho^\prime(\theta)\cos^2(n\theta))d\theta\\
        =&\frac{1}{2}\int_0^{2\pi}\Big (\sin(2n\theta)(\sum_{j=0}^\infty a_j\sin(j\theta) + b_j\cos(j\theta)) \nonumber \\
        &+(1+\cos(2n\theta))(\sum_{j=0}^\infty ja_j\cos(j\theta) - jb_j\sin(j\theta))\Big )d\theta \nonumber \\
        =&n\pi a_{2n} + \frac{1}{2}\pi a_{2n}
    \end{align}

    \begin{align}
        \int_0^{2\pi}(\rho^\prime(\theta)\sin^2(n\theta)&-\rho(\theta)\cos(n\theta)\sin(n\theta))d\theta\\
        =&\frac{1}{2}\int_0^{2\pi}\Big ( (1-\cos(2n\theta))(\sum_{j=0}^\infty ja_j\cos(j\theta) - jb_j\sin(j\theta)) \nonumber \\
        &-\sin(2n\theta)(\sum_{j=0}^\infty a_j\sin(j\theta) + b_j\cos(j\theta))\Big )d\theta \nonumber \\
        =&-n\pi a_{2n} - \frac{1}{2}\pi a_{2n}
    \end{align}

    \begin{align}
        \int_0^{2\pi}(\rho^\prime(\theta)\cos(n\theta)\sin(n\theta)&+\rho(\theta)\sin^2(n\theta))d\theta\\
        =&\frac{1}{2}\int_0^{2\pi}\Big ( \sin(2n\theta)(\sum_{j=0}^\infty ja_j\cos(j\theta) - jb_j\sin(j\theta)) \nonumber \\
        &+(1-\cos(2n\theta))(\sum_{j=0}^\infty a_j\sin(j\theta) + b_j\cos(j\theta))\Big )d\theta \nonumber \\
        =&-n\pi b_{2n}-\frac{1}{2}\pi b_{2n} + \pi b_0\text{ . } 
    \end{align}

Substituting back into our entries of $M_n^{(1)}$ gives
\begin{align}
    m_{n,\,(1,1)}^{(1)} &= -(n^2 + \frac{n}{2})\sqrt{\pi}b_{2n}\\
    m_{n,\,(1,2)}^{(1)} &= -(n^2 + \frac{n}{2})\sqrt{\pi}a_{2n}\\
    m_{n,\,(2,1)}^{(1)} &= -(n^2 + \frac{n}{2})\sqrt{\pi}a_{2n}\\
    m_{n,\,(2,2)}^{(1)} &= (n^2 + \frac{n}{2})\sqrt{\pi}b_{2n} \text{ . } 
\end{align}
So, overall:
\begin{align}
    M_n^{(1)} =  \begin{bmatrix}
        -(n^2 + \frac{n}{2})\sqrt{\pi}b_{2n} & -(n^2 + \frac{n}{2})\sqrt{\pi}a_{2n}\\
        -(n^2 + \frac{n}{2})\sqrt{\pi}a_{2n} & (n^2 + \frac{n}{2})\sqrt{\pi}b_{2n}
    \end{bmatrix}\text{ . } 
\end{align}
From the formula—$\text{Egn}(A) = \frac{\text{tr}(A)\pm\sqrt{\text{tr}^2(A) - 4\det(A)}}{2}$)—for the eigenvalues of a 2x2 matrix, we thus see that
\begin{align}
    \lambda_n^{(1)} &=\frac{\text{tr}(M_n^{(1)})\pm\sqrt{\text{tr}(M_n^{(1)})^2 - 4\det(M_n^{(1)})}}{2}\\
    &= \pm\sqrt{-\det(M_n^{(1)})} \nonumber\\
    &= \pm\sqrt{(n^2 + \frac{n}{2})^2\pi(a_{2n}^2 + b_{2n}^2)} \nonumber\\
    \label{lambdan1}
    &= \pm(n^2 + \frac{n}{2})\sqrt{\pi(a_{2n}^2 + b_{2n}^2)}\text{ . } 
\end{align}
These results align with and generalize the results obtained in \cite{viator2018}, namely Theorem 3.1 and the calculations below Theorem 4.3.

\pagebreak

\section{Second-order behavior of Steklov eigenvalues} 
\label{sec:asymptotic-sec} 
This section is dedicated to the proof of Theorem \ref{thm:noballs}.
\subsection{Integral Constants.}
\label{integralconstants}
{\allowdisplaybreaks
\begin{multicols}{2}[We define the following constants in terms of integrals involving $\rho$ and trigonometric functions from $0$ to $2\pi$ so we may reference them later:] 
\begin{align*}
    \mathcal{A}_n &=\frac{1}{\sqrt{\pi}}\int_0^{2\pi}\rho^2(\theta)\cos^2(n\theta)d\theta\\
    \mathcal{B}_n &=\frac{1}{\sqrt{\pi}}\int_0^{2\pi}\rho(\theta)\cos^2(n\theta)d\theta\\
    \mathcal{C}_n &=\frac{1}{\sqrt{\pi}}\int_0^{2\pi}{\rho^\prime}^2(\theta)\cos^2(n\theta)d\theta\\
    \mathcal{D}_n &=\frac{1}{\sqrt{\pi}}\int_0^{2\pi}\rho(\theta)\rho^\prime(\theta)\sin(n\theta)\cos(n\theta)d\theta\\
    \mathcal{E}_n &=\frac{1}{\sqrt{\pi}}\int_0^{2\pi}\rho^\prime(\theta)\sin(n\theta)\cos(n\theta)d\theta\\
    \mathcal{F}_n &=\frac{1}{\sqrt{\pi}}\int_0^{2\pi}\rho^2(\theta)\sin(n\theta)\cos(n\theta)d\theta\\
    \mathcal{G}_n &=\frac{1}{\sqrt{\pi}}\int_0^{2\pi}\rho(\theta)\sin(n\theta)\cos(n\theta)d\theta\\
    \mathcal{H}_n &=\frac{1}{\sqrt{\pi}}\int_0^{2\pi}{\rho^\prime}^2(\theta)\sin(n\theta)\cos(n\theta)d\theta\\
    \mathcal{I}_n &=\frac{1}{\sqrt{\pi}}\int_0^{2\pi}\rho(\theta)\rho^\prime(\theta)\cos^2(n\theta)d\theta\\
    \mathcal{J}_n &=\frac{1}{\sqrt{\pi}}\int_0^{2\pi}\rho^\prime(\theta)\cos^2(n\theta)d\theta\\
    \mathcal{K}_{n,k} &=\frac{1}{\sqrt{\pi}}\int_0^{2\pi} \rho^\prime(\theta)\sin(k\theta)\cos(n\theta)d\theta\\
    \mathcal{L}_{n,k} &=\frac{1}{\sqrt{\pi}}\int_0^{2\pi} \rho(\theta)\cos(k\theta)\cos(n\theta)d\theta\\
    \mathcal{M}_{n,k} &=\frac{1}{\sqrt{\pi}}\int_0^{2\pi} \rho^\prime(\theta)\cos(k\theta)\cos(n\theta)d\theta\\
    \mathcal{N}_{n,k} &=\frac{1}{\sqrt{\pi}}\int_0^{2\pi} \rho(\theta)\sin(k\theta)\cos(n\theta)d\theta\\
    \mathcal{O}_n &=\frac{1}{\sqrt{\pi}}\int_0^{2\pi}\rho(\theta)\rho^\prime(\theta)\sin^2(n\theta)d\theta\\
    \mathcal{P}_n &=\frac{1}{\sqrt{\pi}}\int_0^{2\pi}\rho^\prime(\theta)\sin^2(n\theta)d\theta\\
    \mathcal{Q}_n &=\frac{1}{\sqrt{\pi}}\int_0^{2\pi}\rho^2(\theta)\sin^2(n\theta)d\theta\\
    \mathcal{R}_n &=\frac{1}{\sqrt{\pi}}\int_0^{2\pi}\rho(\theta)\sin^2(n\theta)d\theta\\
    \mathcal{S}_n &=\frac{1}{\sqrt{\pi}}\int_0^{2\pi}{\rho^\prime}^2(\theta)\sin^2(n\theta)d\theta\\
    \mathcal{T}_{n,k} &=\frac{1}{\sqrt{\pi}}\int_0^{2\pi} \rho^\prime(\theta)\sin(k\theta)\sin(n\theta)d\theta\\
    \mathcal{U}_{n,k} &=\frac{1}{\sqrt{\pi}}\int_0^{2\pi} \rho(\theta)\cos(k\theta)\sin(n\theta)d\theta\\
    \mathcal{V}_{n,k} &=\frac{1}{\sqrt{\pi}}\int_0^{2\pi} \rho^\prime(\theta)\cos(k\theta)\sin(n\theta)d\theta\\
    \mathcal{W}_{n,k} &=\frac{1}{\sqrt{\pi}}\int_0^{2\pi} \rho(\theta)\sin(k\theta)\sin(n\theta)d\theta\text{ . } 
\end{align*}
\end{multicols}}

The computation of these constants in terms of the Fourier coefficients is saved for Appendix \ref{computingintegrals}, but follows the techniques in \ref{computinglambdan1}.

\subsection{Eigenfunction Information at Order-$\vareps$.}
We recall \eqref{orderepscollapsed}. As in \ref{formingmn1} we again generate two equations using "Fourier's Trick," but by extracting different frequencies. We multiply \eqref{orderepscollapsed} by $\cos(m\theta)$ for any $m \in \N$ where $m\neq n$ and integrate from $0$ to $2\pi$ to obtain
\begin{align}
    0 =& \frac{1}{\sqrt{\pi^{n-1}}}\Big[-n\int_0^{2\pi}\rho(\theta)\cos(n\theta)\cos(m\theta)d\theta + n\int_0^{2\pi}\rho^\prime(\theta)\sin(n\theta)\cos(m\theta)d\theta\Big]\alpha_n\\
    &+ \frac{1}{\sqrt{\pi^{n-1}}}\Big[-n\int_0^{2\pi}\rho(\theta)\sin(n\theta)\cos(m\theta)d\theta - n\int_0^{2\pi}\rho^\prime(\theta)\cos(n\theta)\cos(m\theta)d\theta\Big]\gamma_n\nonumber\\
    &+\frac{\pi}{\sqrt{\pi^{m-1}}}(m-n)\beta_m\nonumber\text{ . } 
\end{align}
Re-arranging and recalling \ref{integralconstants}:
\begin{align}
    \label{betam}
    \beta_m=& \frac{\sqrt{\pi^{m-n-1}}}{n-m}n\Big[(-\mathcal{L}_{n,m} + \mathcal{V}_{n,m})\alpha_n + (-\mathcal{U}_{n,m} -\mathcal{M}_{n,m})\gamma_n\Big]\text{ . } 
\end{align}
We repeat the above process except multiplying by $\sin(m\theta)$:
\begin{align}
    0 =& \frac{1}{\sqrt{\pi^{n-1}}}\Big[-n\int_0^{2\pi}\rho(\theta)\cos(n\theta)\sin(m\theta)d\theta + n\int_0^{2\pi}\rho^\prime(\theta)\sin(n\theta)\sin(m\theta)d\theta\Big]\alpha_n\\
    &+ \frac{1}{\sqrt{\pi^{n-1}}}\Big[-n\int_0^{2\pi}\rho(\theta)\sin(n\theta)\sin(m\theta)d\theta - n\int_0^{2\pi}\rho^\prime(\theta)\cos(n\theta)\sin(m\theta)d\theta\Big]\gamma_n\nonumber\\
    &+\frac{\pi}{\sqrt{\pi^{m-1}}}(m-n)\mu_m\nonumber\text{ . } 
\end{align}
Thus,
\begin{align}
    \label{mum}
    \mu_m=& \frac{\sqrt{\pi^{m-n-1}}}{n-m}n\Big[(-\mathcal{N}_{n,m} + \mathcal{T}_{n,m})\alpha_n + (-\mathcal{W}_{n,m} -\mathcal{K}_{n,m})\gamma_n\Big]
\end{align}

\subsection{Forming $M_n^{(2)}$}
We equate the order-$\vareps^2$ terms of \eqref{dnuexpanded} and \eqref{lambdauexpanded}:
\begin{align}
    \label{secondordersum}
    &\sum_{k=0}^\infty \frac{k}{\sqrt{\pi^{k-1}}}\Big[\frac{k-1}{2}((k-2)\rho^2(\theta) + kb_0^2-\frac{1}{2}\sum_{j=1}^\infty[a_j^2 + b_j^2]\\
     &-2(k-1)b_0\rho(\theta))(\delta_{k,n}\alpha_n\cos(k\theta) + \delta_{k,n}\gamma_n\sin(k\theta)) -\frac{1}{2}{\rho^\prime}^2(\theta)\delta_{k,n}\alpha_n\cos(k\theta)\nonumber\\
     &-\frac{1}{2}{\rho^\prime}^2(\theta)\delta_{k,n}\gamma_n\sin(k\theta) + \rho^\prime(\theta)\beta_k\sin(k\theta)-\rho^\prime(\theta)\mu_k\cos(k\theta)-\rho(\theta)\rho^\prime(\theta)\delta_{k,n}\alpha_n\sin(k\theta)\nonumber\\
     &+\rho(\theta)\rho^\prime(\theta)\delta_{k,n}\gamma_n\cos(k\theta)+ \sigma_k\cos(k\theta) + \eta_k\sin(k\theta) + (k-1)(\rho(\theta)-b_0)(\beta_k\cos(k\theta)\nonumber\\
     &+\mu_k\sin(k\theta) + \rho^\prime(\theta)\delta_{k,n}\alpha_n\sin(k\theta) - \rho^\prime(\theta)\delta_{k,n}\gamma_n\cos(k\theta))\Big]\nonumber\\
     =&\sum_{k=0}^\infty\frac{1}{\sqrt{\pi^k}}\Big[n\sqrt{\pi}(\sigma_k\cos(k\theta)+\eta_k\sin(k\theta)) + \lambda_n^{(1)}(\beta_k\cos(k\theta)+\mu_k\sin(k\theta))\nonumber\\
    &+ \lambda_n^{(2)}(\delta_{k,n}\alpha_n\cos(k\theta)+\delta_{k,n}\gamma_n\sin(k\theta)) + n\sqrt{\pi}\frac{k}{2}((k-1)\rho^2(\theta)+(k+1)b_0^2\nonumber\\
    &-\frac{1}{2}\sum_{j=1}^\infty[a_j^2+b_j^2]-2kb_0\rho(\theta))(\delta_{k,n}\alpha_n\cos(k\theta)+\delta_{k,n}\gamma_n\sin(k\theta))+k(\rho(\theta)-b_0)\nonumber\\
    &\cdot(n\sqrt{\pi}(\beta_k\cos(k\theta)+\mu_k\sin(k\theta))+ \lambda_n^{(1)}(\delta_{k,n}\alpha_n\cos(k\theta)+\delta_{k,n}\gamma_n\sin(k\theta)))\Big]\nonumber\text{ . } 
\end{align}
We recall that $\delta_{k,n} = \begin{cases}
    0 & k\neq n\\
    1 & k=n
\end{cases}$ to pull terms out of sums where possible:
\begin{align}
    &\frac{1}{\sqrt{\pi^n}}\lambda_n^{(2)}(\alpha_n\cos(n\theta)+\gamma_n\sin(n\theta))\\
    =&\frac{n}{2\sqrt{\pi^{n-1}}}(n-1)\Big((n-2)\rho^2(\theta)+nb_0^2- \frac{1}{2}\sum_{j=1}^\infty [a_j^2 + b_j^2]\nonumber\\
    &-2(n-1)b_0\rho(\theta)\Big)(\alpha_n \cos(n\theta) + \gamma_n\sin(n\theta))-\frac{n}{2\sqrt{\pi^{n-1}}}{\rho^\prime}^2(\theta)(\alpha_n \cos(n\theta)\nonumber\\
    &+\gamma_n\sin(n\theta)) +\frac{n}{\sqrt{\pi^{n-1}}}\rho(\theta)\rho^\prime(\theta)(\gamma_n\cos(n\theta) -\alpha_n\sin(n\theta))\nonumber\\
    &+\frac{n}{\sqrt{\pi^{n-1}}}(n-1)(\rho(\theta)-b_0)\rho^\prime(\theta)(\alpha_n\sin(n\theta)-\gamma_n\cos(n\theta))\nonumber\\
    &-\frac{n^2}{2\sqrt{\pi^{n-1}}}\Big((n-1)\rho^2(\theta)+(n+1)b_0^2- \frac{1}{2}\sum_{j=1}^\infty [a_j^2 + b_j^2]\nonumber\\
    &-2nb_0\rho(\theta)\Big)(\alpha_n\cos(n\theta) + \gamma_n\sin(n\theta)) - \lambda_n^{(1)}\frac{n}{\sqrt{\pi^n}}(\rho(\theta)-b_0)(\alpha_n\cos(n\theta) + \gamma_n\sin(n\theta))\nonumber\\
    &+\sum_{k=0}^\infty\Big[\frac{k}{\sqrt{\pi^{k-1}}}\rho^\prime(\theta)(\beta_k\sin(k\theta) - \mu_k\cos(k\theta)) + \frac{k}{\sqrt{\pi^{k-1}}}(k-1)(\rho(\theta)-b_0)(\beta_k\cos(k\theta)\nonumber\\
    &+ \mu_k\sin(k\theta)) - \frac{k}{\sqrt{\pi^{k-1}}}n(\rho(\theta)-b_0)(\beta_k\cos(k\theta) + \mu_k\sin(k\theta)) - \frac{\lambda_n^{(1)}}{\sqrt{\pi^{k}}}(\beta_k\cos(k\theta)\nonumber\\
    &+ \mu_k\sin(k\theta)) + \frac{k}{\sqrt{\pi^{k-1}}}(\sigma_k\cos(k\theta) + \eta_k\sin(k\theta)) - \frac{n}{\sqrt{\pi^{k-1}}}(\sigma_k\cos(k\theta) + \eta_k\sin(k\theta))\Big]\nonumber\\[20pt]
    \label{secondordercollapsed}
    &\frac{1}{\sqrt{\pi}}\lambda_n^{(2)}(\alpha_n\cos(n\theta)+\gamma_n\sin(n\theta))\\
    =&\frac{1}{2}n(n-1)\Big((n-2)\rho^2(\theta)+nb_0^2 - \frac{1}{2}\sum_{j=1}^\infty [a_j^2 + b_j^2]\nonumber\\
    &-2(n-1)b_0\rho(\theta)\Big)(\alpha_n \cos(n\theta) + \gamma_n\sin(n\theta))-\frac{1}{2}n{\rho^\prime}^2(\theta)(\alpha_n \cos(n\theta)\nonumber\\
    &+\gamma_n\sin(n\theta)) +n\rho(\theta)\rho^\prime(\theta)(\gamma_n\cos(n\theta) -\alpha_n\sin(n\theta))\nonumber\\
    &+n(n-1)(\rho(\theta)-b_0)\rho^\prime(\theta)(\alpha_n\sin(n\theta)-\gamma_n\cos(n\theta))\nonumber\\
    &-\frac{1}{2}n^2\Big((n-1)\rho^2(\theta)+(n+1)b_0^2 - \frac{1}{2}\sum_{j=1}^\infty [a_j^2 + b_j^2]\nonumber\\
    &-2nb_0\rho(\theta)\Big)(\alpha_n\cos(n\theta) + \gamma_n\sin(n\theta)) - \lambda_n^{(1)}\frac{n}{\sqrt{\pi}}(\rho(\theta)-b_0)(\alpha_n\cos(n\theta) + \gamma_n\sin(n\theta))\nonumber\\
    &+\sum_{k=0}^\infty\sqrt{\pi^{n-k}}\Big[k\rho^\prime(\theta)(\beta_k\sin(k\theta) - \mu_k\cos(k\theta)) + k(k-1)(\rho(\theta)-b_0)(\beta_k\cos(k\theta)\nonumber\\
    &+ \mu_k\sin(k\theta)) - kn(\rho(\theta)-b_0)(\beta_k\cos(k\theta) + \mu_k\sin(k\theta)) - \frac{\lambda_n^{(1)}}{\sqrt{\pi}}(\beta_k\cos(k\theta)\nonumber\\
    &+ \mu_k\sin(k\theta)) + k(\sigma_k\cos(k\theta) + \eta_k\sin(k\theta)) - n(\sigma_k\cos(k\theta) + \eta_k\sin(k\theta))\Big]\nonumber\text{ . } 
\end{align}
We follow the same method as in \ref{formingmn1} to form two equations. Multiplying \eqref{secondordercollapsed} by $\cos(n\theta)$ and integrating from $0$ to $2\pi$ using our definitions in \ref{integralconstants} yields
\begin{align}
    \label{lambdan2alphainitial}
    \lambda_n^{(2)}\alpha_n =& \Big[-n(n-1)\mathcal{A}_n + n((2n-1)b_0-\frac{\lambda_n^{(1)}}{\sqrt{\pi}})\mathcal{B}_n - \frac{1}{2}n\mathcal{C}_n +n(n-2)\mathcal{D}_n\\
    &- n(n-1)b_0\mathcal{E}_n - n^2b_0^2\sqrt{\pi}+\frac{1}{4}n\sqrt{\pi}\sum_{j=1}^\infty[a_j^2 + b_j^2] +\lambda_n^{(1)}nb_0\Big]\alpha_n\nonumber\\
    &+\Big[-n(n-1)\mathcal{F}_n + n((2n-1)b_0-\frac{\lambda_n^{(1)}}{\sqrt{\pi}})\mathcal{G}_n - \frac{1}{2}n\mathcal{H}_n-n(n-2)\mathcal{I}_n\nonumber\\
    &+n(n-1)b_0\mathcal{J}_n\Big]\gamma_n\nonumber\\
    &+ \Big[n\mathcal{E}_n - n\mathcal{B}_n + nb_0\sqrt{\pi} -\lambda_n^{(1)}\Big]\beta_n+\Big[-n\mathcal{J}_n -n \mathcal{G}_n\Big]\mu_n\nonumber\\
    &+ \sum_{k\neq n}\sqrt{\pi^{n-k}}\Big[[k\mathcal{K}_{n,k}+k(k-n-1)\mathcal{L}_{n,k}]\beta_k+ [-k\mathcal{M}_{n,k}+k(k-n-1)\mathcal{N}_{n,k}]\mu_k\Big]\nonumber\text{ . } 
\end{align}
Repeating the same with $\sin(n\theta)$ yields
\begin{align}
    \label{lambdan2gammainitial}
    \lambda_n^{(2)}\gamma_n =& \Big[-n(n-1)\mathcal{F}_n + n((2n-1)b_0 -\frac{\lambda_n^{(1)}}{\sqrt{\pi}})\mathcal{G}_n - \frac{1}{2}n\mathcal{H}_n + n(n-2)\mathcal{O}_n\\
    &- n(n-1)b_0\mathcal{P}_n\Big]\alpha_n + \Big[-n(n-2)\mathcal{D}_n + n(n-1)b_0\mathcal{E}_n-n(n-1)\mathcal{Q}_n\nonumber\\
    &+ n((2n-1)b_0 -\frac{\lambda_n^{(1)}}{\sqrt{\pi}})\mathcal{R}_n - \frac{1}{2}n\mathcal{S}_n-n^2b_0^2\sqrt{\pi} + \frac{1}{4}n\sqrt{\pi}\sum_{j=1}^\infty[a_j^2 + b_j^2]+\lambda_n^{(1)}nb_0\Big]\gamma_n\nonumber\\
    &+ \Big[n\mathcal{P}_n-n\mathcal{G}_n\Big]\beta_n +\Big[-n\mathcal{E}_n-n\mathcal{R}_n + n\sqrt{\pi}b_0 -\lambda_n^{(1)}\Big]\mu_n\nonumber\\
    &+ \sum_{k\neq n}\sqrt{\pi^{n-k}}\Big[[k\mathcal{T}_{n,k}+k(k-n-1)\mathcal{U}_{n,k}]\beta_k + [-k\mathcal{V}_{n,k}+k(k-n-1)\mathcal{W}_{n,k}]\mu_k\Big]\nonumber\text{ . } 
\end{align}
We substitute \eqref{betam} and \eqref{mum} into \eqref{lambdan2alphainitial}:
\begin{align}
    \lambda_n^{(2)}\alpha_n =& \Big[-n(n-1)\mathcal{A}_n + n((2n-1)b_0-\frac{\lambda_n^{(1)}}{\sqrt{\pi}})\mathcal{B}_n - \frac{1}{2}n\mathcal{C}_n +n(n-2)\mathcal{D}_n\\
    &- n(n-1)b_0\mathcal{E}_n - n^2b_0^2\sqrt{\pi}+\frac{1}{4}n\sqrt{\pi}\sum_{j=1}^\infty[a_j^2 + b_j^2] +\lambda_n^{(1)}nb_0\Big]\alpha_n\nonumber\\
    &+\Big[-n(n-1)\mathcal{F}_n + n((2n-1)b_0-\frac{\lambda_n^{(1)}}{\sqrt{\pi}})\mathcal{G}_n - \frac{1}{2}n\mathcal{H}_n-n(n-2)\mathcal{I}_n\nonumber\\
    &+n(n-1)b_0\mathcal{J}_n\Big]\gamma_n\nonumber\\
    &+ \Big[n\mathcal{E}_n - n\mathcal{B}_n + nb_0\sqrt{\pi} -\lambda_n^{(1)}\Big]\beta_n+\Big[-n\mathcal{J}_n -n \mathcal{G}_n\Big]\mu_n\nonumber\\
    &+ \sum_{k\neq n}\frac{nk}{\sqrt{\pi}(n-k)}\Big[[\mathcal{K}_{n,k}+(k-n-1)\mathcal{L}_{n,k}][(-\mathcal{L}_{n,k} + \mathcal{V}_{n,k})\alpha_n\nonumber\\
    &+ (-\mathcal{U}_{n,k} -\mathcal{M}_{n,k})\gamma_n]+ [-\mathcal{M}_{n,k}+(k-n-1)\mathcal{N}_{n,k}][(-\mathcal{N}_{n,k} + \mathcal{T}_{n,k})\alpha_n\nonumber\\
    &+ (-\mathcal{W}_{n,k} -\mathcal{K}_{n,k})\gamma_n]\Big]\nonumber\text{ . } 
\end{align}
Grouping by $\alpha_n$, $\gamma_n$, $\beta_n$ and $\mu_n$ gives
\begin{align}
    \label{lambdan2alphageneral}
    \lambda_n^{(2)}\alpha_n =& \Big[-n(n-1)\mathcal{A}_n + n((2n-1)b_0-\frac{\lambda_n^{(1)}}{\sqrt{\pi}})\mathcal{B}_n - \frac{1}{2}n\mathcal{C}_n +n(n-2)\mathcal{D}_n\\
    &- n(n-1)b_0\mathcal{E}_n - n^2b_0^2\sqrt{\pi}+\frac{1}{4}n\sqrt{\pi}\sum_{j=1}^\infty[a_j^2 + b_j^2] +\lambda_n^{(1)}nb_0\nonumber\\
    &+ \sum_{k\neq n} \frac{nk}{\sqrt{\pi}(n-k)}[(\mathcal{K}_{n,k}+(k-n-1)\mathcal{L}_{n,k})(-\mathcal{L}_{n,k} + \mathcal{V}_{n,k})\nonumber\\
    &+ (-\mathcal{M}_{n,k}+(k-n-1)\mathcal{N}_{n,k})(-\mathcal{N}_{n,k} + \mathcal{T}_{n,k})]\Big]\alpha_n\nonumber\\
    &+\Big[-n(n-1)\mathcal{F}_n + n((2n-1)b_0-\frac{\lambda_n^{(1)}}{\sqrt{\pi}})\mathcal{G}_n - \frac{1}{2}n\mathcal{H}_n-n(n-2)\mathcal{I}_n\nonumber\\
    &+n(n-1)b_0\mathcal{J}_n + \sum_{k\neq n} \frac{nk}{\sqrt{\pi}(n-k)}[(\mathcal{K}_{n,k}+(k-n-1)\mathcal{L}_{n,k})(-\mathcal{U}_{n,k} - \mathcal{M}_{n,k})\nonumber\\
    &+ (-\mathcal{M}_{n,k}+(k-n-1)\mathcal{N}_{n,k})(-\mathcal{W}_{n,k} - \mathcal{K}_{n,k})]\Big]\gamma_n\nonumber\\
    &+ \Big[n\mathcal{E}_n - n\mathcal{B}_n + nb_0\sqrt{\pi} -\lambda_n^{(1)}\Big]\beta_n+\Big[-n\mathcal{J}_n -n \mathcal{G}_n\Big]\mu_n\nonumber\text{ . } 
\end{align}
We repeat this process with \eqref{lambdan2gammainitial}:
\begin{align}
    \lambda_n^{(2)}\gamma_n =& \Big[-n(n-1)\mathcal{F}_n + n((2n-1)b_0 -\frac{\lambda_n^{(1)}}{\sqrt{\pi}})\mathcal{G}_n - \frac{1}{2}n\mathcal{H}_n + n(n-2)\mathcal{O}_n\\
    &- n(n-1)b_0\mathcal{P}_n\Big]\alpha_n + \Big[-n(n-2)\mathcal{D}_n + n(n-1)b_0\mathcal{E}_n-n(n-1)\mathcal{Q}_n\nonumber\\
    &+ n((2n-1)b_0 -\frac{\lambda_n^{(1)}}{\sqrt{\pi}})\mathcal{R}_n - \frac{1}{2}n\mathcal{S}_n-n^2b_0^2\sqrt{\pi} + \frac{1}{4}n\sqrt{\pi}\sum_{j=1}^\infty[a_j^2 + b_j^2]+\lambda_n^{(1)}nb_0\Big]\gamma_n\nonumber\\
    &+ \Big[n\mathcal{P}_n-n\mathcal{G}_n\Big]\beta_n +\Big[-n\mathcal{E}_n-n\mathcal{R}_n + n\sqrt{\pi}b_0 -\lambda_n^{(1)}\Big]\mu_n\nonumber\\
    &+ \sum_{k\neq n}\frac{nk}{\sqrt{\pi}(n-k)}\Big[[\mathcal{T}_{n,k}+(k-n-1)\mathcal{U}_{n,k}][(-\mathcal{L}_{n,k} + \mathcal{V}_{n,k})\alpha_n\nonumber\\
    &+ (-\mathcal{U}_{n,k} -\mathcal{M}_{n,k})\gamma_n] + [-\mathcal{V}_{n,k}+(k-n-1)\mathcal{W}_{n,k}][(-\mathcal{N}_{n,k} + \mathcal{T}_{n,k})\alpha_n\nonumber\\
    &+ (-\mathcal{W}_{n,k} -\mathcal{K}_{n,k})\gamma_n]\Big]\nonumber\\[20pt]
    \label{lambdan2gammageneral}
    \lambda_n^{(2)}\gamma_n =& \Big[-n(n-1)\mathcal{F}_n + n((2n-1)b_0 -\frac{\lambda_n^{(1)}}{\sqrt{\pi}})\mathcal{G}_n - \frac{1}{2}n\mathcal{H}_n + n(n-2)\mathcal{O}_n\\
    &- n(n-1)b_0\mathcal{P}_n + \sum_{k\neq n} \frac{nk}{\sqrt{\pi}(n-k)}[(\mathcal{T}_{n,k}+(k-n-1)\mathcal{U}_{n,k})(-\mathcal{L}_{n,k} + \mathcal{V}_{n,k})\nonumber\\
    &+ (-\mathcal{V}_{n,k}+(k-n-1)\mathcal{W}_{n,k})(-\mathcal{N}_{n,k} + \mathcal{T}_{n,k})]\Big]\alpha_n\nonumber\\
    &+ \Big[-n(n-2)\mathcal{D}_n + n(n-1)b_0\mathcal{E}_n-n(n-1)\mathcal{Q}_n+ n((2n-1)b_0 -\frac{\lambda_n^{(1)}}{\sqrt{\pi}})\mathcal{R}_n\nonumber\\
    &- \frac{1}{2}n\mathcal{S}_n-n^2b_0^2\sqrt{\pi} + \frac{1}{4}n\sqrt{\pi}\sum_{j=1}^\infty[a_j^2 + b_j^2]+\lambda_n^{(1)}nb_0\nonumber\\
    &+ \sum_{k\neq n} \frac{nk}{\sqrt{\pi}(n-k)}[(\mathcal{T}_{n,k}+(k-n-1)\mathcal{U}_{n,k})(-\mathcal{U}_{n,k} - \mathcal{M}_{n,k})\nonumber\\
    &+ (-\mathcal{V}_{n,k}+(k-n-1)\mathcal{W}_{n,k})(-\mathcal{W}_{n,k} - \mathcal{K}_{n,k})]\Big]\gamma_n\nonumber\\
    &+ \Big[n\mathcal{P}_n-n\mathcal{G}_n\Big]\beta_n +\Big[-n\mathcal{E}_n-n\mathcal{R}_n + n\sqrt{\pi}b_0 -\lambda_n^{(1)}\Big]\mu_n\nonumber\text{ . } 
\end{align}

From here on, we will assume $\lambda_n^{(1)} = 0$. \eqref{lambdan1} shows that this assumption is equivalent to assuming $a_{2n} = b_{2n} = 0$ or that $\lambda_n$ does not split at order-$\vareps$. The values found in \ref{integralsvalue} now show $\mathcal{E}_n = \mathcal{J}_n = \mathcal{G}_n = \mathcal{P}_n = 0$ and $\mathcal{B}_n = \mathcal{R}_n = b_0\sqrt{\pi}$. Substituting into \eqref{lambdan2alphageneral} gives
\begin{align}
    \lambda_n^{(2)}\alpha_n =& \Big[-n(n-1)\mathcal{A}_n - \frac{1}{2}n\mathcal{C}_n +n(n-2)\mathcal{D}_n+ n(n-1)b_0^2\sqrt{\pi}+\frac{1}{4}n\sqrt{\pi}\sum_{j=1}^\infty[a_j^2 + b_j^2]\\
    &+ \sum_{k\neq n} \frac{nk}{\sqrt{\pi}(n-k)}[(\mathcal{K}_{n,k}+(k-n-1)\mathcal{L}_{n,k})(-\mathcal{L}_{n,k} + \mathcal{V}_{n,k})\nonumber\\
    &+ (-\mathcal{M}_{n,k}+(k-n-1)\mathcal{N}_{n,k})(-\mathcal{N}_{n,k} + \mathcal{T}_{n,k})]\Big]\alpha_n\nonumber\\
    &+\Big[-n(n-1)\mathcal{F}_n - \frac{1}{2}n\mathcal{H}_n-n(n-2)\mathcal{I}_n+ \sum_{k\neq n} \frac{nk}{\sqrt{\pi}(n-k)}[(\mathcal{K}_{n,k}\nonumber\\
    &+(k-n-1)\mathcal{L}_{n,k})(-\mathcal{U}_{n,k} - \mathcal{M}_{n,k})+ (-\mathcal{M}_{n,k}+(k-n-1)\mathcal{N}_{n,k})(-\mathcal{W}_{n,k}\nonumber\\
    &- \mathcal{K}_{n,k})]\Big]\gamma_n\nonumber\text{ . } 
\end{align}
Substituting into \eqref{lambdan2gammageneral} gives
\begin{align}
    \lambda_n^{(2)}\gamma_n =& \Big[-n(n-1)\mathcal{F}_n - \frac{1}{2}n\mathcal{H}_n + n(n-2)\mathcal{O}_n+ \sum_{k\neq n} \frac{nk}{\sqrt{\pi}(n-k)}[(\mathcal{T}_{n,k}\\
    &+(k-n-1)\mathcal{U}_{n,k})(-\mathcal{L}_{n,k} + \mathcal{V}_{n,k})+ (-\mathcal{V}_{n,k}+(k-n-1)\mathcal{W}_{n,k})(-\mathcal{N}_{n,k} + \mathcal{T}_{n,k})]\Big]\alpha_n\nonumber\\
    &+ \Big[-n(n-2)\mathcal{D}_n -n(n-1)\mathcal{Q}_n+- \frac{1}{2}n\mathcal{S}_n+n(n-1)b_0^2\sqrt{\pi} + \frac{1}{4}n\sqrt{\pi}\sum_{j=1}^\infty[a_j^2 + b_j^2]\nonumber\\
    &+ \sum_{k\neq n} \frac{nk}{\sqrt{\pi}(n-k)}[(\mathcal{T}_{n,k}+(k-n-1)\mathcal{U}_{n,k})(-\mathcal{U}_{n,k} - \mathcal{M}_{n,k})+ (-\mathcal{V}_{n,k}\nonumber\\
    &+(k-n-1)\mathcal{W}_{n,k})(-\mathcal{W}_{n,k} - \mathcal{K}_{n,k})]\Big]\gamma_n\nonumber\text{ . } 
\end{align}
So we see
\begin{align}
\label{lambdan2eigenvalue}
    \lambda_n^{(2)} \begin{bmatrix}
        \alpha_n\\
        \gamma_n
    \end{bmatrix}&= M_n^{(2)}\begin{bmatrix}
        \alpha_n\\
        \gamma_n
    \end{bmatrix}\text{ . } 
\end{align}
i.e. $\lambda_n^{(2)}$ must be an eigenvalue of the matrix $M_n^{(2)} = \begin{bmatrix}
    m_{n,\,(1,1)}^{(2)} & m_{n,\,(1,2)}^{(2)}\\
    m_{n,\,(2,1)}^{(2)} & m_{n,\,(2,2)}^{(2)}
\end{bmatrix}$ given by
\begin{align}
    \label{mn211}
    m_{n,\,(1,1)}^{(2)} =& -n(n-1)\mathcal{A}_n - \frac{1}{2}n\mathcal{C}_n +n(n-2)\mathcal{D}_n+ n(n-1)b_0^2\sqrt{\pi} \\ &+\frac{1}{4}n\sqrt{\pi}\sum_{j=1}^\infty[a_j^2 + b_j^2] + \sum_{k\neq n} \frac{nk}{\sqrt{\pi}(n-k)}\big[(\mathcal{K}_{n,k}+(k-n-1)\mathcal{L}_{n,k})(-\mathcal{L}_{n,k} + \mathcal{V}_{n,k}) \nonumber\\
    &+ (-\mathcal{M}_{n,k}+(k-n-1)\mathcal{N}_{n,k})(-\mathcal{N}_{n,k} + \mathcal{T}_{n,k})\big]\nonumber\\
    \label{mn212}
    m_{n,\,(1,2)}^{(2)} =& -n(n-1)\mathcal{F}_n - \frac{1}{2}n\mathcal{H}_n-n(n-2)\mathcal{I}_n \\ &+ \sum_{k\neq n} \frac{nk}{\sqrt{\pi}(n-k)}\big[(\mathcal{K}_{n,k} +(k-n-1)\mathcal{L}_{n,k})(-\mathcal{U}_{n,k} - \mathcal{M}_{n,k}) \nonumber \\ &+ (-\mathcal{M}_{n,k}+(k-n-1)\mathcal{N}_{n,k})(-\mathcal{W}_{n,k}- \mathcal{K}_{n,k})\big]\nonumber\\
    \label{mn221}
     m_{n,\,(2,1)}^{(2)} =&-n(n-1)\mathcal{F}_n - \frac{1}{2}n\mathcal{H}_n + n(n-2)\mathcal{O}_n+ \sum_{k\neq n} \frac{nk}{\sqrt{\pi}(n-k)}\big[(\mathcal{T}_{n,k}\\
    &+(k-n-1)\mathcal{U}_{n,k})(-\mathcal{L}_{n,k} + \mathcal{V}_{n,k})+ (-\mathcal{V}_{n,k}+(k-n-1)\mathcal{W}_{n,k})(-\mathcal{N}_{n,k} + \mathcal{T}_{n,k})\big]\nonumber\\
    \label{mn222}
     m_{n,\,(2,2)}^{(2)} =&-n(n-2)\mathcal{D}_n -n(n-1)\mathcal{Q}_n- \frac{1}{2}n\mathcal{S}_n+n(n-1)b_0^2\sqrt{\pi} + \frac{1}{4}n\sqrt{\pi}\sum_{j=1}^\infty[a_j^2 + b_j^2]\\
    &+ \sum_{k\neq n} \frac{nk}{\sqrt{\pi}(n-k)}\big[(\mathcal{T}_{n,k}+(k-n-1)\mathcal{U}_{n,k})(-\mathcal{U}_{n,k} - \mathcal{M}_{n,k})+ (-\mathcal{V}_{n,k}\nonumber\\
    &+(k-n-1)\mathcal{W}_{n,k})(-\mathcal{W}_{n,k} - \mathcal{K}_{n,k})\big]\nonumber\text{ . } 
\end{align}

\subsection{$M_n^{(2)}$ is Symmetric}
\label{mn2symmetric}
We demonstrate the $M_n^{(2)}$ is symmetric by showing that $m_{n,\,(1,2)}^{(2)}-m_{n,\,(2,1)}^{(2)} =0$. We substitute in \eqref{mn212} and \eqref{mn221}:
\begin{align}
    &m_{n,\,(1,2)}^{(2)} - m_{n,\,(2,1)}^{(2)}\\
    =&-n(n-1)\mathcal{F}_n - \frac{1}{2}n\mathcal{H}_n-n(n-2)\mathcal{I}_n+ \sum_{k\neq n} \frac{nk}{\sqrt{\pi}(n-k)}\Big[(\mathcal{K}_{n,k}\nonumber\\
    &+(k-n-1)\mathcal{L}_{n,k})(-\mathcal{U}_{n,k} - \mathcal{M}_{n,k})+ (-\mathcal{M}_{n,k}+(k-n-1)\mathcal{N}_{n,k})(-\mathcal{W}_{n,k}- \mathcal{K}_{n,k})\Big]\nonumber\\
    &-\Big(-n(n-1)\mathcal{F}_n - \frac{1}{2}n\mathcal{H}_n + n(n-2)\mathcal{O}_n+ \sum_{k\neq n} \frac{nk}{\sqrt{\pi}(n-k)}\Big[(\mathcal{T}_{n,k\nonumber}\\
    &+(k-n-1)\mathcal{U}_{n,k})(-\mathcal{L}_{n,k} + \mathcal{V}_{n,k})+ (-\mathcal{V}_{n,k}+(k-n-1)\mathcal{W}_{n,k})(-\mathcal{N}_{n,k} + \mathcal{T}_{n,k})\Big]\Big)\nonumber\\
    \label{symmetricsubtraction}
    =&-n(n-2)(\mathcal{I}_n+\mathcal{O}_n) +\sum_{k\neq n}\frac{nk}{\sqrt{\pi}(n-k)}\Big[(\mathcal{K}_{n,k}+(k-n-1)\mathcal{L}_{n,k})(-\mathcal{U}_{n,k} - \mathcal{M}_{n,k})\\
    &+ (-\mathcal{M}_{n,k}+(k-n-1)\mathcal{N}_{n,k})(-\mathcal{W}_{n,k}- \mathcal{K}_{n,k}) -(\mathcal{T}_{n,k}+(k-n-1)\mathcal{U}_{n,k})(-\mathcal{L}_{n,k}\nonumber\\
    &+ \mathcal{V}_{n,k})- (-\mathcal{V}_{n,k}+(k-n-1)\mathcal{W}_{n,k})(-\mathcal{N}_{n,k} + \mathcal{T}_{n,k})\Big]\nonumber\text{ . } 
\end{align}
Notice, applying the formulas in \ref{integralsvalue}, we have
\begin{align}
    \mathcal{I}_n + \mathcal{O}_n =& \frac{\sqrt{\pi}}{4}\Big[\sum_{j=0}^{2n}(2n-j)(a_{j}b_{2n-j} + b_ja_{2n-j}) + 2n\sum_{j=n}^\infty(b_{j-n}a_{j+n}-a_{j-n}b_{j+n})\\
    &-\sum_{j=0}^{2n}(2n-j)(a_{j}b_{2n-j} + b_ja_{2n-j}) - 2n\sum_{j=n}^\infty(b_{j-n}a_{j+n}-a_{j-n}b_{j+n})\Big]\nonumber\\
    \label{In+On}
    =&\frac{\sqrt{\pi}}{4}[0] = 0\text{ . } 
\end{align}
Substituting \eqref{In+On} into \eqref{symmetricsubtraction} gives
\begin{align}
\label{symmetricreduced}
    &m_{n,\,(1,2)}^{(2)} - m_{n,\,(2,1)}^{(2)}\\
    =&\sum_{k\neq n}\frac{nk}{\sqrt{\pi}(n-k)}\Big[(\mathcal{K}_{n,k}+(k-n-1)\mathcal{L}_{n,k})(-\mathcal{U}_{n,k} - \mathcal{M}_{n,k})+ (-\mathcal{M}_{n,k}\nonumber\\
    &+(k-n-1)\mathcal{N}_{n,k})(-\mathcal{W}_{n,k}- \mathcal{K}_{n,k}) -(\mathcal{T}_{n,k}+(k-n-1)\mathcal{U}_{n,k})(-\mathcal{L}_{n,k}+ \mathcal{V}_{n,k})\nonumber\\
    &- (-\mathcal{V}_{n,k}+(k-n-1)\mathcal{W}_{n,k})(-\mathcal{N}_{n,k} + \mathcal{T}_{n,k})\Big]\nonumber\text{ . } 
\end{align}
We consider each term of this remaining sum, substituting the formulas in \ref{integralsvalue} when necessary. When $k\neq n$ we have
\begin{align}
    &(\mathcal{K}_{n,k}+(k-n-1)\mathcal{L}_{n,k})(-\mathcal{U}_{n,k} - \mathcal{M}_{n,k})\\
    &+ (-\mathcal{M}_{n,k}+(k-n-1)\mathcal{N}_{n,k})(-\mathcal{W}_{n,k}- \mathcal{K}_{n,k})\nonumber\\
    &-(\mathcal{T}_{n,k}+(k-n-1)\mathcal{U}_{n,k})(-\mathcal{L}_{n,k}+ \mathcal{V}_{n,k})\nonumber\\
    &- (-\mathcal{V}_{n,k}+(k-n-1)\mathcal{W}_{n,k})(-\mathcal{N}_{n,k} + \mathcal{T}_{n,k})\nonumber\\
    =&-\mathcal{K}_{n,k}\mathcal{U}_{n,k} - (k-n-1)\mathcal{L}_{n,k}\mathcal{U}_{n,k} -(k-n-1)\mathcal{L}_{n,k}\mathcal{M}_{n,k}+ \mathcal{M}_{n,k}\mathcal{W}_{n,k}\nonumber\\
    &- (k-n-1)\mathcal{N}_{n,k}\mathcal{W}_{n,k} - (k-n-1)\mathcal{N}_{n,k}\mathcal{K}_{n,k} +\mathcal{T}_{n,k}\mathcal{L}_{n,k}+ (k-n-1)\mathcal{L}_{n,k}\mathcal{U}_{n,k}\nonumber\\
    &- (k-n-1)\mathcal{U}_{n,k}\mathcal{V}_{n,k} - \mathcal{V}_{n,k}\mathcal{N}_{n,k}+(k-n-1)\mathcal{W}_{n,k}\mathcal{N}_{n,k} - (k-n-1)\mathcal{W}_{n,k}\mathcal{T}_{n,k}\nonumber\\
    =&\frac{\pi}{4}[(0) + (k-n-1)(0) + (k-n-1)(0)]=0\text{ . } 
\end{align}
Thus \eqref{symmetricreduced} becomes
\begin{align}
    m_{n,\,(1,2)}^{(2)} - m_{n,\,(2,1)}^{(2)} =\sum_{k\neq n}\frac{nk}{\sqrt{\pi}(n-k)}[0]=0\text{ . } 
\end{align}
So we see that $M_n^{(2)}$ is symmetric.
\subsection{Evaluation of $M_n^{(2)}$ for Choice of $\rho$}
\label{evaluationmn2}
We evaluate $M_n^{(2)}$ in the specific case where $\rho(\theta) = \cos((n+\lceil\frac{n}{2}\rceil)\theta),\ \forall n > 1$.
\subsubsection{Even $n$}
When $n$ is even we see $\rho(\theta) = \cos((n+\frac{n}{2})\theta)$. Clearly $a_{2n} = b_{2n} = 0$, so $M_n^{(2)}$ governs the behavior of $\lambda_n^{(2)}$. Referencing \ref{integralsvalue} we see
\begin{align}
    \mathcal{D}_n &= \mathcal{F}_n = \mathcal{H}_n = \mathcal{I}_n =\mathcal{O}_n = 0\\
    \mathcal{C}_n &= \mathcal{S}_n = \frac{\sqrt{\pi}}{2}(n+\frac{n}{2})^2\\
    \mathcal{A}_n &= \mathcal{Q}_n = \frac{\sqrt{\pi}}{2}\text{ . } 
\end{align}
Furthermore, we notice see that every integral constant formula in \ref{integralsvalue} which involves both $k$ and $n$ depends only on $k+n$ Fourier coefficients and $k-n$ Fourier coefficients and will be zero if all of these coefficients are zero. Thus, recalling that $k\geq 0$, we see that these constants only survive when $k=\frac{n}{2}$ or $k=2n+ \frac{n}{2}$. In which case, \ref{integralsvalue} tells us:
\begin{align}
    \mathcal{M}_{n,\frac{n}{2}} &=\mathcal{N}_{n,\frac{n}{2}}= \mathcal{T}_{n,\frac{n}{2}} =\mathcal{U}_{n,\frac{n}{2}}= \mathcal{M}_{n,2n+ \frac{n}{2}} =\mathcal{N}_{n,2n+ \frac{n}{2}} =\mathcal{T}_{n,2n+ \frac{n}{2}} =\mathcal{U}_{n,2n+ \frac{n}{2}} =0\\
    \mathcal{K}_{n,\frac{n}{2}} &= \mathcal{V}_{n,\frac{n}{2}}=\mathcal{K}_{n,2n+ \frac{n}{2}} =-\frac{\sqrt{\pi}}{2}(n+\frac{n}{2})\\
    \mathcal{L}_{n,\frac{n}{2}} &=\mathcal{L}_{n,2n+ \frac{n}{2}}=\mathcal{W}_{n,2n+ \frac{n}{2}}=\frac{\sqrt{\pi}}{2}\\
    \mathcal{W}_{n,\frac{n}{2}} &=-\frac{\sqrt{\pi}}{2}\\
    \mathcal{V}_{n,2n+ \frac{n}{2}} &=\frac{\sqrt{\pi}}{2}(n+ \frac{n}{2})\text{ . }
\end{align}

We thus evaluate the entries of $M_n^{(2)}$ for this choice of $\rho$. \eqref{mn211} becomes
\begin{align}
    m_{n,\,(1,1)}^{(2)} =& -n(n-1)\frac{\sqrt{\pi}}{2} - \frac{1}{2}n\frac{\sqrt{\pi}}{2}(n+\frac{n}{2})^2 +\frac{1}{4}n\sqrt{\pi}\\
    &+ \frac{n(\frac{n}{2})}{\sqrt{\pi}(n-\frac{n}{2})}\Big[(-\frac{\sqrt{\pi}}{2}(n+\frac{n}{2})-(\frac{n}{2}+1)\frac{\sqrt{\pi}}{2})(-\frac{\sqrt{\pi}}{2} -\frac{\sqrt{\pi}}{2}(n+\frac{n}{2}))\Big]\nonumber\\
    &+ \frac{n(2n+\frac{n}{2})}{\sqrt{\pi}(-n-\frac{n}{2})}\Big[(-\frac{\sqrt{\pi}}{2}(n+\frac{n}{2}))+(n+\frac{n}{2}-1)\frac{\sqrt{\pi}}{2})(-\frac{\sqrt{\pi}}{2} + \frac{\sqrt{\pi}}{2}(n+\frac{n}{2}))\Big]\nonumber\\
    \label{mn211evaluatedeven}
    =&-n(n-1)\frac{\sqrt{\pi}}{2} - n(n+\frac{n}{2})^2\frac{\sqrt{\pi}}{4} +n\frac{\sqrt{\pi}}{4}\\
    &+\frac{n\sqrt{\pi}}{4}(2n+1)(n+\frac{n}{2}+1)+\frac{n(2n+\frac{n}{2})\sqrt{\pi}}{4(n+\frac{n}{2})}(n+\frac{n}{2}-1)\nonumber
\end{align}
and \eqref{mn222} becomes
\begin{align}
     m_{n,\,(2,2)}^{(2)} =&-n(n-1)\frac{\sqrt{\pi}}{2}- \frac{1}{2}n\frac{\sqrt{\pi}}{2}(n+\frac{n}{2})^2+\frac{1}{4}n\sqrt{\pi}\\
    &+ \frac{n\frac{n}{2}}{\sqrt{\pi}\frac{n}{2}}\Big[(\frac{\sqrt{\pi}}{2}(n+\frac{n}{2})+(\frac{n}{2}+1)\frac{\sqrt{\pi}}{2})(\frac{\sqrt{\pi}}{2} +\frac{\sqrt{\pi}}{2}(n+\frac{n}{2}))\Big]\nonumber\\
    &+ \frac{n(2n+\frac{n}{2})}{\sqrt{\pi}(-n-\frac{n}{2})}\Big[(-\frac{\sqrt{\pi}}{2}(n+\frac{n}{2})+(n+\frac{n}{2}-1)\frac{\sqrt{\pi}}{2})(-\frac{\sqrt{\pi}}{2} +\frac{\sqrt{\pi}}{2}(n+\frac{n}{2}))\Big]\nonumber\\
    \label{mn222evaluatedeven}
    =&-n(n-1)\frac{\sqrt{\pi}}{2}- n(n+\frac{n}{2})^2\frac{\sqrt{\pi}}{4}+n\frac{\sqrt{\pi}}{4}\\
    &+ \frac{n\sqrt{\pi}}{4}(2n+1)(n+\frac{n}{2}+1) + \frac{n(2n+\frac{n}{2})\sqrt{\pi}}{4(n+\frac{n}{2})}(n+\frac{n}{2}-1)\nonumber\text{ . }
\end{align}
Since we have shown in \ref{mn2symmetric} that $M_n^{(2)}$ is symmetric, we have that
\eqref{mn212} and \eqref{mn221} become
\begin{align}
    m_{n,\,(1,2)}^{(2)} =& m_{n,\,(2,1)}^{(2)}=\frac{n\frac{n}{2}}{\sqrt{\pi}\frac{n}{2}}\Big[(-\frac{\sqrt{\pi}}{2}(n+\frac{n}{2})+(-\frac{n}{2}-1)\frac{\sqrt{\pi}}{2})(0)+ (0)(\frac{\sqrt{\pi}}{2}+ \frac{\sqrt{\pi}}{2}(n+\frac{n}{2}))\Big]\\
    &+\frac{n(2n+\frac{n}{2})}{\sqrt{\pi}(n+\frac{n}{2})}\Big[(-\frac{\sqrt{\pi}}{2}(n+\frac{n}{2})+(n+\frac{n}{2}-1)\frac{\sqrt{\pi}}{2})(0)+ (0)(\frac{\sqrt{\pi}}{2}+\frac{\sqrt{\pi}}{2}(n+\frac{n}{2}))\Big]\nonumber\\
    =& 0\text{ . }
\end{align}

\subsubsection{Odd $n$}
When $n$ is odd we see that $\rho(\theta) = \cos((n+\frac{n+1}{2})\theta)$. Clearly, since $n>1$, $a_{2n}=b_{2n}=0$ so $M_n^{(2)}$ governs the behavior of $\lambda_n^{(2)}$. Referencing \ref{integralsvalue} we see
\begin{align}
    \mathcal{D}_n &= \mathcal{F}_n = \mathcal{H}_n = \mathcal{I}_n =\mathcal{O}_n = 0\\
    \mathcal{C}_n &= \mathcal{S}_n = \frac{\sqrt{\pi}}{2}(n+\frac{n+1}{2})^2\\
    \mathcal{A}_n &= \mathcal{Q}_n = \frac{\sqrt{\pi}}{2}\text{ . } 
\end{align}
By the same reasoning as for even $n$, integral coefficients involving both $k$ and $n$ survive only when $k=\frac{n+1}{2}$ and $k=2n+\frac{n+1}{2}$. In which case, \ref{integralsvalue} tells us:
\begin{align}
    \mathcal{M}_{n,\frac{n+1}{2}} &=\mathcal{N}_{n,\frac{n+1}{2}}= \mathcal{T}_{n,\frac{n+1}{2}} =\mathcal{U}_{n,\frac{n+1}{2}}= \mathcal{M}_{n,2n+ \frac{n+1}{2}}\\
    &=\mathcal{N}_{n,2n+ \frac{n+1}{2}} =\mathcal{T}_{n,2n+ \frac{n+1}{2}} =\mathcal{U}_{n,2n+ \frac{n+1}{2}} =0\nonumber\\
    \mathcal{K}_{n,\frac{n+1}{2}} &= \mathcal{V}_{n,\frac{n+1}{2}}=\mathcal{K}_{n,2n+ \frac{n+1}{2}} =-\frac{\sqrt{\pi}}{2}(n+\frac{n+1}{2})\\
    \mathcal{L}_{n,\frac{n+1}{2}} &=\mathcal{L}_{n,2n+ \frac{n+1}{2}}=\mathcal{W}_{n,2n+ \frac{n+1}{2}}=\frac{\sqrt{\pi}}{2}\\
    \mathcal{W}_{n,\frac{n+1}{2}} &=-\frac{\sqrt{\pi}}{2}\\
    \mathcal{V}_{n,2n+ \frac{n+1}{2}} &=\frac{\sqrt{\pi}}{2}(n+ \frac{n+1}{2})\text{ . }
\end{align}
We thus evaluate the entries of $M_n^{(2)}$ for this choice of $\rho$. \eqref{mn211} becomes
\begin{align}
    m_{n,\,(1,1)}^{(2)} =& -n(n-1)\frac{\sqrt{\pi}}{2} - \frac{1}{2}n\frac{\sqrt{\pi}}{2}(n+\frac{n+1}{2})^2 +\frac{1}{4}n\sqrt{\pi}\\
    &+ \frac{n\frac{n+1}{2}}{\sqrt{\pi}(n-\frac{n+1}{2})}\Big[(-\frac{\sqrt{\pi}}{2}(n+\frac{n+1}{2})+(\frac{n+1}{2}-n-1)\frac{\sqrt{\pi}}{2})\\
    &\cdot(-\frac{\sqrt{\pi}}{2} -\frac{\sqrt{\pi}}{2}(n+\frac{n+1}{2})\Big]\nonumber\\
    &+\frac{n(2n+\frac{n+1}{2})}{\sqrt{\pi}(-n-\frac{n+1}{2})}\Big[(-\frac{\sqrt{\pi}}{2}(n+\frac{n+1}{2})+(n+\frac{n+1}{2}-1)\frac{\sqrt{\pi}}{2})\nonumber\\
    &\cdot(-\frac{\sqrt{\pi}}{2} +\frac{\sqrt{\pi}}{2}(n+\frac{n+1}{2}))\Big]\nonumber\\
    \label{mn211evaluatedodd}
    =&-n(n-1)\frac{\sqrt{\pi}}{2} - n(n+\frac{n+1}{2})^2\frac{\sqrt{\pi}}{4} +n\frac{\sqrt{\pi}}{4}\\
    &+ \frac{n\frac{n+1}{2}\sqrt{\pi}}{4(n-\frac{n+1}{2})}(2n+1)(n+\frac{n+1}{2}+1)+\frac{n(2n+\frac{n+1}{2})\sqrt{\pi}}{4(n+\frac{n+1}{2})}(n+\frac{n+1}{2}-1)\nonumber
\end{align}
and \eqref{mn222} becomes
\begin{align}
    m_{n,\,(2,2)}^{(2)} =&-n(n-1)\frac{\sqrt{\pi}}{2}- \frac{1}{2}n\frac{\sqrt{\pi}}{2}(n+\frac{n+1}{2})^2  + \frac{1}{4}n\sqrt{\pi}+ \frac{n\frac{n+1}{2}}{\sqrt{\pi}(n-\frac{n+1}{2})}\Big[(\frac{\sqrt{\pi}}{2}(n+\frac{n+1}{2})\\
    &-(\frac{n+1}{2}-n-1)\frac{\sqrt{\pi}}{2})(\frac{\sqrt{\pi}}{2} +\frac{\sqrt{\pi}}{2}(n+\frac{n+1}{2}))\Big]\nonumber\\
    &+ \frac{n(2n+\frac{n+1}{2})}{\sqrt{\pi}(-n-\frac{n+1}{2})}\Big[(-\frac{\sqrt{\pi}}{2}(n+\frac{n+1}{2})+(n+\frac{n+1}{2}-1)\frac{\sqrt{\pi}}{2})\\
    &\cdot(-\frac{\sqrt{\pi}}{2} +\frac{\sqrt{\pi}}{2}(n+\frac{n+1}{2}))\Big]\nonumber\\
    \label{mn222evaluatedodd}
    =&-n(n-1)\frac{\sqrt{\pi}}{2}- n(n+\frac{n+1}{2})^2\frac{\sqrt{\pi}}{4}  + n\frac{\sqrt{\pi}}{4}\\
    &+ \frac{n\frac{n+1}{2}\sqrt{\pi}}{4(n-\frac{n+1}{2})}(2n+1)(n+\frac{n+1}{2}+1)+ \frac{n(2n+\frac{n+1}{2})\sqrt{\pi}}{4(n+\frac{n+1}{2})}(n+\frac{n+1}{2}-1)\text{ . }\nonumber
\end{align}
Since we have shown in \ref{mn2symmetric} that $M_n^{(2)}$ is symmetric, we have that \eqref{mn212} and \eqref{mn221} become
\begin{align}
    m_{n,\,(1,2)}^{(2)}=m_{n,\,(2,1)}^{(2)} =&\frac{n\frac{n+1}{2}}{\sqrt{\pi}(n-\frac{n+1}{2})}\Big[(0)(-\frac{\sqrt{\pi}}{2} -\frac{\sqrt{\pi}}{2}(n+\frac{n+1}{2}))\\
    &+(\frac{\sqrt{\pi}}{2}(n+\frac{n+1}{2})-(\frac{n+1}{2}-n-1)\frac{\sqrt{\pi}}{2})(0)\Big]\nonumber\\
    &+\frac{n(2n+\frac{n+1}{2})}{\sqrt{\pi}(-n-\frac{n+1}{2})}\Big[(0)(-\frac{\sqrt{\pi}}{2} + \frac{\sqrt{\pi}}{2}(n+\frac{n+1}{2}))\nonumber\\
    &+ (-\frac{\sqrt{\pi}}{2}(n+\frac{n+1}{2})+(n+\frac{n+1}{2}-1)\frac{\sqrt{\pi}}{2})(0)\Big]\nonumber\\
    =&0 \text{ . }
\end{align}
\subsection{Proof of Theorem \ref{thm:noballs}}
We first prove the following proposition:
\begin{proposition}
    \label{prop:choiceofrho}
    If $\rho(\theta) = \cos((n+\lceil\frac{n}{2}\rceil)\theta)$ then both branches of the $n$th Steklov Eigenvalue for $\Omega_\varepsilon$ increase in $\varepsilon$ for all $n> 1$.
\end{proposition}
\begin{proof}
    We see that for all $n>1$, $a_{2n} = b_{2n} = 0$. So, $\lambda_n^{(1)} = 0$ and $M_n^{(2)}$ governs the behavior of $\lambda_n^{(2)}$—as in \eqref{lambdan2eigenvalue}. Thus, it is enough to show that both branches of $\lambda_n^{(2)}$ are positive for this choice of $\rho$. We have shown in \ref{evaluationmn2} that $M_n^{(2)}$ is diagonal in this case; so we simply need to show that for every $n>1$ these diagonal entries are both positive.
    
    Suppose $n>1$ is even. Then the diagonal entries, \eqref{mn211evaluatedeven} and \eqref{mn222evaluatedeven}, are equal. So we must show that:
    \begin{align}
        \lambda_n^{(2)} =& -n(n-1)\frac{\sqrt{\pi}}{2} - n(n+\frac{n}{2})^2\frac{\sqrt{\pi}}{4} +n\frac{\sqrt{\pi}}{4}\\
        &+\frac{n\sqrt{\pi}}{4}(2n+1)(n+\frac{n}{2}+1)+\frac{n(2n+\frac{n}{2})\sqrt{\pi}}{4(n+\frac{n}{2})}(n+\frac{n}{2}-1)>0\text{ . }\nonumber
    \end{align}
    We see:
    \begin{align}
        \lambda_n^{(2)} =& -n(n-1)\frac{\sqrt{\pi}}{2} - n(n+\frac{n}{2})^2\frac{\sqrt{\pi}}{4} +n\frac{\sqrt{\pi}}{4}\\
        &+\frac{n\sqrt{\pi}}{4}(2n+1)(n+\frac{n}{2}+1)+\frac{n(2n+\frac{n}{2})\sqrt{\pi}}{4(n+\frac{n}{2})}(n+\frac{n}{2}-1)\nonumber\\
        &=\frac{\sqrt{\pi}}{4}\Big[\frac{3}{4}n^3+4n^2+\frac{7}{3}n\Big]>0 
    \end{align}
    since $n>0$, by assumption.
    
    Now, suppose $n>1$ is odd. Then the diagonal entries, \eqref{mn211evaluatedodd} and \eqref{mn222evaluatedodd}, are equal. So we must show that:
    \begin{align}
        \lambda_n^{(2)} =& -n(n-1)\frac{\sqrt{\pi}}{2}- n(n+\frac{n+1}{2})^2\frac{\sqrt{\pi}}{4}  + n\frac{\sqrt{\pi}}{4}\\
        &+ \frac{n\frac{n+1}{2}\sqrt{\pi}}{4(n-\frac{n+1}{2})}(2n+1)(n+\frac{n+1}{2}+1)+ \frac{n(2n+\frac{n+1}{2})\sqrt{\pi}}{4(n+\frac{n+1}{2})}(n+\frac{n+1}{2}-1)>0\text{ . }\nonumber
    \end{align}
    We see:
    \begin{align}
        \lambda_n^{(2)} =&  -n(n-1)\frac{\sqrt{\pi}}{2}- n(n+\frac{n+1}{2})^2\frac{\sqrt{\pi}}{4}  + n\frac{\sqrt{\pi}}{4}\\
        &+ \frac{n\frac{n+1}{2}\sqrt{\pi}}{4(n-\frac{n+1}{2})}(2n+1)(n+\frac{n+1}{2}+1)+ \frac{n(2n+\frac{n+1}{2})\sqrt{\pi}}{4(n+\frac{n+1}{2})}(n+\frac{n+1}{2}-1)\nonumber\\
        =& \frac{\sqrt{\pi}}{4}\Big[-2n^2 + 2n -2n^3 - n^2 -\frac{n^3+2n^2+n}{4} + n\nonumber\\
        &+ \frac{\frac{n^2+n}{2}}{n-\frac{n+1}{2}}(3n^2 + 3n + \frac{3n+3}{2})+ \frac{2n^2 + \frac{n^2 + n}{2}}{n+\frac{n+1}{2}}\frac{3n-1}{2}\Big]\nonumber\\
        =& \frac{\sqrt{\pi}}{4}\Big[-\frac{9n^3+14n^2-11n}{4}+ \frac{6n^4 + 15n^3 + 12n^2 + 3n}{2n-2} + \frac{15n^3 -2n^2 - n}{6n+2}\Big]\nonumber\\
        =& \frac{\sqrt{\pi}}{8}\Big[\frac{3n^4+25n^3+49n^2-5n}{2n-2} + \frac{15n^3 -2n^2 - n}{3n+1}\Big]\nonumber\\
        =& \frac{\sqrt{\pi}}{8}\Big[\frac{9n^5+108n^4 + 138n^3+36n^2-3n}{6n^2-4n-2}\Big] > 0
    \end{align}
    since $n>1$, by assumption.
\end{proof}
The proof of our main theorem follows immediately:
\begin{proof}[Proof of Theorem \ref{thm:noballs}]
    For any $n>1$, \ref{prop:choiceofrho} assures that there exists a domain of unit area, $\Omega_\vareps$, such that both branches of the $n$th the Steklov Eigenvalue on that domain are greater than on the disk of unit area, $\Omega_0$.
\end{proof}

\subsubsection*{Acknowledgements.} This project was funded by the Tarble Summer Research Fellowship at Swarthmore College. The authors would also like to acknowledge Yerim Kone and Amy Liu for their contributions during the Summer Research program at Swarthmore in 2023, on which this project is based. Finally, the authors would like to thank Braxton Osting at University of Utah for providing numerical resources for the Steklov problem on nearly-circular domains.

\printbibliography

\appendix
\section{Integral Constants}
\subsection{Computing Integral Constants}
\label{computingintegrals}
This appendix is dedicated to computing values of the constants defined in \ref{integralconstants} in terms of the Fourier coefficients of $\rho$. Proceeding as in \eqref{eq:rhoSquaredFourier}, we apply Cauchy's product for infinite sums, \eqref{cauchyproduct}, to the Fourier expansions for $\rho$ and $\rho^\prime$, \eqref{eq:rhoFourier} \eqref{eq:rhoPrimeFourier}, to calculate the expansions:
\begin{align}
    \rho^2(\theta) =& \frac{1}{2}\sum^{\infty}_{j=0}\sum^{j}_{i=0}[-a_ia_{j-i}(\cos(j\theta)-\cos((j-2i)\theta)) + a_ib_{j-i}(\sin(j\theta)-\sin((j-2i)\theta)) \\
    &+ b_ia_{j-i}(\sin(j\theta)+\sin((j-2i)\theta))+b_ib_{j-i}(\cos(j\theta)+\cos((j-2i)\theta))]\nonumber\\
    {\rho^\prime}^2(\theta) =&\frac{1}{2}\sum^{\infty}_{j=0}\sum^{j}_{i=0} i(j-i)[a_ia_{j-i}(\cos(j\theta) +\cos((j-2i)\theta))\\
    &- a_ib_{j-i}(\sin(j\theta)+\sin((j-2i)\theta))-b_ia_{j-i}(\sin(j\theta)-\sin((j-2i)\theta))\nonumber\\
    &- b_ib_{j-i}(\cos(j\theta)-\cos((j-2i)\theta))]\nonumber\\
    \rho(\theta)\rho^\prime(\theta) =& \frac{1}{2}\sum^{\infty}_{j=0}\sum^{j}_{i=0}(j-i)[a_ia_{j-i}(\sin(j\theta)-\sin((j-2i)\theta))\\
    &+a_ib_{j-i}(\cos(j\theta)-\cos((j-2i)\theta))+b_ia_{j-i}(\cos(j\theta)+\cos((j-2i)\theta))\nonumber\\
    &-b_ib_{j-i}(\sin(j\theta)+\sin((j-2i)\theta))]\nonumber
\end{align}
We again apply from the orthogonality of sine and cosine, \eqref{sinecosine} and \eqref{sinesinecosinecosine}, to first compute our integral constants involving only $n$ ($\mathcal{A}_n$ to $\mathcal{J}_n$ and $\mathcal{O}_n$ to $\mathcal{S}_n$):
\begin{align*}
    \addtocounter{equation}{1}\tag{\theequation}\mathcal{A}_n =&\frac{1}{\sqrt{\pi}}\int_0^{2\pi}\rho^2(\theta)\cos^2(n\theta)d\theta\\
    =&\frac{1}{2\sqrt{\pi}}\int_0^{2\pi}\rho^2(\theta)(1+\cos(2n\theta))d\theta\\
    =&\frac{1}{4\sqrt{\pi}}\int_0^{2\pi}(1+\cos(2n\theta))\sum^{\infty}_{j=0}\sum^{j}_{i=0}[-a_ia_{j-i}(\cos(j\theta)-\cos((j-2i)\theta))\\
    &+ a_ib_{j-i}(\sin(j\theta)-\sin((j-2i)\theta))+ b_ia_{j-i}(\sin(j\theta)+\sin((j-2i)\theta))\\
    &+b_ib_{j-i}(\cos(j\theta)+\cos((j-2i)\theta))]d\theta\\
    =&\frac{1}{4\sqrt{\pi}}\sum^{\infty}_{j=0}\sum^{j}_{i=0}\int_0^{2\pi}(1+\cos(2n\theta))[-a_ia_{j-i}(\cos(j\theta)-\cos((j-2i)\theta))\\
    &+ a_ib_{j-i}(\sin(j\theta)-\sin((j-2i)\theta))+ b_ia_{j-i}(\sin(j\theta)+\sin((j-2i)\theta))\\
    &+b_ib_{j-i}(\cos(j\theta)+\cos((j-2i)\theta))]d\theta\\
    =&\frac{\sqrt{\pi}}{4}[2\sum^{\infty}_{j=0}\sum^{j}_{i=0} (\delta_{j,0}(b_ib_{j-i}-a_ia_{j-i}) + \delta_{j-2i, 0}(a_ia_{j-i} + b_ib_{j-i}))\\
    &+ \sum^{\infty}_{j=0}\sum^{j}_{i=0}(\delta_{j,2n}(-a_ia_{j-i}+b_ib_{j-i}) + (\delta_{j-2i, 2n}+\delta_{2i-j, 2n})(a_ia_{j-i}+b_ib_{j-i}))]\\
    =&\frac{\sqrt{\pi}}{4}[2b_0^2 - 2a_0^2 + 2\sum_{\text{even }j}^\infty(a_{\frac{j}{2}}^2 + b_{\frac{j}{2}}^2)+ \sum_{i=0}^{2n}(b_ib_{2n-i}-a_ia_{2n-i})\\
    &+\sum_{\text{even } j \geq 2n}(a_{\frac{j}{2}-n}a_{\frac{j}{2}+n}+b_{\frac{j}{2}-n}b_{\frac{j}{2}+n})+\sum_{\text{even } j \geq 2n}(a_{n+\frac{j}{2}}a_{\frac{j}{2}-n}+b_{n+\frac{j}{2}}b_{\frac{j}{2}-n})]\\
    =&\frac{\sqrt{\pi}}{4}[2b_0^2 - 2a_0^2 + 2\sum_{j=0}^\infty(a_j^2 + b_j^2)\\
    &+ \sum_{i=0}^{2n}(b_ib_{2n-i}-a_ia_{2n-i}) +2\sum_{j=n}^\infty(b_{j-n}b_{j+n}+a_{j-n}a_{j+n})]\\
    \addtocounter{equation}{1}\tag{\theequation}=&\frac{\sqrt{\pi}}{4}[4b_0^2 + 2   \sum_{j=1}^\infty(a_j^2 + b_j^2) + \sum_{j=0}^{2n}(b_jb_{2n-j}-a_ja_{2n-j}) +2\sum_{j=n}^\infty(b_{j-n}b_{j+n}+a_{j-n}a_{j+n})]\\
    \addtocounter{equation}{1}\tag{\theequation}\mathcal{B}_n =&\frac{1}{\sqrt{\pi}}\int_0^{2\pi}\rho(\theta)\cos^2(n\theta)d\theta\\
    =&\frac{1}{2\sqrt{\pi}}\int_0^{2\pi}\rho(\theta)(1+\cos(2n\theta))d\theta\\
    =&\frac{1}{2\sqrt{\pi}}\int_0^{2\pi}(1+\cos(2n\theta))\sum_{j=0}^\infty a_j\sin(j\theta) + b_j\cos(j\theta)d\theta\\
    =&\frac{1}{2\sqrt{\pi}}\sum_{j=0}^\infty \int_0^{2\pi}(1+\cos(2n\theta))(a_j\sin(j\theta) + b_j\cos(j\theta))d\theta\\
    \addtocounter{equation}{1}\tag{\theequation}=&\frac{\sqrt{\pi}}{2}(2b_0 + b_{2n})\\
    \addtocounter{equation}{1}\tag{\theequation}\mathcal{C}_n =&\frac{1}{\sqrt{\pi}}\int_0^{2\pi}{\rho^\prime}^2(\theta)\cos^2(n\theta)d\theta\\
    =&\frac{1}{2\sqrt{\pi}}\int_0^{2\pi}{\rho^\prime}^2(\theta)(1+\cos(2n\theta))d\theta\\
    =&\frac{1}{4\sqrt{\pi}}\int_0^{2\pi}(1+\cos(2n\theta))\sum^{\infty}_{j=0}\sum^{j}_{i=0} i(j-i)[a_ia_{j-i}(\cos(j\theta) +\cos((j-2i)\theta))\\
    &- a_ib_{j-i}(\sin(j\theta)+\sin((j-2i)\theta))-b_ia_{j-i}(\sin(j\theta)-\sin((j-2i)\theta))\\
    &- b_ib_{j-i}(\cos(j\theta)-\cos((j-2i)\theta))] d\theta\\
    =&\frac{1}{4\sqrt{\pi}}\sum^{\infty}_{j=0}\sum^{j}_{i=0} \int_0^{2\pi}(1+\cos(2n\theta))i(j-i)[a_ia_{j-i}(\cos(j\theta) +\cos((j-2i)\theta))\\
    &- a_ib_{j-i}(\sin(j\theta)+\sin((j-2i)\theta))-b_ia_{j-i}(\sin(j\theta)-\sin((j-2i)\theta))\\
    &- b_ib_{j-i}(\cos(j\theta)-\cos((j-2i)\theta))] d\theta\\
    =&\frac{\sqrt{\pi}}{4} [2\sum^{\infty}_{j=0}\sum^{j}_{i=0} \delta_{j-2i, 0}i(j-i)(a_ia_{j-i} + b_ib_{j-i})\\
    &+ \sum^{\infty}_{j=0}\sum^{j}_{i=0}(\delta_{j, 2n}i(j-i)(a_ia_{j-i} - b_ib_{j-i})\\
    &+ (\delta_{j-2i, 2n}+ \delta_{2i-j, 2n})i(j-i)(a_ia_{j-i} + b_ib_{j-i}))]\\
    =&\frac{\sqrt{\pi}}{4} [2\sum_{\text{even }j}(\frac{j}{2})^2(a_{\frac{j}{2}}^2 + b_{\frac{j}{2}}^2) + \sum_{i=0}^{2n}i(2n-i)(a_ia_{2n-i} - b_ib_{2n-i})\\
    &+ \sum_{\text{even }j\geq 2n} (\frac{j}{2}-n)(\frac{j}{2}+n)(a_{\frac{j}{2}-n}a_{\frac{j}{2}+n} + b_{\frac{j}{2}-n}b_{\frac{j}{2}+n})\\
    &+\sum_{\text{even }j\geq 2n} (n+\frac{j}{2})(\frac{j}{2}-n)(a_{n+\frac{j}{2}}a_{\frac{j}{2}-n} + b_{n+\frac{j}{2}}b_{\frac{j}{2}-n})]\\
    =&\frac{\sqrt{\pi}}{4} [2\sum_{j=0}^\infty j^2(a_j^2 + b_j^2)+ \sum_{i=0}^{2n}i(2n-i)(a_ia_{2n-i} - b_ib_{2n-i})\\
    &+2\sum_{j=n}^\infty (j-n)(j+n)(a_{j-n}a_{j+n} + b_{j-n}b_{j+n})]\\
    \addtocounter{equation}{1}\tag{\theequation}=&\frac{\sqrt{\pi}}{4} [2\sum_{j=1}^\infty j^2(a_j^2 + b_j^2)+ \sum_{j=0}^{2n}j(2n-j)(a_ja_{2n-j} - b_jb_{2n-j})\\
    &+2\sum_{j=n}^\infty (j^2-n^2)(a_{j-n}a_{j+n} + b_{j-n}b_{j+n})]\\
    \addtocounter{equation}{1}\tag{\theequation}\mathcal{D}_n =&\frac{1}{\sqrt{\pi}}\int_0^{2\pi}\rho(\theta)\rho^\prime(\theta)\sin(n\theta)\cos(n\theta)d\theta\\
    =&\frac{1}{2\sqrt{\pi}}\int_0^{2\pi}\rho(\theta)\rho^\prime(\theta)\sin(2n\theta)d\theta\\
    =&\frac{1}{4\sqrt{\pi}}\int_0^{2\pi}\sin(2n\theta)\sum^{\infty}_{j=0}\sum^{j}_{i=0} (j-i)[a_ia_{j-i}(\sin(j\theta)-\sin((j-2i)\theta))\\
    &+a_ib_{j-i}(\cos(j\theta)-\cos((j-2i)\theta))+b_ia_{j-i}(\cos(j\theta)+\cos((j-2i)\theta))\\
    &-b_ib_{j-i}(\sin(j\theta)+\sin((j-2i)\theta))]d\theta\\
    =&\frac{1}{4\sqrt{\pi}}\sum^{\infty}_{j=0}\sum^{j}_{i=0} \int_0^{2\pi}\sin(2n\theta)(j-i)[a_ia_{j-i}(\sin(j\theta)-\sin((j-2i)\theta))\\
    &+a_ib_{j-i}(\cos(j\theta)-\cos((j-2i)\theta))+b_ia_{j-i}(\cos(j\theta)+\cos((j-2i)\theta))\\
    &-b_ib_{j-i}(\sin(j\theta)+\sin((j-2i)\theta))]d\theta\\
    =&\frac{\sqrt{\pi}}{4}\sum^{\infty}_{j=0}\sum^{j}_{i=0} (\delta_{j, 2n}(j-i)(a_ia_{j-i}-b_ib_{j-i})\\
    &+ (\delta_{j-2i, 2n}-\delta_{2i-j, 2n})(j-i)(-a_ia_{j-i}-b_ib_{j-i}))\\
    =&\frac{\sqrt{\pi}}{4}[\sum_{i=0}^{2n}(2n-i)(a_{i}a_{2n-i} - b_ib_{2n-i}) + \sum_{\text{even }j\geq 2n }(\frac{j}{2}+n)(-a_{\frac{j}{2}-n}a_{\frac{j}{2}+n} - b_{\frac{j}{2}-n}b_{\frac{j}{2}+n})]\\
    &-\sum_{\text{even }j\geq 2n }(\frac{j}{2}-n)(-a_{n+\frac{j}{2}}a_{\frac{j}{2}-n} - b_{n+\frac{j}{2}}b_{\frac{j}{2}-n})]\\
    =&\frac{\sqrt{\pi}}{4}[\sum_{i=0}^{2n}(2n-i)(a_{i}a_{2n-i} - b_ib_{2n-i}) + \sum_{j=n}^\infty(j+n)(-a_{j-n}a_{j+n} - b_{j-n}b_{j+n})\\
    &-\sum_{j=n}^\infty(j-n)(-a_{j-n}a_{j+n} - b_{j-n}b_{j+n})]\\
    =&\frac{\sqrt{\pi}}{4}[\sum_{i=0}^{2n}(2n-i)(a_{i}a_{2n-i} - b_ib_{2n-i})\\
    &+ \sum_{j=n}^\infty((j+n)-(j-n))(-a_{j-n}a_{j+n} - b_{j-n}b_{j+n})\\
    =&\frac{\sqrt{\pi}}{4}[\sum_{i=0}^{2n}(2n-i)(a_{i}a_{2n-i} - b_ib_{2n-i}) + 2n\sum_{j=n}^\infty(-a_{j-n}a_{j+n} - b_{j-n}b_{j+n})\\
    \addtocounter{equation}{1}\tag{\theequation}=&\frac{\sqrt{\pi}}{4}[\sum_{j=0}^{2n}(2n-j)(a_{j}a_{2n-j} - b_jb_{2n-j}) - 2n\sum_{j=n}^\infty(a_{j-n}a_{j+n} + b_{j-n}b_{j+n})]\\
    \addtocounter{equation}{1}\tag{\theequation}\mathcal{E}_n =&\frac{1}{\sqrt{\pi}}\int_0^{2\pi}\rho^\prime(\theta)\sin(n\theta)\cos(n\theta)d\theta\\
    =&\frac{1}{2\sqrt{\pi}}\int_0^{2\pi}\rho^\prime(\theta)\sin(2n\theta)d\theta\\
    =&\frac{1}{2\sqrt{\pi}}\int_0^{2\pi}\sin(2n\theta)\sum_{j=0}^\infty ja_j\cos(j\theta) - jb_j\sin(j\theta) d\theta\\
    =&\frac{1}{2\sqrt{\pi}}\sum_{j=0}^\infty\int_0^{2\pi}\sin(2n\theta)(ja_j\cos(j\theta) - jb_j\sin(j\theta)) d\theta\\
    \addtocounter{equation}{1}\tag{\theequation}=&-n\sqrt{\pi}b_{2n}\\
    \addtocounter{equation}{1}\tag{\theequation}\mathcal{F}_n =&\frac{1}{\sqrt{\pi}}\int_0^{2\pi}\rho^2(\theta)\sin(n\theta)\cos(n\theta)d\theta\\
    =&\frac{1}{2\sqrt{\pi}}\int_0^{2\pi}\rho^2(\theta)\sin(2n\theta)d\theta\\
    =&\frac{1}{4\sqrt{\pi}}\int_0^{2\pi}\sin(2n\theta)\sum^{\infty}_{j=0}\sum^{j}_{i=0}[-a_ia_{j-i}(\cos(j\theta)-\cos((j-2i)\theta))\\
    &+ a_ib_{j-i}(\sin(j\theta)-\sin((j-2i)\theta))+ b_ia_{j-i}(\sin(j\theta)+\sin((j-2i)\theta))\\
    &+b_ib_{j-i}(\cos(j\theta)+\cos((j-2i)\theta))]d\theta\\
    =&\frac{1}{4\sqrt{\pi}}\sum^{\infty}_{j=0}\sum^{j}_{i=0}\int_0^{2\pi}\sin(2n\theta)[-a_ia_{j-i}(\cos(j\theta)-\cos((j-2i)\theta))\\
    &+ a_ib_{j-i}(\sin(j\theta)-\sin((j-2i)\theta))+ b_ia_{j-i}(\sin(j\theta)+\sin((j-2i)\theta))\\
    &+b_ib_{j-i}(\cos(j\theta)+\cos((j-2i)\theta))]d\theta\\
    =&\frac{\sqrt{\pi}}{4}[\sum^{\infty}_{j=0}\sum^{j}_{i=0}(\delta_{j,2n}(a_ib_{j-i}+b_ia_{j-i}) + (\delta_{j-2i, 2n}-\delta_{2i-j, 2n})(-a_ib_{j-i}+b_ia_{j-i}))]\\
    =&\frac{\sqrt{\pi}}{4}[\sum_{i=0}^{2n}(a_ib_{2n-i} + b_ia_{2n-i}) + \sum_{\text{even }j \geq 2n}(-a_{\frac{j}{2}-n}b_{\frac{j}{2}+n} + b_{\frac{j}{2}-n}a_{\frac{j}{2}+n})\\
    &- \sum_{\text{even }j \geq 2n}(-a_{n+\frac{j}{2}}b_{\frac{j}{2}-n} + b_{n+\frac{j}{2}}a_{\frac{j}{2}-n})]\\
    =&\frac{\sqrt{\pi}}{4}[\sum_{i=0}^{2n}(a_ib_{2n-i} + b_ia_{2n-i}) + \sum_{j=n}^\infty(-a_{j-n}b_{j+n} + b_{j-n}a_{j+n})\\
    &- \sum_{j=n}^\infty(-a_{n+j}b_{j-n} + b_{n+j}a_{j-n})]\\
    =&\frac{\sqrt{\pi}}{4}[\sum_{i=0}^{2n}(a_ib_{2n-i} + b_ia_{2n-i}) + 2\sum_{j=n}^\infty(-a_{j-n}b_{j+n} + b_{j-n}a_{j+n})]\\
    \addtocounter{equation}{1}\tag{\theequation}=&\frac{\sqrt{\pi}}{4}[\sum_{j=0}^{2n}(a_jb_{2n-j} + b_ja_{2n-j}) + 2\sum_{j=n}^\infty(b_{j-n}a_{j+n}-a_{j-n}b_{j+n})]\\
    \addtocounter{equation}{1}\tag{\theequation}\mathcal{G}_n =&\frac{1}{\sqrt{\pi}}\int_0^{2\pi}\rho(\theta)\sin(n\theta)\cos(n\theta)d\theta\\
    =& \frac{1}{2\sqrt{\pi}}\int_0^{2\pi}\rho(\theta)\sin(2n\theta)d\theta\\
    =& \frac{1}{2\sqrt{\pi}}\int_0^{2\pi}\sin(2n\theta) \sum_{j=0}^\infty a_j\sin(j\theta) + b_j \cos(j\theta) d\theta\\
    =&\frac{1}{2\sqrt{\pi}}\sum_{j=0}^\infty\int_0^{2\pi}\sin(2n\theta)(a_j\sin(j\theta) + b_j \cos(j\theta))d\theta\\
    \addtocounter{equation}{1}\tag{\theequation}=&\frac{\sqrt{\pi}}{2}a_{2n}\\
    \addtocounter{equation}{1}\tag{\theequation}\mathcal{H}_n =&\frac{1}{\sqrt{\pi}}\int_0^{2\pi}{\rho^\prime}^2(\theta)\sin(n\theta)\cos(n\theta)d\theta\\
    =&\frac{1}{2\sqrt{\pi}}\int_0^{2\pi}{\rho^\prime}^2(\theta)\sin(2n\theta)d\theta\\
    =&\frac{1}{4\sqrt{\pi}}\int_0^{2\pi}\sin(2n\theta)\sum^{\infty}_{j=0}\sum^{j}_{i=0} i(j-i)[a_ia_{j-i}(\cos(j\theta) +\cos((j-2i)\theta))\\
    &- a_ib_{j-i}(\sin(j\theta)+\sin((j-2i)\theta))-b_ia_{j-i}(\sin(j\theta)-\sin((j-2i)\theta))\\
    &- b_ib_{j-i}(\cos(j\theta)-\cos((j-2i)\theta))] d\theta\\
    =&\frac{1}{4\sqrt{\pi}}\sum^{\infty}_{j=0}\sum^{j}_{i=0} \int_0^{2\pi}\sin(2n\theta)i(j-i)[a_ia_{j-i}(\cos(j\theta) +\cos((j-2i)\theta))\\
    &- a_ib_{j-i}(\sin(j\theta)+\sin((j-2i)\theta))-b_ia_{j-i}(\sin(j\theta)-\sin((j-2i)\theta))\\
    &- b_ib_{j-i}(\cos(j\theta)-\cos((j-2i)\theta))] d\theta\\
    =&\frac{\sqrt{\pi}}{4} [\sum^{\infty}_{j=0}\sum^{j}_{i=0}(\delta_{j, 2n}i(j-i)(-a_ib_{j-i} - b_ia_{j-i})\\
    &+ (\delta_{j-2i, 2n} - \delta_{2i-j, 2n})i(j-i)(-a_ib_{j-i} + b_ia_{j-i}))]\\
    =&\frac{\sqrt{\pi}}{4} [\sum_{i=0}^{2n}i(2n-i)(-a_ib_{2n-i} - b_ia_{2n-i})\\
    &+ \sum_{\text{even }j\geq 2n} (\frac{j}{2}-n)(\frac{j}{2}+n)(-a_{\frac{j}{2}-n}b_{\frac{j}{2}+n} + b_{\frac{j}{2}-n}a_{\frac{j}{2}+n})\\
    &-\sum_{\text{even }j\geq 2n} (n+\frac{j}{2})(\frac{j}{2}-n)(-a_{n+\frac{j}{2}}b_{\frac{j}{2}-n} + b_{n+\frac{j}{2}}a_{\frac{j}{2}-n})]\\
    =&\frac{\sqrt{\pi}}{4} [\sum_{i=0}^{2n}i(2n-i)(-a_ib_{2n-i} - b_ia_{2n-i})\\
    &+ \sum_{j=n}^\infty (j-n)(j+n)(-a_{j-n}b_{j+n} + b_{j-n}a_{j+n})\\
    &-\sum_{j=n}^\infty (n+j)(j-n)(-a_{n+j}b_{j-n} + b_{n+j}a_{j-n})]\\
    =&\frac{\sqrt{\pi}}{4} [\sum_{i=0}^{2n}i(2n-i)(-a_ib_{2n-i} - b_ia_{2n-i})+2\sum_{j=n}^\infty (j^2-n^2)(-a_{j-n}b_{j+n} + b_{j-n}a_{j+n})]\\
    \addtocounter{equation}{1}\tag{\theequation}=&\frac{\sqrt{\pi}}{4} [\sum_{j=0}^{2n}j(2n-j)(-a_jb_{2n-j} - b_ja_{2n-j})+2\sum_{j=n}^\infty (j^2-n^2)(b_{j-n}a_{j+n}-a_{j-n}b_{j+n})]\\
    \addtocounter{equation}{1}\tag{\theequation}\mathcal{I}_n =&\frac{1}{\sqrt{\pi}}\int_0^{2\pi}\rho(\theta)\rho^\prime(\theta)\cos^2(n\theta)d\theta\\
    =&\frac{1}{2\sqrt{\pi}}\int_0^{2\pi}\rho(\theta)\rho^\prime(\theta)(1+\cos(2n\theta))d\theta\\
    =&\frac{1}{4\sqrt{\pi}}\int_0^{2\pi}(1+\cos(2n\theta))\sum^{\infty}_{j=0}\sum^{j}_{i=0} (j-i)[a_ia_{j-i}(\sin(j\theta)-\sin((j-2i)\theta))\\
    &+a_ib_{j-i}(\cos(j\theta)-\cos((j-2i)\theta))+b_ia_{j-i}(\cos(j\theta)+\cos((j-2i)\theta))\\
    &-b_ib_{j-i}(\sin(j\theta)+\sin((j-2i)\theta))]d\theta\\
    =&\frac{1}{4\sqrt{\pi}}\sum^{\infty}_{j=0}\sum^{j}_{i=0} \int_0^{2\pi}(1+\cos(2n\theta))(j-i)[a_ia_{j-i}(\sin(j\theta)-\sin((j-2i)\theta))\\
    &+a_ib_{j-i}(\cos(j\theta)-\cos((j-2i)\theta))+b_ia_{j-i}(\cos(j\theta)+\cos((j-2i)\theta))\\
    &-b_ib_{j-i}(\sin(j\theta)+\sin((j-2i)\theta))]d\theta\\
    =&\frac{\sqrt{\pi}}{4}[2\sum^{\infty}_{j=0}\sum^{j}_{i=0} \delta_{j-2i, 0}(j-i)(-a_ib_{j-i}+b_ia_{j-i})\\
    &+ \sum^{\infty}_{j=0}\sum^{j}_{i=0} (\delta_{j, 2n}(j-i)(a_ib_{j-i}+b_ia_{j-i})\\
    &+ (\delta_{j-2i, 2n}+\delta_{2i-j, 2n})(j-i)(-a_ib_{j-i}+b_ia_{j-i}))]\\
    =&\frac{\sqrt{\pi}}{4}[2\sum_{\text{even }j }\frac{j}{2}(-a_{\frac{j}{2}}b_{\frac{j}{2}} + b_{\frac{j}{2}}a_{\frac{j}{2}}) + \sum_{i=0}^{2n}(2n-i)(a_{i}b_{2n-i} + b_ia_{2n-i})\\
    &+ \sum_{\text{even }j\geq 2n }(\frac{j}{2}+n)(-a_{\frac{j}{2}-n}b_{\frac{j}{2}+n} + b_{\frac{j}{2}-n}a_{\frac{j}{2}+n})]\\
    &+\sum_{\text{even }j\geq 2n }(\frac{j}{2}-n)(-a_{n+\frac{j}{2}}b_{\frac{j}{2}-n} + b_{n+\frac{j}{2}}a_{\frac{j}{2}-n})]\\
    =&\frac{\sqrt{\pi}}{4}[2\sum_{j=0}^\infty (j(0))+\sum_{i=0}^{2n}(2n-i)(a_{i}b_{2n-i} + b_ia_{2n-i})\\
    &+ \sum_{j=n}^\infty(j+n)(-a_{j-n}b_{j+n} + b_{j-n}a_{j+n})+\sum_{j=n}^\infty(j-n)(-a_{j-n}b_{j+n} + b_{j-n}a_{j+n})]\\
    =&\frac{\sqrt{\pi}}{4}[\sum_{i=0}^{2n}(2n-i)(a_{i}b_{2n-i} + b_ia_{2n-i})\\
    &+ \sum_{j=n}^\infty((j+n)-(j-n))(-a_{j-n}b_{j+n} + b_{j-n}a_{j+n})\\
    =&\frac{\sqrt{\pi}}{4}[\sum_{i=0}^{2n}(2n-i)(a_{i}b_{2n-i} + b_ia_{2n-i}) + 2n\sum_{j=n}^\infty(-a_{j-n}b_{j+n} + b_{j-n}a_{j+n})\\
    \addtocounter{equation}{1}\tag{\theequation}=&\frac{\sqrt{\pi}}{4}[\sum_{j=0}^{2n}(2n-j)(a_{j}b_{2n-j} + b_ja_{2n-j}) + 2n\sum_{j=n}^\infty(b_{j-n}a_{j+n}-a_{j-n}b_{j+n})]\\
    \addtocounter{equation}{1}\tag{\theequation}\mathcal{J}_n =&\frac{1}{\sqrt{\pi}}\int_0^{2\pi}\rho^\prime(\theta)\cos^2(n\theta)d\theta\\
    =&\frac{1}{2\sqrt{\pi}}\int_0^{2\pi}\rho^\prime(\theta)(1+\cos(2n\theta))d\theta\\
    =&\frac{1}{2\sqrt{\pi}}\int_0^{2\pi}(1+\cos(2n\theta))\sum_{j=0}^\infty ja_j\cos(j\theta) - jb_j\sin(j\theta)d\theta\\
    =&\frac{1}{2\sqrt{\pi}}\sum_{j=0}^\infty\int_0^{2\pi}(1+\cos(2n\theta))( ja_j\cos(j\theta) - jb_j\sin(j\theta))d\theta\\
    \addtocounter{equation}{1}\tag{\theequation}=&n\sqrt{\pi}a_{2n}\\
    \addtocounter{equation}{1}\tag{\theequation}\mathcal{O}_n =&\frac{1}{\sqrt{\pi}}\int_0^{2\pi}\rho(\theta)\rho^\prime(\theta)\sin^2(n\theta)d\theta\\
    =&\frac{1}{2\sqrt{\pi}}\int_0^{2\pi}\rho(\theta)\rho^\prime(\theta)(1-\cos(2n\theta))d\theta\\
    =&\frac{1}{4\sqrt{\pi}}\int_0^{2\pi}(1-\cos(2n\theta))\sum^{\infty}_{j=0}\sum^{j}_{i=0} (j-i)[a_ia_{j-i}(\sin(j\theta)-\sin((j-2i)\theta))\\
    &+a_ib_{j-i}(\cos(j\theta)-\cos((j-2i)\theta))+b_ia_{j-i}(\cos(j\theta)+\cos((j-2i)\theta))\\
    &-b_ib_{j-i}(\sin(j\theta)+\sin((j-2i)\theta))]d\theta\\
    =&\frac{1}{4\sqrt{\pi}}\sum^{\infty}_{j=0}\sum^{j}_{i=0} \int_0^{2\pi}(1-\cos(2n\theta))(j-i)[a_ia_{j-i}(\sin(j\theta)-\sin((j-2i)\theta))\\
    &+a_ib_{j-i}(\cos(j\theta)-\cos((j-2i)\theta))+b_ia_{j-i}(\cos(j\theta)+\cos((j-2i)\theta))\\
    &-b_ib_{j-i}(\sin(j\theta)+\sin((j-2i)\theta))]d\theta\\
    =&\frac{\sqrt{\pi}}{4}[2\sum^{\infty}_{j=0}\sum^{j}_{i=0} \delta_{j-2i, 0}(j-i)(-a_ib_{j-i}+b_ia_{j-i})\\
    &- \sum^{\infty}_{j=0}\sum^{j}_{i=0} (\delta_{j, 2n}(j-i)(a_ib_{j-i}+b_ia_{j-i})\\
    &+ (\delta_{j-2i, 2n}+\delta_{2i-j, 2n})(j-i)(-a_ib_{j-i}+b_ia_{j-i}))]\\
    =&\frac{\sqrt{\pi}}{4}[2\sum_{\text{even }j }\frac{j}{2}(-a_{\frac{j}{2}}b_{\frac{j}{2}} + b_{\frac{j}{2}}a_{\frac{j}{2}}) - \sum_{i=0}^{2n}(2n-i)(a_{i}b_{2n-i} + b_ia_{2n-i})\\
    &- \sum_{\text{even }j\geq 2n }(\frac{j}{2}+n)(-a_{\frac{j}{2}-n}b_{\frac{j}{2}+n} + b_{\frac{j}{2}-n}a_{\frac{j}{2}+n})]\\
    &-\sum_{\text{even }j\geq 2n }(\frac{j}{2}-n)(-a_{n+\frac{j}{2}}b_{\frac{j}{2}-n} + b_{n+\frac{j}{2}}a_{\frac{j}{2}-n})]\\
    =&\frac{\sqrt{\pi}}{4}[2\sum_{j=0}^\infty (j(0))-\sum_{i=0}^{2n}(2n-i)(a_{i}b_{2n-i} + b_ia_{2n-i})\\
    &- \sum_{j=n}^\infty(j+n)(-a_{j-n}b_{j+n} + b_{j-n}a_{j+n})-\sum_{j=n}^\infty(j-n)(-a_{j-n}b_{j+n} + b_{j-n}a_{j+n})]\\
    =&\frac{\sqrt{\pi}}{4}[-\sum_{i=0}^{2n}(2n-i)(a_{i}b_{2n-i} + b_ia_{2n-i})\\
    &- \sum_{j=n}^\infty((j+n)-(j-n))(-a_{j-n}b_{j+n} + b_{j-n}a_{j+n})\\
    =&\frac{\sqrt{\pi}}{4}[-\sum_{i=0}^{2n}(2n-i)(a_{i}b_{2n-i} + b_ia_{2n-i}) - 2n\sum_{j=n}^\infty(-a_{j-n}b_{j+n} + b_{j-n}a_{j+n})\\
    \addtocounter{equation}{1}\tag{\theequation}=&\frac{\sqrt{\pi}}{4}[-\sum_{j=0}^{2n}(2n-j)(a_{j}b_{2n-j} + b_ja_{2n-j}) - 2n\sum_{j=n}^\infty(b_{j-n}a_{j+n}-a_{j-n}b_{j+n})]\\
    \addtocounter{equation}{1}\tag{\theequation}\mathcal{P}_n =&\frac{1}{\sqrt{\pi}}\int_0^{2\pi}\rho^\prime(\theta)\sin^2(n\theta)d\theta\\
    =&\frac{1}{2\sqrt{\pi}}\int_0^{2\pi}\rho^\prime(\theta)(1-\cos(2n\theta))d\theta\\
    =&\frac{1}{2\sqrt{\pi}}\int_0^{2\pi}(1-\cos(2n\theta))\sum_{j=0}^\infty ja_j\cos(j\theta) - jb_j\sin(j\theta)d\theta\\
    =&\frac{1}{2\sqrt{\pi}}\sum_{j=0}^\infty\int_0^{2\pi}(1-\cos(2n\theta))(ja_j\cos(j\theta) - jb_j\sin(j\theta))d\theta\\
    \addtocounter{equation}{1}\tag{\theequation}=&-n\sqrt{\pi}a_{2n}\\
    \addtocounter{equation}{1}\tag{\theequation}\mathcal{Q}_n =&\frac{1}{\sqrt{\pi}}\int_0^{2\pi}\rho^2(\theta)\sin^2(n\theta)d\theta\\
    =&\frac{1}{2\sqrt{\pi}}\int_0^{2\pi}\rho^2(\theta)(1-\cos(2n\theta))d\theta\\
    =&\frac{1}{4\sqrt{\pi}}\int_0^{2\pi}(1-\cos(2n\theta))\sum^{\infty}_{j=0}\sum^{j}_{i=0}[-a_ia_{j-i}(\cos(j\theta)-\cos((j-2i)\theta))\\
    &+ a_ib_{j-i}(\sin(j\theta)-\sin((j-2i)\theta))+ b_ia_{j-i}(\sin(j\theta)+\sin((j-2i)\theta))\\
    &+b_ib_{j-i}(\cos(j\theta)+\cos((j-2i)\theta))]d\theta\\
    =&\frac{1}{4\sqrt{\pi}}\sum^{\infty}_{j=0}\sum^{j}_{i=0}\int_0^{2\pi}(1-\cos(2n\theta))[-a_ia_{j-i}(\cos(j\theta)-\cos((j-2i)\theta))\\
    &+ a_ib_{j-i}(\sin(j\theta)-\sin((j-2i)\theta))+ b_ia_{j-i}(\sin(j\theta)+\sin((j-2i)\theta))\\
    &+b_ib_{j-i}(\cos(j\theta)+\cos((j-2i)\theta))]d\theta\\
    =&\frac{\sqrt{\pi}}{4}[2\sum^{\infty}_{j=0}\sum^{j}_{i=0} (\delta_{j,0}(b_ib_{j-i}-a_ia_{j-i}) + \delta_{j-2i, 0}(a_ia_{j-i} + b_ib_{j-i}))\\
    &- \sum^{\infty}_{j=0}\sum^{j}_{i=0}(\delta_{j,2n}(-a_ia_{j-i}+b_ib_{j-i}) + (\delta_{j-2i, 2n}+\delta_{2i-j, 2n})(a_ia_{j-i}+b_ib_{j-i}))]\\
    =&\frac{\sqrt{\pi}}{4}[2b_0^2 - 2a_0^2 + 2\sum_{\text{even }j}^\infty(a_{\frac{j}{2}}^2 + b_{\frac{j}{2}}^2)- \sum_{i=0}^{2n}(b_ib_{2n-i}-a_ia_{2n-i})\\
    &-\sum_{\text{even } j \geq 2n}(a_{\frac{j}{2}-n}a_{\frac{j}{2}+n}+b_{\frac{j}{2}-n}b_{\frac{j}{2}+n})-\sum_{\text{even } j \geq 2n}(a_{n+\frac{j}{2}}a_{\frac{j}{2}-n}+b_{n+\frac{j}{2}}b_{\frac{j}{2}-n})]\\
    =&\frac{\sqrt{\pi}}{4}[2b_0^2 - 2a_0^2 + 2\sum_{j=0}^\infty(a_j^2 + b_j^2)\\
    &- \sum_{i=0}^{2n}(b_ib_{2n-i}-a_ia_{2n-i}) -2\sum_{j=n}^\infty(b_{j-n}b_{j+n}+a_{j-n}a_{j+n})]\\
    \addtocounter{equation}{1}\tag{\theequation}=&\frac{\sqrt{\pi}}{4}[4b_0^2 + 2\sum_{j=1}^\infty(a_j^2 + b_j^2) - \sum_{j=0}^{2n}(b_jb_{2n-j}-a_ja_{2n-j}) -2\sum_{j=n}^\infty(b_{j-n}b_{j+n}+a_{j-n}a_{j+n})]\\
    \addtocounter{equation}{1}\tag{\theequation}\mathcal{R}_n =&\frac{1}{\sqrt{\pi}}\int_0^{2\pi}\rho(\theta)\sin^2(n\theta)d\theta\\
    =&\frac{1}{2\sqrt{\pi}}\int_0^{2\pi}\rho(\theta)(1-\cos(2n\theta))d\theta\\
    =&\frac{1}{2\sqrt{\pi}}\int_0^{2\pi}(1-\cos(2n\theta))\sum_{j=0}^\infty a_j\sin(j\theta) + b_j\cos(j\theta)d\theta\\
    =&\frac{1}{2\sqrt{\pi}}\sum_{j=0}^\infty \int_0^{2\pi}(1-\cos(2n\theta))(a_j\sin(j\theta) + b_j\cos(j\theta))d\theta\\
    \addtocounter{equation}{1}\tag{\theequation}=&\frac{\sqrt{\pi}}{2}(2b_0 -b_{2n})\\
    \addtocounter{equation}{1}\tag{\theequation}\mathcal{S}_n =&\frac{1}{\sqrt{\pi}}\int_0^{2\pi}{\rho^\prime}^2(\theta)\sin^2(n\theta)d\theta\\
    =&\frac{1}{2\sqrt{\pi}}\int_0^{2\pi}{\rho^\prime}^2(\theta)(1-\cos(2n\theta))d\theta\\
    =&\frac{1}{4\sqrt{\pi}}\int_0^{2\pi}(1-\cos(2n\theta))\sum^{\infty}_{j=0}\sum^{j}_{i=0} i(j-i)[a_ia_{j-i}(\cos(j\theta) +\cos((j-2i)\theta))\\
    &- a_ib_{j-i}(\sin(j\theta)+\sin((j-2i)\theta))-b_ia_{j-i}(\sin(j\theta)-\sin((j-2i)\theta))\\
    &- b_ib_{j-i}(\cos(j\theta)-\cos((j-2i)\theta))] d\theta\\
    =&\frac{1}{4\sqrt{\pi}}\sum^{\infty}_{j=0}\sum^{j}_{i=0} \int_0^{2\pi}(1-\cos(2n\theta))i(j-i)[a_ia_{j-i}(\cos(j\theta) +\cos((j-2i)\theta))\\
    &- a_ib_{j-i}(\sin(j\theta)+\sin((j-2i)\theta))-b_ia_{j-i}(\sin(j\theta)-\sin((j-2i)\theta))\\
    &- b_ib_{j-i}(\cos(j\theta)-\cos((j-2i)\theta))] d\theta\\
    =&\frac{\sqrt{\pi}}{4} [2\sum^{\infty}_{j=0}\sum^{j}_{i=0} \delta_{j-2i, 0}i(j-i)(a_ia_{j-i} + b_ib_{j-i})\\
    &- \sum^{\infty}_{j=0}\sum^{j}_{i=0}(\delta_{j, 2n}i(j-i)(a_ia_{j-i} - b_ib_{j-i})\\
    &+ (\delta_{j-2i, 2n}+ \delta_{2i-j, 2n})i(j-i)(a_ia_{j-i} + b_ib_{j-i}))]\\
    =&\frac{\sqrt{\pi}}{4} [2\sum_{\text{even }j}(\frac{j}{2})^2(a_{\frac{j}{2}}^2 + b_{\frac{j}{2}}^2) - \sum_{i=0}^{2n}i(2n-i)(a_ia_{2n-i} - b_ib_{2n-i})\\
    &- \sum_{\text{even }j\geq 2n} (\frac{j}{2}-n)(\frac{j}{2}+n)(a_{\frac{j}{2}-n}a_{\frac{j}{2}+n} + b_{\frac{j}{2}-n}b_{\frac{j}{2}+n})\\
    &-\sum_{\text{even }j\geq 2n} (n+\frac{j}{2})(\frac{j}{2}-n)(a_{n+\frac{j}{2}}a_{\frac{j}{2}-n} + b_{n+\frac{j}{2}}b_{\frac{j}{2}-n})]\\
    =&\frac{\sqrt{\pi}}{4} [2\sum_{j=0}^\infty j^2(a_j^2 + b_j^2)- \sum_{i=0}^{2n}i(2n-i)(a_ia_{2n-i} - b_ib_{2n-i})\\
    &-2\sum_{j=n}^\infty (j-n)(j+n)(a_{j-n}a_{j+n} + b_{j-n}b_{j+n})]\\
    \addtocounter{equation}{1}\tag{\theequation}=&\frac{\sqrt{\pi}}{4} [2\sum_{j=1}^\infty j^2(a_j^2 + b_j^2)- \sum_{j=0}^{2n}j(2n-j)(a_ja_{2n-j} - b_jb_{2n-j})\\
    &-2\sum_{j=n}^\infty (j^2-n^2)(a_{j-n}a_{j+n} + b_{j-n}b_{j+n})]
\end{align*}
Since we have seen our integral constants involving both $n$ and $k$ are equivalent to ones calculated above whenever $n=k$, we now calculate these constants ($\mathcal{K}_{n,k}$ to $\mathcal{N}_{k,n}$ and $\mathcal{T}_{n,k}$ to $\mathcal{W}_{k,n}$) first assuming $n<k$:
\begin{align*}
    \addtocounter{equation}{1}\tag{\theequation}\mathcal{K}_{n,k} =&\frac{1}{\sqrt{\pi}}\int_0^{2\pi} \rho^\prime(\theta)\sin(k\theta)\cos(n\theta)d\theta\\
    =&\frac{1}{2\sqrt{\pi}}\int_0^{2\pi} \rho^\prime(\theta)(\sin((k-n)\theta) + \sin((k+n)\theta))d\theta\\
    =&\frac{1}{2\sqrt{\pi}}\int_0^{2\pi} (\sin((k-n)\theta) + \sin((k+n)\theta))\sum_{j=0}^\infty ja_j\cos(j\theta) - jb_j\sin(j\theta)d\theta\\
    =&\frac{1}{2\sqrt{\pi}}\sum_{j=0}^\infty \int_0^{2\pi} (\sin((k-n)\theta) + \sin((k+n)\theta))(ja_j\cos(j\theta) - jb_j\sin(j\theta))d\theta\\
    \addtocounter{equation}{1}\tag{\theequation}=&\frac{\sqrt{\pi}}{2}(-(k-n)b_{k-n} - (k+n)b_{k+n})\\
    \addtocounter{equation}{1}\tag{\theequation}\mathcal{L}_{n,k} =&\frac{1}{\sqrt{\pi}}\int_0^{2\pi} \rho(\theta)\cos(k\theta)\cos(n\theta)d\theta\\
    =&\frac{1}{2\sqrt{\pi}}\int_0^{2\pi} \rho(\theta)(\cos((k-n)\theta) + \cos((k+n)\theta))d\theta\\
    =&\frac{1}{2\sqrt{\pi}}\int_0^{2\pi} (\cos((k-n)\theta) + \cos((k+n)\theta))\sum_{j=0}^\infty a_j\sin(j\theta) + b_j\cos(j\theta)d\theta\\
    =&\frac{1}{2\sqrt{\pi}}\sum_{j=0}^\infty \int_0^{2\pi} (\cos((k-n)\theta) + \cos((k+n)\theta))(a_j\sin(j\theta) + b_j\cos(j\theta))d\theta\\
    \addtocounter{equation}{1}\tag{\theequation}=&\frac{\sqrt{\pi}}{2}(b_{k-n} + b_{k+n})\\
    \addtocounter{equation}{1}\tag{\theequation}\mathcal{M}_{n,k} =&\frac{1}{\sqrt{\pi}}\int_0^{2\pi} \rho^\prime(\theta)\cos(k\theta)\cos(n\theta)d\theta\\
    =&\frac{1}{2\sqrt{\pi}}\int_0^{2\pi} \rho^\prime(\theta)(\cos((k-n)\theta) + \cos((k+n)\theta))d\theta\\
    =&\frac{1}{2\sqrt{\pi}}\int_0^{2\pi} (\cos((k-n)\theta) + \cos((k+n)\theta))\sum_{j=0}^\infty ja_j\cos(j\theta) - jb_j\sin(j\theta)d\theta\\
    =&\frac{1}{2\sqrt{\pi}}\sum_{j=0}^\infty \int_0^{2\pi} (\cos((k-n)\theta) + \cos((k+n)\theta))(ja_j\cos(j\theta) - jb_j\sin(j\theta))d\theta\\
    \addtocounter{equation}{1}\tag{\theequation}=&\frac{\sqrt{\pi}}{2}((k-n)a_{k-n} + (k+n)a_{k+n})\\
    \addtocounter{equation}{1}\tag{\theequation}\mathcal{N}_{n,k} =&\frac{1}{\sqrt{\pi}}\int_0^{2\pi} \rho(\theta)\sin(k\theta)\cos(n\theta)d\theta\\
    =&\frac{1}{2\sqrt{\pi}}\int_0^{2\pi} \rho(\theta)(\sin((k-n)\theta) + \sin((k+n)\theta))d\theta\\
    =&\frac{1}{2\sqrt{\pi}}\int_0^{2\pi} (\sin((k-n)\theta) + \sin((k+n)\theta))\sum_{j=0}^\infty a_j\sin(j\theta) + b_j\cos(j\theta)d\theta\\
    =&\frac{1}{2\sqrt{\pi}}\sum_{j=0}^\infty \int_0^{2\pi} (\sin((k-n)\theta) + \sin((k+n)\theta))(a_j\sin(j\theta) + b_j\cos(j\theta))d\theta\\
    \addtocounter{equation}{1}\tag{\theequation}=&\frac{\sqrt{\pi}}{2}(a_{k-n} + a_{k+n})\\
    \addtocounter{equation}{1}\tag{\theequation}\mathcal{T}_{n,k} =&\frac{1}{\sqrt{\pi}}\int_0^{2\pi} \rho^\prime(\theta)\sin(k\theta)\sin(n\theta)d\theta\\
    =&\frac{1}{2\sqrt{\pi}}\int_0^{2\pi} \rho^\prime(\theta)(\cos((k-n)\theta) - \cos((k+n)\theta))d\theta\\
    =&\frac{1}{2\sqrt{\pi}}\int_0^{2\pi} (\cos((k-n)\theta) - \cos((k+n)\theta))\sum_{j=0}^\infty ja_j\cos(j\theta) - jb_j\sin(j\theta)d\theta\\
    =&\frac{1}{2\sqrt{\pi}}\sum_{j=0}^\infty \int_0^{2\pi} (\cos((k-n)\theta) - \cos((k+n)\theta))(ja_j\cos(j\theta) - jb_j\sin(j\theta))d\theta\\
    \addtocounter{equation}{1}\tag{\theequation}=&\frac{\sqrt{\pi}}{2}((k-n)a_{k-n} - (k+n)a_{k+n})\\
    \addtocounter{equation}{1}\tag{\theequation}\mathcal{U}_{n,k} =&\frac{1}{\sqrt{\pi}}\int_0^{2\pi} \rho(\theta)\cos(k\theta)\sin(n\theta)d\theta\\
    =&\frac{1}{2\sqrt{\pi}}\int_0^{2\pi} \rho(\theta)(- \sin((k-n)\theta) + \sin((k+n)\theta) )d\theta\\
    =&\frac{1}{2\sqrt{\pi}}\int_0^{2\pi} (- \sin((k-n)\theta) + \sin((k+n)\theta))\sum_{j=0}^\infty a_j\sin(j\theta) + b_j\cos(j\theta)d\theta\\
    =&\frac{1}{2\sqrt{\pi}}\sum_{j=0}^\infty \int_0^{2\pi} (- \sin((k-n)\theta) + \sin((k+n)\theta))(a_j\sin(j\theta) + b_j\cos(j\theta))d\theta\\
    \addtocounter{equation}{1}\tag{\theequation}=&\frac{\sqrt{\pi}}{2}(-a_{k-n} + a_{k+n})\\
    \addtocounter{equation}{1}\tag{\theequation}\mathcal{V}_{n,k} =&\frac{1}{\sqrt{\pi}}\int_0^{2\pi} \rho^\prime(\theta)\cos(k\theta)\sin(n\theta)d\theta\\
    =&\frac{1}{2\sqrt{\pi}}\int_0^{2\pi} \rho^\prime(\theta)(- \sin((k-n)\theta) + \sin((k+n)\theta) )d\theta\\
    =&\frac{1}{2\sqrt{\pi}}\int_0^{2\pi} (- \sin((k-n)\theta) + \sin((k+n)\theta))\sum_{j=0}^\infty ja_j\cos(j\theta) - jb_j\sin(j\theta)d\theta\\
    =&\frac{1}{2\sqrt{\pi}}\sum_{j=0}^\infty \int_0^{2\pi} (- \sin((k-n)\theta) + \sin((k+n)\theta))(ja_j\cos(j\theta) - jb_j\sin(j\theta))d\theta\\
    \addtocounter{equation}{1}\tag{\theequation}=&\frac{\sqrt{\pi}}{2}((k-n)b_{k-n} - (k+n)b_{k+n})\\
    \addtocounter{equation}{1}\tag{\theequation}\mathcal{W}_{n,k} =&\frac{1}{\sqrt{\pi}}\int_0^{2\pi} \rho(\theta)\sin(k\theta)\sin(n\theta)d\theta\\
    =&\frac{1}{2\sqrt{\pi}}\int_0^{2\pi} \rho(\theta)(\cos((k-n)\theta) - \cos((k+n)\theta))d\theta\\
    =&\frac{1}{2\sqrt{\pi}}\int_0^{2\pi} (\cos((k-n)\theta) - \cos((k+n)\theta))\sum_{j=0}^\infty a_j\sin(j\theta) + b_j\cos(j\theta)d\theta\\
    =&\frac{1}{2\sqrt{\pi}}\sum_{j=0}^\infty \int_0^{2\pi} (\cos((k-n)\theta) - \cos((k+n)\theta))(a_j\sin(j\theta) + b_j\cos(j\theta))d\theta\\
    \addtocounter{equation}{1}\tag{\theequation}=&\frac{\sqrt{\pi}}{2}(b_{k-n} - b_{k+n})
\end{align*}
We now repeat the above calculations assuming $k < n$:
\begin{align*}
    \addtocounter{equation}{1}\tag{\theequation}\mathcal{K}_{n,k} =&\frac{1}{\sqrt{\pi}}\int_0^{2\pi} \rho^\prime(\theta)\sin(k\theta)\cos(n\theta)d\theta\\
    =&\frac{1}{2\sqrt{\pi}}\int_0^{2\pi} \rho^\prime(\theta)(-\sin((n-k)\theta) + \sin((k+n)\theta))d\theta\\
    =&\frac{1}{2\sqrt{\pi}}\int_0^{2\pi} (-\sin((n-k)\theta) + \sin((k+n)\theta))\sum_{j=0}^\infty ja_j\cos(j\theta) - jb_j\sin(j\theta)d\theta\\
    =&\frac{1}{2\sqrt{\pi}}\sum_{j=0}^\infty \int_0^{2\pi} (-\sin((n-k)\theta) + \sin((k+n)\theta))(ja_j\cos(j\theta) - jb_j\sin(j\theta))d\theta\\
    =&\frac{\sqrt{\pi}}{2}((n-k)b_{n-k} - (k+n)b_{k+n})\\
    \addtocounter{equation}{1}\tag{\theequation}=&\frac{\sqrt{\pi}}{2}(-(k-n)b_{n-k} - (k+n)b_{k+n})\\
    \addtocounter{equation}{1}\tag{\theequation}\mathcal{L}_{n,k} =&\frac{1}{\sqrt{\pi}}\int_0^{2\pi} \rho(\theta)\cos(k\theta)\cos(n\theta)d\theta\\
    =&\frac{1}{2\sqrt{\pi}}\int_0^{2\pi} \rho(\theta)(\cos((n-k)\theta) + \cos((k+n)\theta))d\theta\\
    =&\frac{1}{2\sqrt{\pi}}\int_0^{2\pi} (\cos((n-k)\theta) + \cos((k+n)\theta))\sum_{j=0}^\infty a_j\sin(j\theta) + b_j\cos(j\theta)d\theta\\
    =&\frac{1}{2\sqrt{\pi}}\sum_{j=0}^\infty \int_0^{2\pi} (\cos((n-k)\theta) + \cos((k+n)\theta))(a_j\sin(j\theta) + b_j\cos(j\theta))d\theta\\
    \addtocounter{equation}{1}\tag{\theequation}=&\frac{\sqrt{\pi}}{2}(b_{n-k} + b_{k+n})\\
    \addtocounter{equation}{1}\tag{\theequation}\mathcal{M}_{n,k} =&\frac{1}{\sqrt{\pi}}\int_0^{2\pi} \rho^\prime(\theta)\cos(k\theta)\cos(n\theta)d\theta\\
    =&\frac{1}{2\sqrt{\pi}}\int_0^{2\pi} \rho^\prime(\theta)(\cos((n-k)\theta) + \cos((k+n)\theta))d\theta\\
    =&\frac{1}{2\sqrt{\pi}}\int_0^{2\pi} (\cos((n-k)\theta) + \cos((k+n)\theta))\sum_{j=0}^\infty ja_j\cos(j\theta) - jb_j\sin(j\theta)d\theta\\
    =&\frac{1}{2\sqrt{\pi}}\sum_{j=0}^\infty \int_0^{2\pi} (\cos((n-k)\theta) + \cos((k+n)\theta))(ja_j\cos(j\theta) - jb_j\sin(j\theta))d\theta\\
    =&\frac{\sqrt{\pi}}{2}((n-k)a_{n-k} + (k+n)a_{k+n})\\
    \addtocounter{equation}{1}\tag{\theequation}=&\frac{\sqrt{\pi}}{2}(-(k-n)a_{n-k} + (k+n)a_{k+n})\\
    \addtocounter{equation}{1}\tag{\theequation}\mathcal{N}_{n,k} =&\frac{1}{\sqrt{\pi}}\int_0^{2\pi} \rho(\theta)\sin(k\theta)\cos(n\theta)d\theta\\
    =&\frac{1}{2\sqrt{\pi}}\int_0^{2\pi} \rho(\theta)(-\sin((n-k)\theta) + \sin((k+n)\theta))d\theta\\
    =&\frac{1}{2\sqrt{\pi}}\int_0^{2\pi} (-\sin((n-k)\theta) + \sin((k+n)\theta))\sum_{j=0}^\infty a_j\sin(j\theta) + b_j\cos(j\theta)d\theta\\
    =&\frac{1}{2\sqrt{\pi}}\sum_{j=0}^\infty \int_0^{2\pi} (-\sin((n-k)\theta) + \sin((k+n)\theta))(a_j\sin(j\theta) + b_j\cos(j\theta))d\theta\\
    \addtocounter{equation}{1}\tag{\theequation}=&\frac{\sqrt{\pi}}{2}(-a_{n-k} + a_{k+n})\\
    \addtocounter{equation}{1}\tag{\theequation}\mathcal{T}_{n,k} =&\frac{1}{\sqrt{\pi}}\int_0^{2\pi} \rho^\prime(\theta)\sin(k\theta)\sin(n\theta)d\theta\\
    =&\frac{1}{2\sqrt{\pi}}\int_0^{2\pi} \rho^\prime(\theta)(\cos((n-k)\theta) - \cos((k+n)\theta))d\theta\\
    =&\frac{1}{2\sqrt{\pi}}\int_0^{2\pi} (\cos((n-k)\theta) - \cos((k+n)\theta))\sum_{j=0}^\infty ja_j\cos(j\theta) - jb_j\sin(j\theta)d\theta\\
    =&\frac{1}{2\sqrt{\pi}}\sum_{j=0}^\infty \int_0^{2\pi} (\cos((n-k)\theta) - \cos((k+n)\theta))(ja_j\cos(j\theta) - jb_j\sin(j\theta))d\theta\\
    =&\frac{\sqrt{\pi}}{2}((n-k)a_{n-k} - (k+n)a_{k+n})\\
    \addtocounter{equation}{1}\tag{\theequation}=&\frac{\sqrt{\pi}}{2}(-(k-n)a_{n-k} - (k+n)a_{k+n})\\
    \addtocounter{equation}{1}\tag{\theequation}\mathcal{U}_{n,k} =&\frac{1}{\sqrt{\pi}}\int_0^{2\pi} \rho(\theta)\cos(k\theta)\sin(n\theta)d\theta\\
    =&\frac{1}{2\sqrt{\pi}}\int_0^{2\pi} \rho(\theta)(\sin((n-k)\theta) + \sin((k+n)\theta) )d\theta\\
    =&\frac{1}{2\sqrt{\pi}}\int_0^{2\pi} (\sin((n-k)\theta) + \sin((k+n)\theta))\sum_{j=0}^\infty a_j\sin(j\theta) + b_j\cos(j\theta)d\theta\\
    =&\frac{1}{2\sqrt{\pi}}\sum_{j=0}^\infty \int_0^{2\pi} (\sin((n-k)\theta) + \sin((k+n)\theta))(a_j\sin(j\theta) + b_j\cos(j\theta))d\theta\\
    \addtocounter{equation}{1}\tag{\theequation}=&\frac{\sqrt{\pi}}{2}(a_{n-k} + a_{k+n})\\
    \addtocounter{equation}{1}\tag{\theequation}\mathcal{V}_{n,k} =&\frac{1}{\sqrt{\pi}}\int_0^{2\pi} \rho^\prime(\theta)\cos(k\theta)\sin(n\theta)d\theta\\
    =&\frac{1}{2\sqrt{\pi}}\int_0^{2\pi} \rho^\prime(\theta)(\sin((n-k)\theta) + \sin((k+n)\theta) )d\theta\\
    =&\frac{1}{2\sqrt{\pi}}\int_0^{2\pi} (\sin((n-k)\theta) + \sin((k+n)\theta))\sum_{j=0}^\infty ja_j\cos(j\theta) - jb_j\sin(j\theta)d\theta\\
    =&\frac{1}{2\sqrt{\pi}}\sum_{j=0}^\infty \int_0^{2\pi} (\sin((n-k)\theta) + \sin((k+n)\theta))(ja_j\cos(j\theta) - jb_j\sin(j\theta))d\theta\\
    =&\frac{\sqrt{\pi}}{2}(-(n-k)b_{n-k} - (k+n)b_{k+n})\\
    \addtocounter{equation}{1}\tag{\theequation}=&\frac{\sqrt{\pi}}{2}((k-n)b_{n-k} - (k+n)b_{k+n})\\
    \addtocounter{equation}{1}\tag{\theequation}\mathcal{W}_{n,k} =&\frac{1}{\sqrt{\pi}}\int_0^{2\pi} \rho(\theta)\sin(k\theta)\sin(n\theta)d\theta\\
    =&\frac{1}{2\sqrt{\pi}}\int_0^{2\pi} \rho(\theta)(\cos((n-k)\theta) - \cos((k+n)\theta))d\theta\\
    =&\frac{1}{2\sqrt{\pi}}\int_0^{2\pi} (\cos((n-k)\theta) - \cos((k+n)\theta))\sum_{j=0}^\infty a_j\sin(j\theta) + b_j\cos(j\theta)d\theta\\
    =&\frac{1}{2\sqrt{\pi}}\sum_{j=0}^\infty \int_0^{2\pi} (\cos((n-k)\theta) - \cos((k+n)\theta))(a_j\sin(j\theta) + b_j\cos(j\theta))d\theta\\
    \addtocounter{equation}{1}\tag{\theequation}=&\frac{\sqrt{\pi}}{2}(b_{n-k} - b_{k+n})
\end{align*}
We notice that the convention $-a_j = a_{-j}$ and $b_j = b_{-j}, \forall j \in \Z\setminus\{0\}$ aligns the two formulas of each constant involving $k$ for all $k \neq n$. We will apply these conventions going forward to present a single formula for each such constant.

\subsection{Integral Constant Values}
\label{integralsvalue}
For reference, here we present the values of all integral constants, computed in \ref{computingintegrals}.
\begin{align*}
    \mathcal{A}_n =&\frac{\sqrt{\pi}}{4}[4b_0^2 + 2   \sum_{j=1}^\infty(a_j^2 + b_j^2) + \sum_{j=0}^{2n}(b_jb_{2n-j}-a_ja_{2n-j}) +2\sum_{j=n}^\infty(b_{j-n}b_{j+n}+a_{j-n}a_{j+n})]\\
    \mathcal{B}_n =&\frac{\sqrt{\pi}}{2}(2b_0 + b_{2n})\\
    \mathcal{C}_n =&\frac{\sqrt{\pi}}{4} [2\sum_{j=1}^\infty j^2(a_j^2 + b_j^2)+ \sum_{j=0}^{2n}j(2n-j)(a_ja_{2n-j} - b_jb_{2n-j})\\
    &+2\sum_{j=n}^\infty (j^2-n^2)(a_{j-n}a_{j+n} + b_{j-n}b_{j+n})]\nonumber\\
    \mathcal{D}_n =&\frac{\sqrt{\pi}}{4}[\sum_{j=0}^{2n}(2n-j)(a_{j}a_{2n-j} - b_jb_{2n-j}) - 2n\sum_{j=n}^\infty(a_{j-n}a_{j+n} + b_{j-n}b_{j+n})]\\
    \mathcal{E}_n =&-n\sqrt{\pi}b_{2n}\\
    \mathcal{F}_n =&\frac{\sqrt{\pi}}{4}[\sum_{j=0}^{2n}(a_jb_{2n-j} + b_ja_{2n-j}) + 2\sum_{j=n}^\infty(b_{j-n}a_{j+n}-a_{j-n}b_{j+n})]\\
    \mathcal{G}_n =&\frac{\sqrt{\pi}}{2}a_{2n}\\
    \mathcal{H}_n =&\frac{\sqrt{\pi}}{4} [\sum_{j=0}^{2n}j(2n-j)(-a_jb_{2n-j} - b_ja_{2n-j})+2\sum_{j=n}^\infty (j^2-n^2)(b_{j-n}a_{j+n}-a_{j-n}b_{j+n})]\\
    \mathcal{I}_n =&\frac{\sqrt{\pi}}{4}[\sum_{j=0}^{2n}(2n-j)(a_{j}b_{2n-j} + b_ja_{2n-j}) + 2n\sum_{j=n}^\infty(b_{j-n}a_{j+n}-a_{j-n}b_{j+n})]\\
    \mathcal{J}_n =&n\sqrt{\pi}a_{2n}\\
    \mathcal{O}_n =&\frac{\sqrt{\pi}}{4}[-\sum_{j=0}^{2n}(2n-j)(a_{j}b_{2n-j} + b_ja_{2n-j}) - 2n\sum_{j=n}^\infty(b_{j-n}a_{j+n}-a_{j-n}b_{j+n})]\\
    \mathcal{P}_n =&-n\sqrt{\pi}a_{2n}\\
    \mathcal{Q}_n =&\frac{\sqrt{\pi}}{4}[4b_0^2 + 2\sum_{j=1}^\infty(a_j^2 + b_j^2) - \sum_{j=0}^{2n}(b_jb_{2n-j}-a_ja_{2n-j}) -2\sum_{j=n}^\infty(b_{j-n}b_{j+n}+a_{j-n}a_{j+n})]\\
    \mathcal{R}_n =&\frac{\sqrt{\pi}}{2}(2b_0-b_{2n})\\
    \mathcal{S}_n =&\frac{\sqrt{\pi}}{4} [2\sum_{j=1}^\infty j^2(a_j^2 + b_j^2)- \sum_{j=0}^{2n}j(2n-j)(a_ja_{2n-j} - b_jb_{2n-j})\\
    &-2\sum_{j=n}^\infty (j^2-n^2)(a_{j-n}a_{j+n} + b_{j-n}b_{j+n})]\nonumber
\end{align*}
The following constant formulas use the convention discussed in \ref{computingintegrals} where we take $-a_j = a_{-j}$ and $b_j = b_{-j}, \forall j \in \Z\setminus\{0\}$.
\begin{align*}
    \mathcal{K}_{n,k} =&\frac{\sqrt{\pi}}{2}(-(k-n)b_{k-n} - (k+n)b_{k+n})\\
    \mathcal{L}_{n,k} =&\frac{\sqrt{\pi}}{2}(b_{k-n} + b_{k+n})\\
    \mathcal{M}_{n,k} =&\frac{\sqrt{\pi}}{2}((k-n)a_{k-n} + (k+n)a_{k+n})\\
    \mathcal{N}_{n,k} =&\frac{\sqrt{\pi}}{2}(a_{k-n} + a_{k+n})\\
    \mathcal{T}_{n,k} =&\frac{\sqrt{\pi}}{2}((k-n)a_{k-n} - (k+n)a_{k+n})\\
    \mathcal{U}_{n,k} =&\frac{\sqrt{\pi}}{2}(-a_{k-n} + a_{k+n})\\
    \mathcal{V}_{n,k} =&\frac{\sqrt{\pi}}{2}((k-n)b_{k-n} - (k+n)b_{k+n})\\
    \mathcal{W}_{n,k}=&\frac{\sqrt{\pi}}{2}(b_{k-n} - b_{k+n})
\end{align*}

\end{document}